\documentclass[a4paper,10pt]{amsart}

\usepackage{amsfonts,amsthm,amssymb,mathabx}
\usepackage[pdftex]{graphicx}
\usepackage{graphics}
\usepackage[matrix,arrow]{xy}
\usepackage[active]{srcltx}
\usepackage{accents}
\usepackage{mathrsfs}

\usepackage[colorlinks=true]{hyperref}
\hypersetup{colorlinks=true,citecolor=blue}

\newcommand{\N}{\mathbb{N}}
\newcommand{\Z}{\mathbb{Z}}
\newcommand{\Q}{\mathbb{Q}}
\newcommand{\R}{\mathbb{R}}
\newcommand{\C}{\mathbb{C}}

\newcommand{\F}{\mathcal{F}}

\newcommand{\CZ}{{\rm CZ}}
\renewcommand{\P}{\mathscr{P}}
\newcommand{\inte}[1]{{\rm int({#1})}}

\newcommand{\hloc}{\mathscr{H}}

\newcommand{\M}{\mathcal{M}}
\newcommand{\CM}{{\rm CM}}
\newcommand{\HM}{{\rm HM}}

\newcommand{\cl}[1]{\overline{{#1}}}
\newcommand{\intr}{{\rm int}}

\newcommand{\dist}{\mathrm{dist\,}}

\newcommand{\diam}{\mathrm{diam\,}}
\newcommand{\ind}{\mathrm{ind}}

\newcommand{\diag}{\mathrm{diag\,}}

\newcommand{\supp}{\mathrm{supp\,}}

\newcommand{\crit}{\mathrm{Crit}\,}

\newcommand{\A}{\mathbb{A}}
\newcommand{\fix}{{\rm Fix}}

\newcommand{\iso}{{\rm Iso}}
\renewcommand{\div}{{\rm Div}}
\renewcommand{\O}{\mathcal{O}}

\theoremstyle{plain}
\newtheorem{theorem}{Theorem}[section]
\newtheorem{proposition}[theorem]{Proposition}
\newtheorem{lemma}[theorem]{Lemma}
\newtheorem{corollary}[theorem]{Corollary}

\theoremstyle{definition}
\newtheorem{definition}[theorem]{Definition}

\theoremstyle{remark}
\newtheorem{remark}[theorem]{Remark}

\begin{document}

\title[Transversality in local Morse homology with symmetries]{Transversality for local Morse homology with symmetries and applications}

\author[Doris Hein]{Doris Hein}
\author[Umberto Hryniewicz]{Umberto Hryniewicz}
\author[Leonardo Macarini]{Leonardo Macarini}

\address{Doris Hein\\
Mathematisches Institut, 
Albert-Ludwigs-Universit\"at Freiburg, Eckerstrasse 1, 79104 Freiburg, Germany}
\email{doris.hein@math.uni-freiburg.de}

\address{Umberto Hryniewicz\\
Universidade Federal do Rio de Janeiro -- Departamento de Matema\'tica Aplicada, Av. Athos da Silveira Ramos 149, Rio de Janeiro RJ, Brazil 21941-909.}
\email{umbertolh@gmail.com}

\address{Leonardo Macarini\\
Universidade Federal do Rio de Janeiro -- Departamento de Matema\'tica, Av. Athos da Silveira Ramos 149, Rio de Janeiro RJ, Brazil 21941-909.}
\email{leomacarini@gmail.com}




\begin{abstract} 
We prove the transversality result necessary for defining local Morse chain complexes with finite cyclic group symmetry. Our arguments use special regularized distance functions constructed using classical covering lemmas, and an inductive perturbation process indexed by the strata of the isotropy set. A global existence theorem for symmetric Morse-Smale pairs is also proved. Regarding applications, we focus on Hamiltonian dynamics and rigorously establish a local contact homology package based on discrete action functionals. We prove a persistence theorem, analogous to the classical shifting lemma for geodesics, asserting that the iteration map is an isomorphism for good and admissible iterations. We also consider a  Chas-Sullivan product on non-invariant local Morse homology, which plays the role of pair-of-pants product, and study its relationship to symplectically degenerate maxima. Finally, we explore how our invariants can be used to study bifurcation of critical points (and periodic points) under additional symmetries.
\end{abstract}

\setcounter{tocdepth}{2}

\maketitle

\tableofcontents
 

\section{Introduction and main results}

\subsection{Introduction}

The aim of this paper is to provide a rigorous Morse homological construction of certain local invariants of periodic points of Hamiltonian diffeomorphisms via elementary methods. These invariants are of a subharmonic nature, and in general differ from local Floer homology. Their existence was predicted in~\cite{HM} as local contact homology. However, transversality problems are usually (but not always!) present when trying to define versions of contact homology via standard Floer theoretic methods. In full generality one needs an alternative approach. One possibility is the Polyfold Theory due to Hofer, Wysocki and Zehnder~\cite{hofer,HWZ1,HWZ2,HWZ3}, which will provide the analytic background for such constructions.

Our goal is to implement a finite-dimensional Morse homological approach to local contact homology at the chain level. For invariance properties we rely on the interplay with a parallel construction using singular homology. In fact, the invariants could even be defined using singular homology instead of Morse homology, but then local chain complexes would not be directly related to dynamics, the connection to the SFT-like construction from~\cite{HM} would be lost and, most importantly, we would not be able to prove the Persistence Theorem (Theorem~\ref{thm_persistence_invariant}). As we know from the geodesic case, the Persistence Theorem is at the heart of applications. It answers a question raised in~\cite{GGo} about precise iteration properties of the local invariants. A byproduct of our methods is the existence of a symmetric Morse-Smale pair in any closed manifold with a finite-cyclic action.

In one form or another, transversality issues in SFT are usually related to symmetries. These difficulties incarnate in many different forms. For instance, one may try to achieve transversality for Floer homology in period $k>1$ using $1$-periodic data (Hamiltonian and almost complex structure). If this was possible then the Floer chain complex, say in the aspherical case with complex coefficients, would inherit an action of $\Z_k := \Z/k\Z$ by chain maps. The isotypical components in homology would be invariants of the symplectic manifold, provided transversality for $\Z_k$-equivariant continuation maps could also be achieved. In this global situation, the only non-trivial isotypical component is the one corresponding to the trivial action, the resulting invariant is just standard Floer homology. But this construction would still provide interesting dynamical applications. The most immediate ones come from the fact that Floer homology groups would have to be generated by good periodic points, in every period. Hence new multiplicity results, which are invisible for standard Floer homology, could be proved. However in local situations, such as ours, other isotypical components may not vanish and do provide new invariants. This general framework will be exploited in future work, here we are concerned only with the isotypical component corresponding to the trivial action since this is the appropriate substitute of local contact homology.

There may be other approaches to our invariants. For instance by applying the Borel construction, similarly to~\cite{BO} and~\cite{GG2}, in order to build a $\Z_k$-equivariant version of local Floer homology. Here $\Z_k$ acts by time-reparametrization of loops $\gamma(t) \mapsto \gamma(t+1)$. It is hard to recover time-symmetry at the chain level as one perturbs the data to achieve transversality, not to mention further limits in finite-dimensional approximations of $B\Z_k$, making it very hard to work with the symmetry. See Appendix~\ref{app_Borel} for a more detailed explanation based on the toy model of finite-dimensional Morse theory. The approach of~\cite{HM} is geometrically transparent but suffers from usual transversality problems. Symmetry in the chain complex is also lost in the interesting approach outlined by Hutchings~\cite{Hu_blog}.

In order to work on the chain level we study transversality for finite-dimensional local Morse homology at an isolated critical point in the case that both function and critical point are invariant under an ambient $\Z_k$-action. We show (Theorem~\ref{main1}) that we can $C^2$-perturb to achieve transversality keeping $\Z_k$-symmetry, allowing the group to act at the chain level. We apply this result to approach local contact homology via discrete action functionals as in~\cite{chaperon1,chaperon2,mazz_SDM}.

We also look at a finite-dimensional approach to local Floer homology, taking the path of Mazzucchelli~\cite{mazz_SDM} in order to study {\it symplectically degenerate maxima} (SDM) originally defined in~\cite{Gi,Hi}. Using our non-symmetric Persistence Theorem~\ref{thm_persistence_non_invariant}, analogous to the main result of~\cite{GG}, we simplify the definition of SDM's from~\cite{mazz_SDM} in the discretized set-up. Then, using Chas-Sullivan type products, which play the role of local pair-of-pants products, we characterize SDM's in terms of idempotency. This is in alignment with results of Goresky and Hingston~\cite{GH}, where the notion of SDM is not explicit but it is implicit as the study of {\it maximal versus minimal} index growth under iterations; see~\cite[section~12]{GH}. Of course, this topic goes back to Hingston~\cite{hingston,hingston2}. Our idempotency statement is made, but not proved, in~\cite[section~5.2]{GG}, and is in perfect analogy to \cite[section~12]{GH}.

Finally, we remark that our local invariants serve as tools to study bifurcation theory of isolated critical points in the presence of finite cyclic symmetry groups. At the end of this introduction in Section~\ref{sssec_bifurcation_pics} we provide explicit examples in dimension two for certain generic bifurcations studied by Deng and Xia~\cite{xia}: we study cases when two, four and eight critical points bifurcate from the singularity. \\

\noindent {\it Organization of the paper.} In the remaining of the introduction we state and discuss our main results, constructions, and examples. Basic properties of our local homology theory are established in Section~\ref{sec_properties}, where technical details are postponed to Appendix~\ref{app_invariance}. Section~\ref{sec_shifting_lemma} is devoted to persistence theorems (shifting lemmas), both symmetric and non-symmetric versions. Local Chas-Sullivan products are defined in Section~\ref{sec_chas_sullivan}. Sections~\ref{sec_prelim_transv},~\ref{sec:MS_local} and~\ref{sec:MS_global} deal with transversality. Our local transversality result relies on the construction of special distance functions presented in Appendix~\ref{app_reg_dist_functions}, where we modify some of the analysis from Stein~\cite{stein}. In Appendix~C we explain relations between our construction and the Borel construction. \\

\noindent {\it Acknowledgements.} We are grateful to Viktor Ginzburg for useful comments regarding this paper. UH and LM would like to thank J. Fish, M. Hutchings, J. Nelson and K. Wehrheim for organizing the AIM Workshop ``Transversality in contact homology'' in December 2014, where Morse homology in the presence of symmetries was a topic of intense discussion. UH is extremely grateful to A.~Abbondandolo for numerous insightful conversations concerning the analysis involved in this work. UH also thanks the Floer Center of Geometry (Bochum) for its warm hospitality, and acknowledges the generous support of the Alexander von Humboldt Foundation during the preparation of this manuscript.

\subsection{Main results}

In Section~\ref{sec_transv_statements} we state our transversality results, which are used in Section~\ref{sssec_inv_MH} to define local invariant Morse homology. Applications to subharmonic invariants of isolated periodic points are described in Section~\ref{ssec_disc_action}, where the comparison to local contact homology is explained. Our Persistence Theorem is stated in Section~\ref{sssec_persistence_intro}, along with its non-symmetric version. Chas-Sullivan products and symplectically degenerate maxima are discussed in Section~\ref{sssec_SDM_intro}. In Section~\ref{sssec_bifurcation_pics} we compute our invariants in three bidimensional bifurcation scenarios under $\Z_2$-symmetry (according to~\cite{xia}, generically these are the only three cases in two dimensions). In Sections~\ref{sssec_transv_impossible} and~\ref{sssec_global_equiv_hom} we discuss basic global examples.

\subsubsection{Transversality for invariant local Morse homology}\label{sec_transv_statements}

Consider a smooth Riemannian manifold without boundary $(M,\theta)$ of dimension $d$, and let $f:M\to\R$ be smooth. We denote the $\theta$-gradient of $f$ by $\nabla^\theta f$, and the flow of $-\nabla^\theta f$ by $\phi_{f,\theta}^t$.

\begin{definition}\label{def_stable_unstable_mfds}
For $p \in \crit(f)$ we denote
\begin{equation*}
\begin{aligned}
&W^s(p;f,\theta) = \left\{ x\in M \left| \text{$\phi_{f,\theta}^t(x)$ is defined $\forall t\in[0,+\infty)$ and } \lim_{t\to+\infty}\phi_{f,\theta}^t(x)= p \right. \right\} \\
&W^u(p;f,\theta) = \left\{ x\in M \left| \text{$\phi_{f,\theta}^t(x)$ is defined $\forall t\in(-\infty,0]$ and } \lim_{t\to-\infty}\phi_{f,\theta}^t(x)=p \right. \right\}
\end{aligned}
\end{equation*}
\end{definition}

\begin{remark}
If $p$ is non-degenerate and has Morse index $\mu$ then $W^u(p;f,\theta)$ and $W^s(p;f,\theta)$ are smooth embedded balls of dimension $\mu$ and $d - \mu$, respectively. They are called the {\it unstable and stable manifolds of $p$}.
\end{remark}

\begin{definition}\label{def_MS}
The pair $(f,\theta)$ is said to be {\it Morse-Smale on $M$} if 
\begin{itemize}
\item[(i)] $f$ is a Morse function.
\item[(ii)] $W^u(p;f,\theta)$ intersects $W^s(q;f,\theta)$ transversely for all $p,q \in \crit(f)$.
\item[(iii)] There exist compact sets $K_0\subset K_1\subset M$ such that $\crit(f)\subset \inte{K_0}$, $f$ oscillates strictly more than $\max_{p,q\in\crit(f)}|f(p)-f(q)|$ along any piece of $\theta$-gradient trajectory with one endpoint in $K_0$ and the other in $M\setminus\inte{K_1}$.
\end{itemize}
\end{definition}

\begin{remark}
Note that the (pre-)compactness condition (iii) is vacuous when $M$ is compact. This condition is just one of many ways of getting compactness suited to our local problem. Obviously, condition (iii) will hold for small perturbations of a pair when $M$ is taken as an isolating open neighborhood of an isolated critical point (see the definition below). It follows from (iii) that 
\[
W^u(p;f,\theta)\cap W^s(q;f,\theta)\subset \inte{K_1} \qquad \forall p,q\in\crit(f)
\]
and that being Morse-Smale is stable under small perturbations supported on a fixed compact subset.
\end{remark}

The statement below is the main technical tool in defining proper substitutes of the chain complexes of local contact homology. If $p$ is an isolated critical point of a smooth function then a neighborhood $U$ of $p$ will be called {\it isolating for $(f,p)$} if $\crit(f) \cap \cl{U} = \{p\}$.

\begin{theorem}\label{main1}
Let $M$ be a smooth manifold without boundary equipped with a smooth $\Z_k$-action, $f:M\to\R$ be an invariant smooth function and $p\in M$ be a fixed point of the action which is an isolated critical point of $f$. Let $\theta$ be an invariant metric on $M$, and let $U$ be an open, relatively compact, isolating neighborhood of $(f,p)$. Then in any $C^2$-neighborhood of $(f,\theta)$ there exists a $\Z_k$-invariant pair $(f',\theta')$ which is Morse-Smale on $U$.
\end{theorem}

The proof is found in Section~\ref{sec:MS_local}, after preliminaries in Section~\ref{sec_prelim_transv}. Our methods yield the following byproduct of a global nature. For the proof see section~\ref{sec:MS_global}.

\begin{theorem}\label{main2}
On any smooth closed manifold $M$ equipped with a $\Z_k$-action there exists a Morse-Smale $\Z_k$-invariant pair $(f,\theta)$.
\end{theorem}

Theorem~\ref{main2} {\bf does not} claim that an invariant pair can be slightly $C^2$-perturbed to an invariant Morse-Smale pair. This is not possible in general, for a simple example see Section~\ref{sssec_transv_impossible}.

\subsubsection{Definition of invariant local Morse homology}\label{sssec_inv_MH}

Let $f,\theta,p,U$ be as in the statement of Theorem~\ref{main1}. The theorem guarantees the existence of a $\Z_k$-invariant pair $(f',\theta')$ arbitrarily $C^2$-close to $(f,\theta)$ which is Morse-Smale on $U$. Choose an orientation of the unstable manifold of each critical point of $f'$ in $U$. Let $\CM(f',\theta',U)$ be the vector space over $\Q$ freely generated by the critical points of $f'$ on $U$, graded by the Morse index. Now, in a standard fashion, the differential is defined by counting signed anti-gradient trajectories in $U$ connecting critical points of index difference~$1$. The differential depends on the choice of orientations, but the resulting homology groups do not (up to isomorphism).

Let $a:M\to M$ be the diffeomorphism generating the $\Z_k$-action. Then $j\in\Z_k$ acts on $\CM(f',\theta',U)$ as follows: $j\cdot x = a^j(x)$ if $a^j$ preserves the chosen orientations of the unstable manifolds of $x$ and $a^j(x)$, or $j\cdot x = -a^j(x)$ otherwise. Using the symmetry of $(f',\theta')$ one can show that this is an action by chain maps. The homology of the subcomplex of invariant chains will be denoted by
\begin{equation}\label{notation_inv_local_Morse_hom}
\HM(f,p)^{\Z_k}
\end{equation}
and called the {\it $\Z_k$-invariant local Morse homology} of $f$ at $p$. As the notation suggests, this invariant is independent of $\theta$ and $U$, and of the perturbation $(f',\theta')$. See Section~\ref{ssec_invariant} for details. This construction will be applied to discrete action functionals in order to provide adequate substitutes of local contact homology groups. It can also be used as a new tool to study bifurcations when a finite-cyclic group symmetry is present, see Section~\ref{sssec_bifurcation_pics}.

Similarly as above, one defines an invariant subcomplex of the Morse chain complex of a pair $(f,\theta)$ given by Theorem~\ref{main2}. It turns out that the homology of this subcomplex, called {\it $\Z_k$-invariant Morse homology}, is isomorphic to $\Z_k$-equivariant homology of $M$. This is only possible since we use coefficients in a field, say $\Q$. Theorem~\ref{main2} puts the alternative Morse-theoretical description of $\Z_k$-equivariant homology from~\cite[appendix]{GHHM} onto rigorous grounds. A simple example is described in Section~\ref{sssec_global_equiv_hom} below.

\subsubsection{Definition of local invariants of isolated periodic points}\label{ssec_disc_action}

Let $H = H_t$ be a $1$-periodic Hamiltonian defined on a symplectic manifold $(M,\omega)$. It generates an isotopy $\varphi^t_H$ by $\frac{d}{dt}\varphi_H^t = X_{H_t}\circ\varphi_H^t$ with initial condition $\varphi_H^0=id$, where $X_{H_t}$ is the Hamiltonian vector field defined as $i_{X_{H_t}}\omega=dH_t$. Suppose that $\varphi^t_H$ is defined for all $t\in[0,1]$ on a neighborhood of $p\in\fix_0(\varphi^1_H)$. We are interested in the germ of $\varphi^1_H$ near $p$. Up to changing the isotopy and choosing Darboux coordinates centered at~$p$, there is no loss of generality to assume that $(M,\omega)=(\R^{2n},\omega_0 = \sum_{j=1}^n dx_j\wedge dy_j)$ with $p=0$ and $dH_t(0)=0$ for all $t$.

With $N\in\N$ fixed, the sequence of germs of diffeomorphisms 
\begin{equation}\label{germs_small_steps}
\psi_i := \varphi_H^{i/N} \circ (\varphi_H^{(i-1)/N})^{-1} \ \ \ \ \ \ (i\in\Z)
\end{equation}
is $N$-periodic. If $N$ is large enough then there are generating functions $S_i$ near the origin, namely
\begin{equation}\label{gen_function_formulas}
\psi_i(x,y) = (X,Y) \ \ \ \Leftrightarrow \ \ \ \left\{ \begin{aligned} X-x &= \nabla_2S_i(x,Y) \\ y-Y &= \nabla_1S_i(x,Y) \end{aligned} \right. \ .
\end{equation}
The quality of being ``large enough'' will be given a precise meaning in Section~\ref{sssec_gen_functions}, i.e., we ask $N$ to be adapted to $H$ as in~\eqref{N_adapted}. We normalize $S_i$ by $S_i(0)=0$. The family $\{S_i\}$ is also $N$-periodic. Fixing $k\in\N$, the discrete action function is
\begin{equation}\label{def_discrete_action}
\A_{H,k,N}(z_1,\dots,z_{kN}) = \sum_{i=1}^{kN} x_i(y_{i+1}-y_i) + S_i(x_i,y_{i+1}) \qquad (i\mod kN)
\end{equation}
defined on a small neighborhood of the origin in $\R^{2nkN}$. Here $z_i = (x_i,y_i) \in \R^{2n}$. All this goes back to Chaperon~\cite{chaperon1,chaperon2}. A $\Z_k$-symmetry of $\A_{H,k,N}$ is generated by the (right-)shift map
\begin{equation}\label{shift_map_discrete_action}
\begin{aligned}
\tau:\R^{2nkN} &\to \R^{2nkN} \\ 
(z_1,\dots,z_{kN}) &\mapsto (z_{(k-1)N+1},\dots,z_{kN},z_1,\dots,z_{(k-1)N})
\end{aligned}
\end{equation}
on discrete loops $(z_1,\dots,z_{kN})$.

Now assume that $0\in\R^{2n}$ is an isolated fixed point of $\varphi^k_H$. Hence $0\in\R^{2nkN}$ is an isolated critical point of $\A_{H,k,N}$. By an application of Theorem~\ref{main1}, we achieve local transversality with $\Z_k$-symmetry and follow the construction explained in Section~\ref{sssec_inv_MH} to define invariant local Morse homology groups $\HM_{*}(\A_{H,k,N},0)^{\Z_k}$. It will be shown in Lemma~\ref{lemma_inflation} that there exist so-called inflation maps $$ \mathscr{I}_N^{\Z_k} : \HM_{*}(\A_{H,k,N},0)^{\Z_k} \to \HM_{*+2nk}(\A_{H,k,N+2},0)^{\Z_k} $$ which are isomorphisms, and with respect to which we can take the direct limit
\begin{equation}
\begin{aligned}
\hloc_*^{\rm inv}(H,k,0) &= \lim_{N\to\infty} \HM_{*+nk2N}(\A_{H,k,2N},0)^{\Z_k}.
\end{aligned}
\end{equation}
These are adequate substitutes of the local contact homology groups; see Remark~\ref{rem_inflation_reason}.

In order to link this to local contact homology, note that $\alpha = (H+c)dt+\lambda_0$ is a contact form on the solid torus $\R/\Z\times B$ when $c\gg1$, having $\gamma:t\in \R/c\Z \mapsto (t/c,0)$ as a closed $\alpha$-Reeb orbit. Here $B$ is a small open ball centered at $0\in\R^{2n}$ and $\lambda_0 = \frac 12 \sum_{i=1}^n x_idy_i - y_idx_i$. Its $k$-th iterate $\gamma^k$ is an isolated closed $\alpha$-Reeb orbit and one defines a chain complex generated over $\Q$ by the {\it good} closed $\alpha'$-Reeb orbits that $\gamma^k$ splits to as we slightly perturb $\alpha$ to a generic $\alpha'$. Grading is given by Conley-Zehnder indices up a to shift depending only on $n$. The differential is given by the count of rigid finite-energy pseudo-holomorphic cylinders in $\R\times\R/\Z\times B$. One needs to prove a compactness statement, which is done in~\cite[section 3]{HM}, and assume the existence of generic almost complex structures, which might not exist. Assuming even more regularity, the resulting homology is shown to be independent of the small perturbation~$\alpha'$. It is denoted by ${\rm HC}(\alpha,\gamma^k)$ and called the {\it local contact homology} of $\gamma^k$. The notion of a {\it good orbit} is reviewed in Section~\ref{sssec_good_adm} below.

Local contact homology can be studied for stable Hamiltonian structures, such as the one on $\R/\Z\times B$ induced by $H$ and $\omega_0$. A $1$-periodic $\omega_0$-compatible almost complex structure $J_t$ on $B$ induces an almost complex structure on $\R\times\R/\Z\times B$ of the kind used in contact homology. After perturbing $H_t$ to a generic $H'_t$ keeping $1$-periodicity, the finite-energy solutions to be counted are nothing but graphs of finite-energy solutions of Floer's equation defined on $\R\times\R/k\Z$ with values on $B$. Transversality can not be achieved keeping $1$-periodicity of $J_t$, but let us assume that this is possible for the moment. Then the action on the time variable $t\mapsto t+1$ generates a $\Z_k$-action by chain maps on the local chain complex associated to $(H',J)$ at period $k$. This is proved in~\cite[Section 6]{HM}, where the following statement is also shown: {\it Local contact homology is the homology of the subcomplex of $\Z_k$-invariant chains.} The transversality problem in local contact homology is finally revealed as the problem of achieving transversality in $k$-periodic local Floer homology using $1$-periodic geometric data. The analogy to our local invariants is transparent if we substitute the action functional $\mathcal{A}_{H,k} = \int \alpha$ by its discretized version $\A_{H,k,N}$~\eqref{def_discrete_action}, and the time shift $t \mapsto t+1$ by its discrete version~\eqref{shift_map_discrete_action}. Such a clear analogy is only possible by our constructions at the chain level, and will be used in~\cite{HHM_prep} to show that $\hloc^{\rm inv}(H,k,0)$ is isomorphic to ${\rm HC}(\alpha,\gamma^k)$ whenever the latter can be defined. We will rely on~\cite[Proposition 2.5]{mazz_SDM} to make sure that the isomorphism is grading-preserving.

Finally, we note that inflation maps without $\Z_k$-symmetry also exist as maps
\[
\mathscr{I}_N : \HM_{*}(\A_{H,k,N},0) \to \HM_{*+nk}(\A_{H,k,N+1},0)
\]
and allow for the definition non-invariant local homology groups
\[
\hloc_*(H,k,0) = \lim_{N\to\infty} \HM_{*+nkN}(\A_{H,k,N},0)
\]
which are discrete versions of local Floer homology groups.

\begin{remark}\label{rem_inflation_reason}
The fact that we have the term $N+2$ in the symmetric inflation map $\mathscr{I}^{\Z_k}_N$ and $N+1$ in the non-symmetric one (and consequently the shift in the degree is $2nk$ for $\mathscr{I}^{\Z_k}_N$ and $nk$ for $\mathscr{I}_N$) is due to an orientation issue; see Remark~\ref{rmk:symmetric x non_symmetric inflation map} for details.
\end{remark}

\begin{remark}[Gradings]\label{rmk_gradings}
The {\it Conley-Zehnder index} of a path $M:[0,T]\to Sp(2n)$ satisfying $M(0)=I$, $\det M(T)-I\neq0$ is defined in a standard way, see for instance~\cite{SZ,RSindex}. Its extension to general paths is not standard. Here we use its maximal lower semicontinuous extension. Namely, if $\det M(T)-I=0$ then its Conley-Zehnder index is defined as the $\liminf$ in $C^0$ as $\tilde M \to M$ of the Conley-Zehnder indices of paths $\tilde M:[0,T]\to Sp(2n)$ satisfying $\det (\tilde M(T)-I)\neq0$. With these conventions, the Conley-Zehnder index of $t\in[0,k]\mapsto d\varphi^t_H(0)$ will be denoted by $\CZ(H,k)$. This extension to degenerate paths is smaller than or equal to the one defined in~\cite{RSindex}, and in general disagrees with it; for instance the constant path equal to the identity matrix will have index $-n$ according to our conventions, but will have index zero according to the conventions of~\cite{RSindex}. Its mean Conley-Zehnder index is denoted by $\Delta_{\CZ}(H,k)\in\R$, and satisfies $\Delta_{\CZ}(H,km)=m\Delta_{\CZ}(H,k)$ for all $k,m\in\N$. Denote $$ \nu(H,k) = \dim \ker (d\varphi_H^k(0)-I). $$ Note that $\nu(H,k)$ is also the nullity of $0\in\R^{2nkN}$ as a critical point of $\A_{H,k,N}$ (see Section~\ref{sec_shifting_lemma}). By~\cite{LL1,LL2} we have
\[
[\CZ(H,k),\CZ(H,k)+\nu(H,k)] \subset [\Delta_{\CZ}(H,k)-n,\Delta_{\CZ}(H,k)+n].
\]
Lemma~\ref{lemma_gradings} will show that the graded groups $\hloc_*(H,k,0)$, $\hloc_*^{\rm inv}(H,k,0)$ are supported in degrees $[\CZ(H,k),\CZ(H,k)+\nu(H,k)]$. Moreover, if $\Delta_{\CZ}(H,k)+n$ or $\Delta_{\CZ}(H,k)-n$ are attained then $0$ is a {\it totally degenerate} fixed point of $\varphi^k_H$, i.e., $1$ is the only eigenvalue of $d\varphi^k_H(0)$, see Lemma~\ref{lemma_grading_tot_deg}.
\end{remark}

We collect here some of the basic properties of $\hloc^{\rm inv}(H,k,0)$ that will be proved in Section~\ref{sec_properties}.
\begin{itemize}
\item {\bf (Homotopy)} If $\{H^\tau_t\}_{\tau\in[0,1]}$ is a family of $1$-periodic Hamiltonians defined near $0\in\R^{2n}$, $dH^\tau_t(0)\equiv0$, such that $0$ is an isolated fixed point of the family $\{\varphi^k_{H^\tau}\}_{\tau\in[0,1]}$ then there exists an isomorphism
\[
\hloc^{\rm inv}_*(H^0,k,0) \simeq \hloc^{\rm inv}_*(H^1,k,0)
\]
\item {\bf (Support)} $\hloc^{\rm inv}_j(H,k,0)$ vanishes if $j\not\in [\CZ(H,k),\CZ(H,k)+\nu(H,k)]$, in particular also if $j\not\in[\Delta_{\CZ}(H,k)-n,\Delta_{\CZ}(H,k)+n]$, see Remark~\ref{rmk_gradings}.
\item {\bf (Change of isotopy)} Let $G_t$ be a $1$-periodic Hamiltonian defined near $0\in\R^{2n}$ such that $\varphi^1_G$ is the identity germ at $0$. If we set $K_t = (G\#H)_t = G_t + H_t \circ (\varphi_G^t)^{-1}$ then there exists an isomorphism 
\[
\hloc^{\rm inv}_*(H,k,0) \simeq \hloc^{\rm inv}_{*+2m}(K,k,0).
\]
where $m$ is the Maslov index of the loop $t\in[0,k] \mapsto d\varphi^t_G(0) \in Sp(2n)$.
\end{itemize}
The {\bf homotopy} and {\bf change of isotopy} properties are proved in Section~\ref{sssec_isotopy}. The {\bf support} property is proved in Section~\ref{sssec_grading}.

\subsubsection{Persistence Theorems}\label{sssec_persistence_intro}

If $\alpha$ is a contact form on some manifold and $\gamma$ is a closed $\alpha$-Reeb orbit such that all iterates $\gamma^j$ are isolated among closed $\alpha$-Reeb orbits, then crucial to dynamical applications are the iteration properties of the sequence ${\rm HC}(\alpha,\gamma^j)$. The extension of Gromoll-Meyer's result from~\cite{GM2} for Reeb flows done in~\cite{HM} relies on that fact that $\dim {\rm HC}(\alpha,\gamma^j) \leq \dim {\rm HF}(\phi^j,p)$ for all $j$, where $p\in\gamma$ and $\phi$ is the local first return map to a transverse local section at $p$. By the result of~\cite{GG} the latter is bounded in $j$, hence so is the former.

More precise information about the sequence ${\rm HC}(\alpha,\gamma^j)$ was not available, with one exception: in~\cite{GGo} it is shown that the Euler characteristics
\[
\chi(\alpha,\gamma^j) = \sum_{i\in\Z} (-1)^i \dim {\rm HC}_i(\alpha,\gamma^j)
\]
are recovered from Lefschetz theory by the non-trivial formula
\begin{equation}\label{formula_GGo}
\chi(\alpha,\gamma^j) = \frac{1}{j} \sum_{d\in\div(j)} \#\{1\leq l< j/d : \gcd(l,j/d)=1 \} \ i_{\phi^d}(p).
\end{equation}
where $i_{\phi^d}(p)$ is the index of the fixed point $p$ of the map $\phi^d$. Hence, the sequence $\chi(\alpha,\gamma^j)$ is periodic in $j$.

\begin{definition}\label{good_admissible_iterations}
Given $M\in Sp(2n)$, $k\in\N$ is an {\it admissible} iteration for $M$ if $1$ has the same algebraic multiplicity for $M$ and $M^k$. It is a {\it good} iteration if the numbers of eigenvalues of $M$ and $M^k$ in $(-1,0)$ have the same parity.
\end{definition}

Let $H_t$ be a $1$-periodic Hamiltonian defined near $0\in\R^{2n}$, $dH_t(0)=0$ for all~$t$, such that $0$ is an isolated fixed point of $\varphi_H^m$, $m\in\N$. The number $k\in\N$ is an admissible iteration for $\varphi_H^m$ if it is admissible for $d\varphi_H^m(0)$. It is a good iteration for $\varphi_H^m$ if it is a good iteration for $d\varphi_H^m(0)$.

\begin{theorem}[Persistence Theorem -- Invariant case]\label{thm_persistence_invariant}
Let $k$ be an admissible and good iteration for $\varphi^m_H$. Then the iteration map
\begin{equation}
\mathcal{I} : \hloc^{\rm inv}_*(H,m,0) \to \hloc^{\rm inv}_{*+s_{k,m}}(H,km,0)
\end{equation}
is well-defined and is an isomorphism, where $s_{k,m}=\CZ(H,km)-\CZ(H,m)$. In particular, if $k_i \to \infty$ is a sequence of admissible and good iterations for $\varphi^m_H$ then $|s_{k_i,m}-k_i\Delta_{\CZ}(H,m)|$ is bounded.
\end{theorem}

We will prove this theorem in Section~\ref{sec_shifting_lemma}. It completely determines the iteration properties of the sequence $\hloc^{\rm inv}(H,j,0)$ when $0$ is isolated for all $\varphi^j_H$; see~\cite[Remark 3.15]{GGo}. It follows that the sequence $\iota_j := \dim \hloc^{\rm inv}(H,j,0)$ is periodic in $j$, but in fact more is true. One finds a finite set $\mathscr{J}\subset\N$ satisfying $1\in\mathscr{J}$ and if $j,j'\in\mathscr{J}$ then ${\rm lcm}(j,j')\in\mathscr{J}$. Furthermore, the sequence $\{\iota_j\}_{j\in\N}$ is subordinated to $\mathscr{J}$ in the following sense: for all $j\in\N$, we have $\iota_j=\iota_{q(j)}$ where $q(j)$ is the maximal divisor of $j$ in $\mathscr{J}$, i.e., $q(j)=\max\{i\in\mathscr{J} : i\in\div(j)\}$.

We also prove a persistence theorem in the absence of symmetries. Viewing $\hloc(H,k,0)$ as a discrete version of local Floer homology, the statement below is the analogue of the main result of~\cite{GG}. It is also proved in section~\ref{sec_shifting_lemma}. (Notice however that \cite{GG} does not provide a precise description of the shift $s_k$ in the grading.)

\begin{theorem}[Persistence Theorem -- Non-invariant case]\label{thm_persistence_non_invariant}
Let $k$ be an admissible iteration for $\varphi_H^1$. Then the iteration map
\begin{equation}
\mathcal{I} : \hloc_*(H,1,0) \to \hloc_{*+s_k}(H,k,0)
\end{equation}
is well-defined and is an isomorphism, where $s_k=\CZ(H,k)-\CZ(H,1)$. In particular, if $k_i \to \infty$ is a sequence of admissible iterations for $\varphi^1_H$ then $|s_{k_i}-k_i\Delta_{\CZ}(H,1)|$ is bounded.
\end{theorem}

\subsubsection{Products, symplectically degenerate maxima and idempotency}\label{sssec_SDM_intro}

The results discussed here disregard the group symmetry. {\it Symplectically degenerate maxima} (SDM) were first used by Hingston~\cite{Hi} in order to confirm the Conley conjecture on standard symplectic tori, although she called such special critical points {\it topologically degenerate}. It was then systematically studied and used by Ginzburg~\cite{Gi} and his collaborators to confirm the Conley conjecture in more general symplectic manifolds. In fact, the term SDM was introduced in~\cite{Gi}. In~\cite{Hi}, Hingston studied the action functional via Fourier series, and in~\cite{Gi}, Ginzburg used Floer homology. Mazzucchelli adapted this notion in~\cite{mazz_SDM} to the set-up of generating functions and discrete action functionals. Here we take the latter viewpoint.

We study a fixed point of a Hamiltonian diffeomorphism on a symplectic manifold. As is well-known, there is no loss of generality to assume that this point is the origin in $(\R^{2n},\omega_0)$, and that the $1$-periodic germ of Hamiltonian $H_t$ defined near the origin satisfies $dH_t(0) = 0$ for all $t$. Inspired by~\cite{Gi,GG} we define

\begin{definition}\label{defn_SDM}
The fixed point $0\in\R^{2n}$ is a {\it symplectically degenerate maximum of $H$} if it is an isolated fixed point of $\varphi^1_H$, $\Delta_{\CZ}(H,1)=0$ and $\hloc_n(H,1,0) \neq 0$.
\end{definition}

\begin{remark}
It is interesting to contrast this with Mazzucchelli's definition~\cite[page 729]{mazz_SDM}. There one asks for the existence of some $N$ large such that the germs $\psi_i$~\eqref{germs_small_steps} admit generating functions $S_i$ for which $0\in\R^{2n}$ is an isolated local maximum for all $i$, and ${\rm C}_{n+nkN}(\A_{H,k,N},0)\neq0$ for infinitely many $k$. Here ${\rm C}_*(\A_{H,k,N},0)$ stands for the local critical groups
\[
{\rm C}_*(\A_{H,k,N},0) = H_*(\{\A_{H,k,N}<0\}\cup\{0\},\{\A_{H,k,N}<0\})
\]
where $H_*$ is singular homology. Standard arguments in Morse theory imply that there is a canonical isomorphism ${\rm C}_*(\A_{H,k,N},0) = \HM_*(\A_{H,k,N},0)$. Hence, \cite[page 729]{mazz_SDM} asks that $\hloc_n(H,k,0)\neq 0$ for infinitely many iterates $k$. Let us examine the consequences. By Remark~\ref{rmk_gradings}, the homology $\hloc_*(H,k,0)$ is supported in degrees
\[
[\Delta_{\CZ}(H,k)-n,\Delta_{\CZ}(H,k)+n] = [k\Delta_{\CZ}(H,1)-n,k\Delta_{\CZ}(H,1)+n].
\]
If $\Delta_{\CZ}(H,1)\neq0$ then $n$ does not belong to this interval when $k$ is large enough. Hence $\Delta_{\CZ}(H,1)=0$. Since an end of this interval is achieved, $0$ must be a totally degenerate fixed point of $\varphi^1_H$ (Remark~\ref{rmk_gradings}). Thus every $k$ is admissible and Theorem~\ref{thm_persistence_non_invariant} implies that $\hloc_n(H,1,0)\neq0$. We have shown that an SDM in the sense of~\cite[page 729]{mazz_SDM} is an SDM in the sense of Definition~\ref{defn_SDM}. The converse can be proved only up to linear symplectic change of coordinates and deformation of the Hamiltonian keeping the time-$1$ map (germ) fixed. In fact, if $0$ is an SDM for $H$ as in Definition~\ref{defn_SDM} then, as explained in Remark~\ref{rmk_gradings}, $0$ must be a totally degenerate fixed point of $\varphi^1_H$. It follows that there exists $M\in Sp(2n)$ such that $M\varphi^1_H M^{-1}=\varphi^1_{M_*H}$ becomes arbitrarily $C^1$-close to $id$. We can now choose $K=K_t$ satisfying $\varphi^1_K=\varphi^1_{M_*H}$ which is arbitrarily and uniformly (in $t$) $C^2$-small. Note that the Maslov index of $t\in\R/\Z \mapsto d\varphi^t_{M_*H}(0)(d\varphi^t_K(0))^{-1}$ is an integer close to $\Delta_{\CZ}(H,1)=0$ because $d\varphi^t_K(0)$ is uniformly close to $I$. Hence this Maslov index vanishes and Lemma~\ref{lemma_change_of_isotopy} yields an isomorphism $\hloc_*(H,1,0)=\hloc_*(K,1,0)$. Taking $K$ sufficiently $C^2$-small then $N=1$ is adapted to $K$ as in~\eqref{N_adapted}. It follows by definition that $0\neq \hloc_n(K,1,0)=\HM_{2n}(F,0)$ where $F$ is a generating function for $\varphi^1_K$. Hence $0$ is an isolated local maximum of $F$. Theorem~\ref{thm_persistence_non_invariant} now implies that $0$ is an SDM for $K$ in the sense of~\cite[page 729]{mazz_SDM}. 
\end{remark}

In~\cite{HHM_prep} we will show that Definition~\ref{defn_SDM} is equivalent to the definition from~\cite{Gi}. Evidence to this fact is given by the following lemma (see~\cite[Proposition 5.1]{GG}).

\begin{lemma}
If $0$ is an isolated fixed point of $\varphi^1_H$ then the following are equivalent.
\begin{itemize}
\item[a)] $0$ is an SDM for $H$.
\item[b)] $\hloc_n(H,k_i,0)\neq0$ for a sequence $k_i\to\infty$ of admissible iterations.
\item[c)] $0$ is totally degenerate, $\hloc_n(H,1,0)\neq0$, and $\hloc_n(H,k,0)\neq0$ for some $k>n$. 
\end{itemize}
\end{lemma}

\begin{proof}
We will prove a) $\Rightarrow$ b) $\Rightarrow$ c) $\Rightarrow$ a). Implication a) $\Rightarrow$ b) follows from Theorem~\ref{thm_persistence_non_invariant}. 

Assume b). Then $\Delta_{\CZ}(H,1)=0$ since, otherwise, $\hloc_n(H,k_i,0)$ would vanish for $i$ large; see Remark~\ref{rmk_gradings}. Again by Remark~\ref{rmk_gradings}, the point $0$ is a totally degenerate fixed point of $\varphi^{k_i}_H$. Since $k_i$ is admissible, $0$ must be a totally degenerate fixed point of $\varphi^1_H$. In particular, every $k$ is admissible and Theorem~\ref{thm_persistence_non_invariant} implies that $\hloc_n(H,1,0)\neq \hloc_n(H,k,0)$ for every $k$. We have proved b) $\Rightarrow$ c).

Assume c). Then $\Delta_{\CZ}(H,1)\in 2\Z$ by~\cite{SZ} and total degeneracy of $0$. By Remark~\ref{rmk_gradings} we know that
\[
n\in [\Delta_{\CZ}(H,1)-n,\Delta_{\CZ}(H,1)+n] \cap [k\Delta_{\CZ}(H,1)-n,k\Delta_{\CZ}(H,1)+n]
\]
for some $k>n$. In particular, $0\leq \Delta_{\CZ}(H,1)<2$. It follows that $\Delta_{\CZ}(H,1)=0$ and we have shown that c) $\Rightarrow$ a).
\end{proof}

The local invariants $\hloc(H,k,0)$ can be described in terms of singular homology, as explained in Section~\ref{ssec_invariant}. This point of view is not Morse homological, and hence useless if one wants to make a comparison to local Floer homology, but it is helpful to define operations
\begin{equation*}
\bullet^{(m)} : \hloc_{i_1}(H,k_1,0) \otimes \dots \otimes \hloc_{i_m}(H,k_m,0) \to \hloc_{i_1+\dots+i_m-(m-1)n}(H,k_1+\dots+k_m,0)
\end{equation*}
provided $0$ is an isolated fixed point of all $\varphi^{k_i}_H$ and of $\varphi^{k_1+\dots+k_m}_H$.

The map $\bullet^{(2)}$ yields a product
\begin{equation*}
\bullet : \hloc_i(H,k,0) \otimes \hloc_j(H,m,0) \to \hloc_{i+j-n}(H,k+m,0)
\end{equation*}
given by $a\bullet b = \bullet^{(2)}(a,b)$. It is associative and anti-commutative in the sense that $b\bullet a = (-1)^{|a||b|}a\bullet b$. See Section~\ref{sec_chas_sullivan} for details. It plays the role of the pair-of-pants product in local Floer homology, whose definition is not found in the literature but can be easily constructed by the knowledgeable reader. 

We have the following statement: {\it If $0$ is not an SDM then there exists some $ r_0>0$ depending on $H$ such that $\bullet^{(r)}(a_1,\dots,a_r)=0$ for all integers $r\geq r_0$ which are admissible for $\varphi^1_H$, and $a_i\in\hloc_*(H,1,0)$.} This is proved in the context of local Floer homology as~\cite[Proposition~5.3]{GG}. The same proof goes through since it is based on degree considerations. We complement it with a proof of

\begin{proposition}\label{prop_prod_SDM}
If $0$ is an SDM of $H$ then $\hloc(H,1,0)$ is supported in degree $n$, $\hloc_n(H,1,0) \simeq \Q$ and $\bullet^{(r)}(e,\dots,e)\neq0$ for all $r\in\N$ and $e\neq0$ in $\hloc(H,1,0)$.
\end{proposition}

This statement is found in~\cite[Section 5]{GG} with no proof in the context of local Floer homology. The proof of Proposition~\ref{prop_prod_SDM} can be found in Section~\ref{ssec_special_case_prod}.

\subsubsection{Bifurcations of isolated critical points with symmetry}\label{sssec_bifurcation_pics}

The invariant~\eqref{notation_inv_local_Morse_hom} can be seen as an invariant of bifurcations of isolated critical points which are symmetric with respect to a finite cyclic group action. In general, it is different from standard local Morse homology. It retains information of the birth-death process that happens at the moment of bifurcation, provided that the group symmetry is respected. We give three examples of symmetric bifurcation scenarios in two variables with symmetry group~$\Z_2$, in increasing degree of complexity: firstly two critical points bifurcate, then four and then, finally, eight critical points bifurcate.

The first scenario is shown in Figure~1, where we study the plane with the $\Z_2$-action generated by reflection along an horizontal axis containing the point $x$. A family of $\Z_2$-invariant functions $f_t$ having $x$ as a critical point is analyzed, for $t<0$ we have a saddle at $x$, for $t=0$ bifurcation happens and $x$ is a degenerate isolated critical point of $f_0$, for $t>0$ we have saddles at points $y,z$ which are symmetric to each other, and a maximum at $x$. The metric is the Euclidean one for all $t$. Grey arrows indicate the chosen orientations of the unstable manifolds of the saddles. For $t>0$ we orient the unstable manifold of $x$ by the canonical orientation of the plane. The local Morse chain complex for $t<0$ has a single generator $x$ in degree $1$, and $1\in\Z_2$ acts by $1\cdot x = -x$ because reflection reverts orientation of the grey arrow. Hence $\HM(f_0,x)$ has a generator in degree $1$ but $\HM(f_0,x)^{\Z_2}$ vanishes. By Proposition~\ref{prop_invariance_with_symmetries}, the same conclusion must be achieved when $t>0$. In fact, the complex has a generator $x$ in degree $2$, and two more $y,z$ in degree $1$ . The Morse differential is $\partial x = y-z$. The $\Z_2$-action is determined by $1\cdot x = -x$, since reflection reverts orientations on the plane, and $1\cdot y=-y$, $1\cdot z=-z$ since reflection reverts orientations of grey arrows. Thus, in the basis $\{x,y,z\}$ the operator $1\cdot$ is represented by minus the identity matrix and, as such, certainly commutes with $\partial$, i.e., $\Z_2$ acts by chain maps. Moreover, $1$ is not an eigenvalue of $1\cdot$, confirming that $\HM(f_0,x)^{\Z_2}$ vanishes.

\begin{figure}\label{fig1}
\begin{center}
\includegraphics[width=110mm]{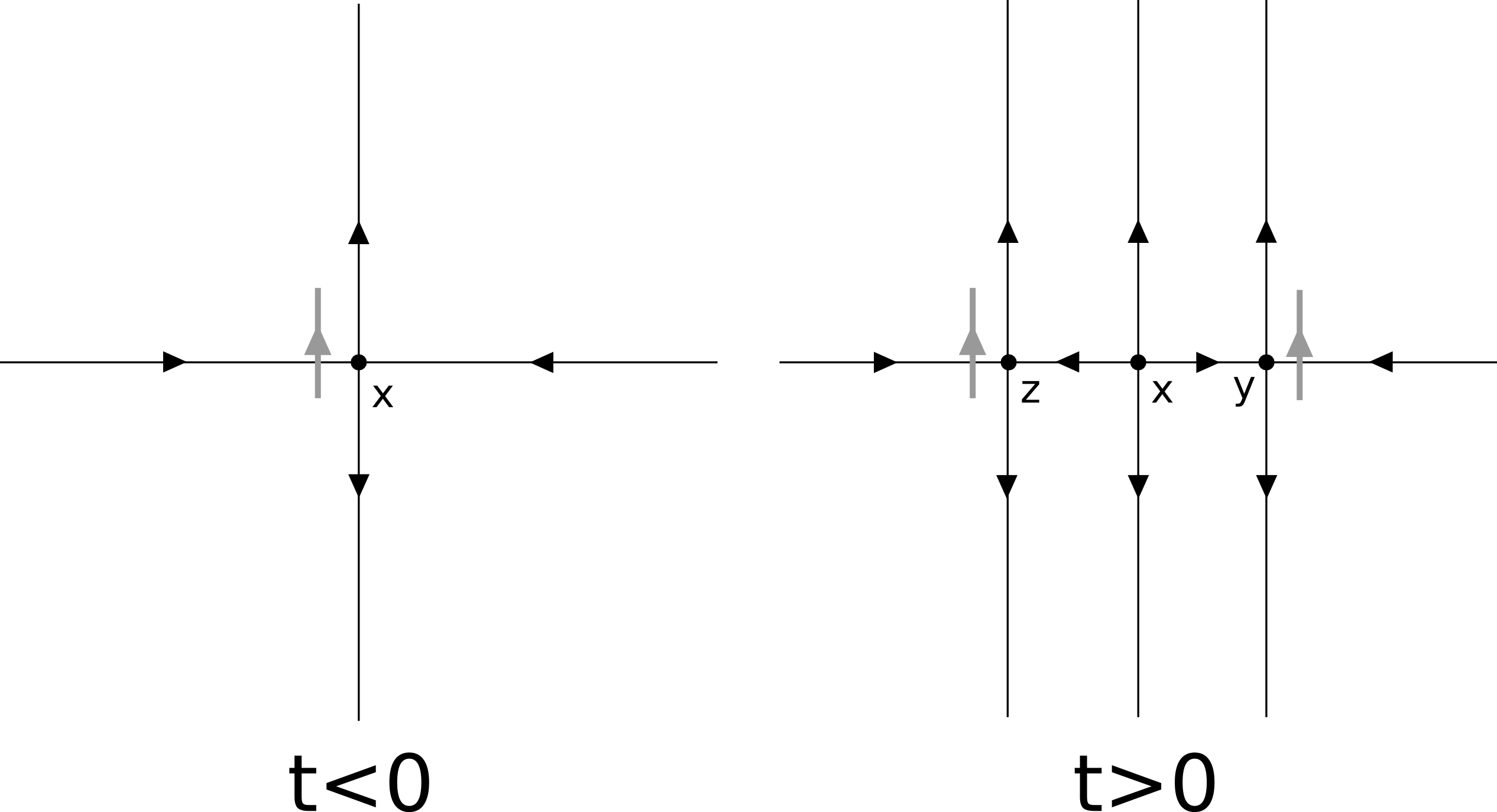}
\caption{{\small Symmetry is given by reflection with respect to the horizontal axis. $\HM(f_0,x)$ has a single generator in degree $1$, $\HM(f_0,x)^{\Z_2}$ vanishes.}}
\end{center}
\end{figure}

The next scenario is shown in Figure~2. Here $\Z_2$ acts by reflection with respect to the vertical axis. For $t<0$ we have an isolated minimum, for $t=0$ bifurcation happens, for $t>0$ there is a maximum at $x$, minima at $y,z$ and saddles at $u,v$. For all $t\neq0$ unstable manifolds of the minima are oriented by $+1$. For $t>0$ the unstable manifold of $x$ is oriented by the canonical orientation of the plane, while the unstable manifolds of the saddles are oriented by the grey arrows. Looking at the trivial local Morse chain complex for $t<0$ we conclude that $\HM(f_0,x) = \HM(f_0,x)^{\Z_2}$ has a single generator in degree $0$. Hence, we must obtain the same conclusion for $t>0$. In fact, $\partial x = u+v$, $\partial u =y-z$ and $\partial v=z-y$. The $\Z_2$-action reads $1\cdot x=-x$, since reflection is orientation reversing on the plane, $1\cdot u=-v$ and $1\cdot v=-u$ since reflection does not preserve orientations of grey arrows, $1\cdot y=y$ and $1\cdot z=z$ since $y,z$ are fixed. It follows that $\partial$ commutes with $1\cdot$, as expected. Moreover, there are no invariant chains in degree $2$, invariant chains are generated by $u-v$ in degree $1$ and by $y,z$ in degree $0$. However, $\partial (u-v)=2(y-z)$ shows that invariant homology vanishes in degree 1 and is generated by the homology class of $y+z$ in degree $0$. The result is again that $\HM(f_0,x)^{\Z_2}$ has one generator in degree~$0$.

\begin{figure}\label{fig2}
\begin{center}
\includegraphics[width=110mm]{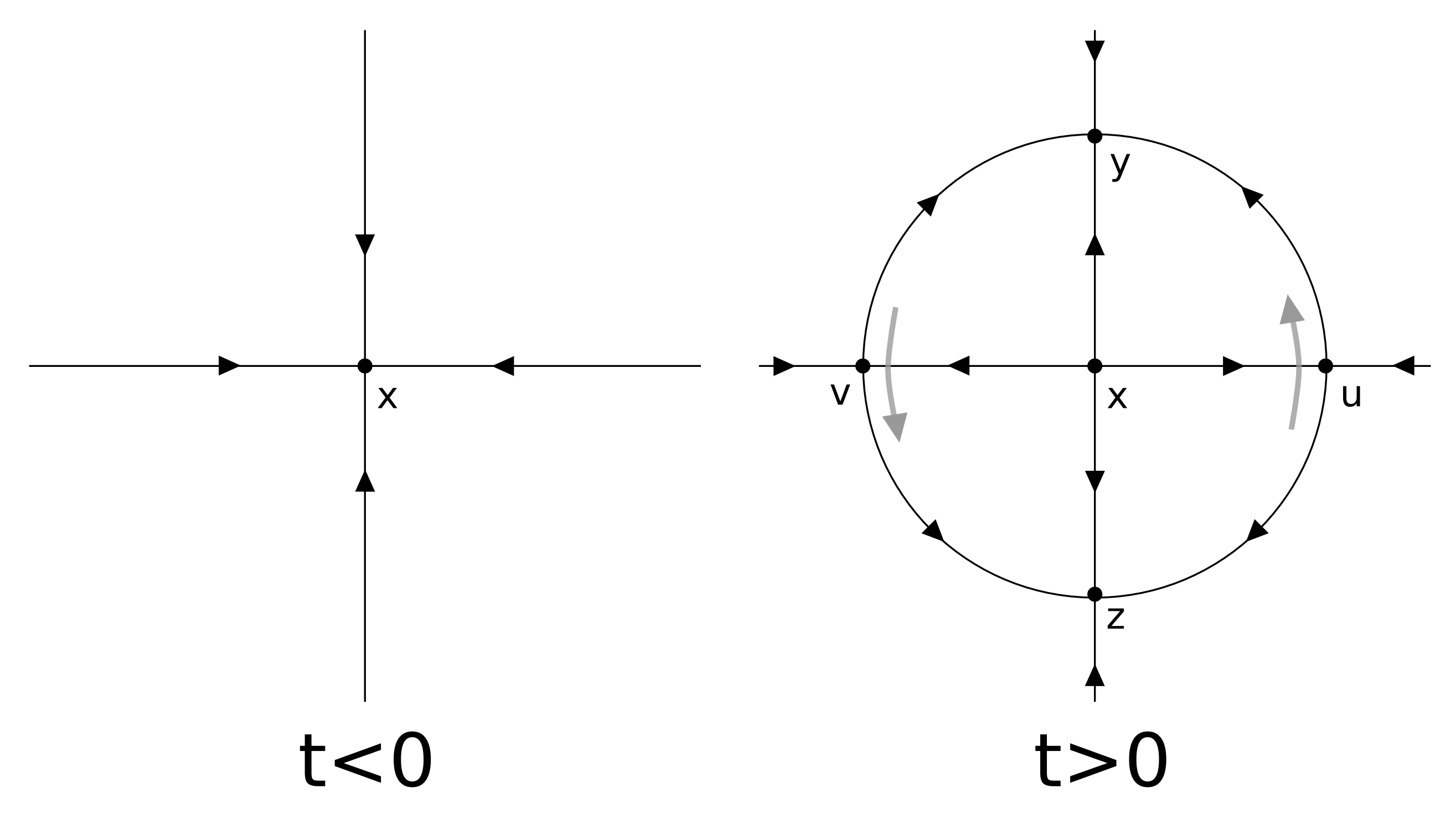}
\caption{{\small Symmetry is given by reflection with respect to the vertical axis. $\HM(f_0,x)=\HM(f_0,x)^{\Z_2}$ has a single generator in degree $0$.}}
\end{center}
\end{figure}

\begin{figure}\label{fig3}
\begin{center}
\includegraphics[width=110mm]{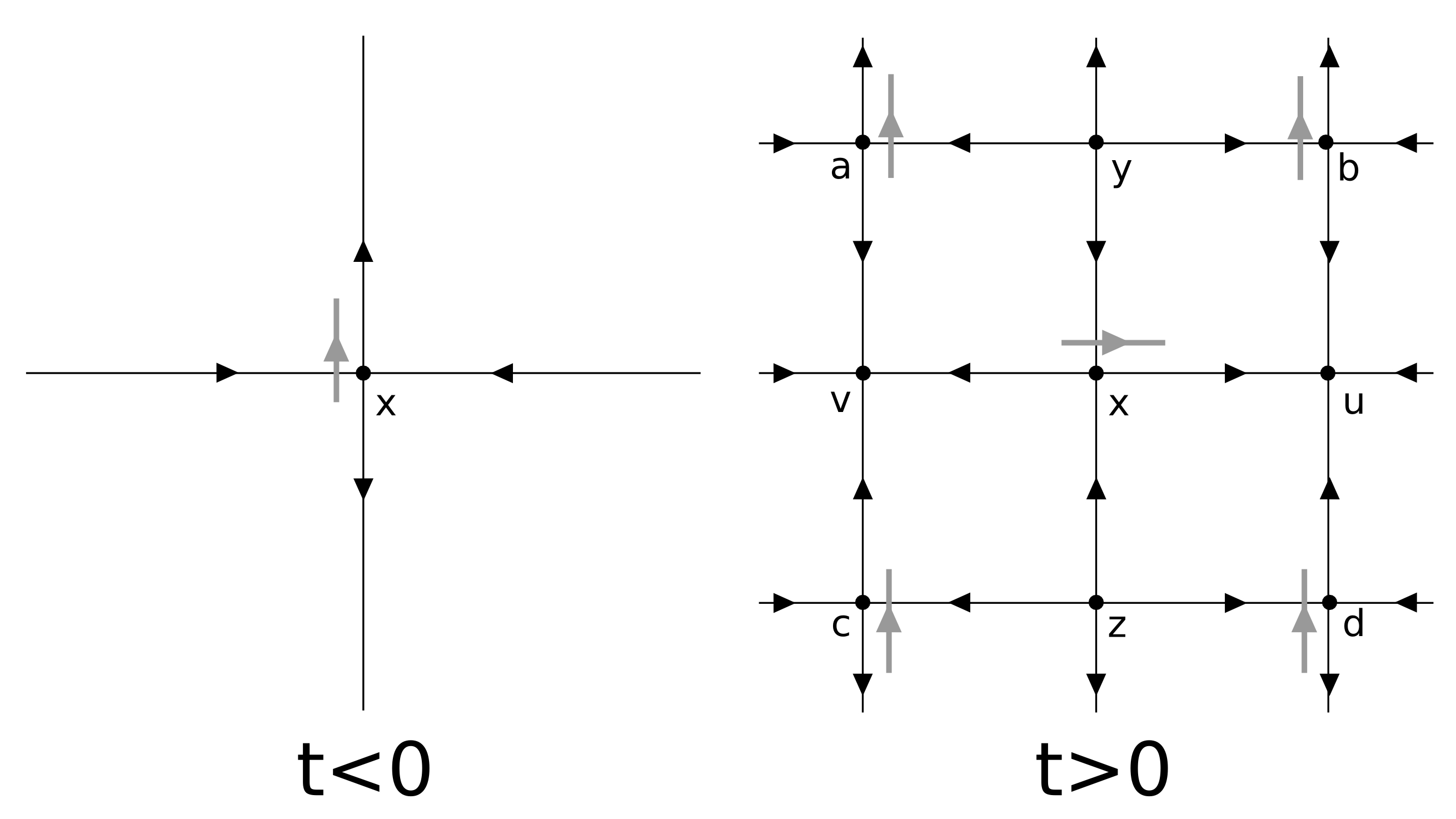}
\caption{{\small Symmetry is given by reflection with respect to the horizontal axis. $\HM(f_0,x)$ has a generator in degree $1$, $\HM(f_0,x)^{\Z_2}$ vanishes.}}
\end{center}
\end{figure}

Our final scenario, where eight critical points bifurcate from $x$, is shown in Figure~3. As in the first scenario, the $\Z_2$-action is generated by reflection along the horizontal axis. Grey arrows orient unstable manifolds of saddles, for all $t\neq 0$. For $t>0$ unstable manifolds of local maxima are oriented by the canonical orientation of the plane, while those of local minima are oriented by $+1$. Looking at $t<0$ we conclude that $\HM(f_0,x)$ has a generator in degree $1$, while $\HM(f_0,x)^{\Z_2}$ vanishes. In fact, for $t<0$ the local Morse chain complex has a single generator $x$ in degree $1$, but $1\cdot x = -x$ since reflection reverses vertical grey arrows. By the continuation property (Proposition~\ref{prop_invariance_with_symmetries}) the same simple conclusion must be obtained by analyzing the more complicated Morse chain complex for $t>0$. We have nine generators: two $\{y,z\}$ in degree $2$, five $\{x,a,b,c,d\}$ in degree $1$ and two $\{u,v\}$ in degree $0$. The local Morse differential reads
\begin{equation*}
\begin{aligned}
& \partial y=b-a+x \\
& \partial z = d-c-x \\
& \partial x=u-v \\
& \partial a = -v = -\partial c \\
& \partial b = -u = -\partial d
\end{aligned}
\end{equation*}
and the action of $1\in\Z_2$ is
\begin{equation*}
\begin{aligned}
& 1\cdot y = -z && 1\cdot z = -y \\
& 1\cdot x = x && \\
& 1\cdot a = -c && 1\cdot c = -a \\
& 1\cdot b = -d && 1\cdot d = -b \\
& 1\cdot u = u && 1\cdot v = v.
\end{aligned}
\end{equation*}
We used that reflection reverses orientations of vertical arrows, preserves orientations of horizontal arrows, and reverts orientations of the plane. The reader can check that $1\cdot$ commutes with $\partial$. The subcomplex of invariant chains has a generator $\{y-z\}$ in degree $2$, three generators $\{x,a-c,b-d\}$ in degree $1$, and two $\{u,v\}$ in degree $0$. Note that $\partial (y-z) = 2x+b-d+c-a$, that the closed invariants chains in degree $1$ are precisely generated by $\partial(y-z)$, and that both $u$ and $v$ are exact. This is in agreement with the vanishing of $\HM(f_0,x)^{\Z_2}$.

\subsubsection{Global transversality with symmetry can not be achieved by $C^2$-small perturbations}\label{sssec_transv_impossible}

It is not always possible to perturb a symmetric pair to a symmetric Morse-Smale pair, as the following simple and well-known example shows. Consider a `bagel-like' $2$-torus embedded in $\R^3$. We cut the bagel open with respect to a plane, and assume that the bagel is symmetric with respect to $\Z_2$-action generated by reflection along this cutting plane. As a $\Z_2$-symmetric Morse function we choose a height function along an axis in the cutting plane, see Figure~4. The metric is the one inherited from the Euclidean metric in $\R^3$. The gradient vector field must be tangent to the fixed-point set by symmetry. Hence for any $C^2$-small symmetric perturbation the grey circle will contain two saddles and anti-gradient trajectories connecting them. Such a configuration is not allowed by the Morse-Smale condition.

\begin{figure}\label{fig4}
\begin{center}
\includegraphics[width=140mm]{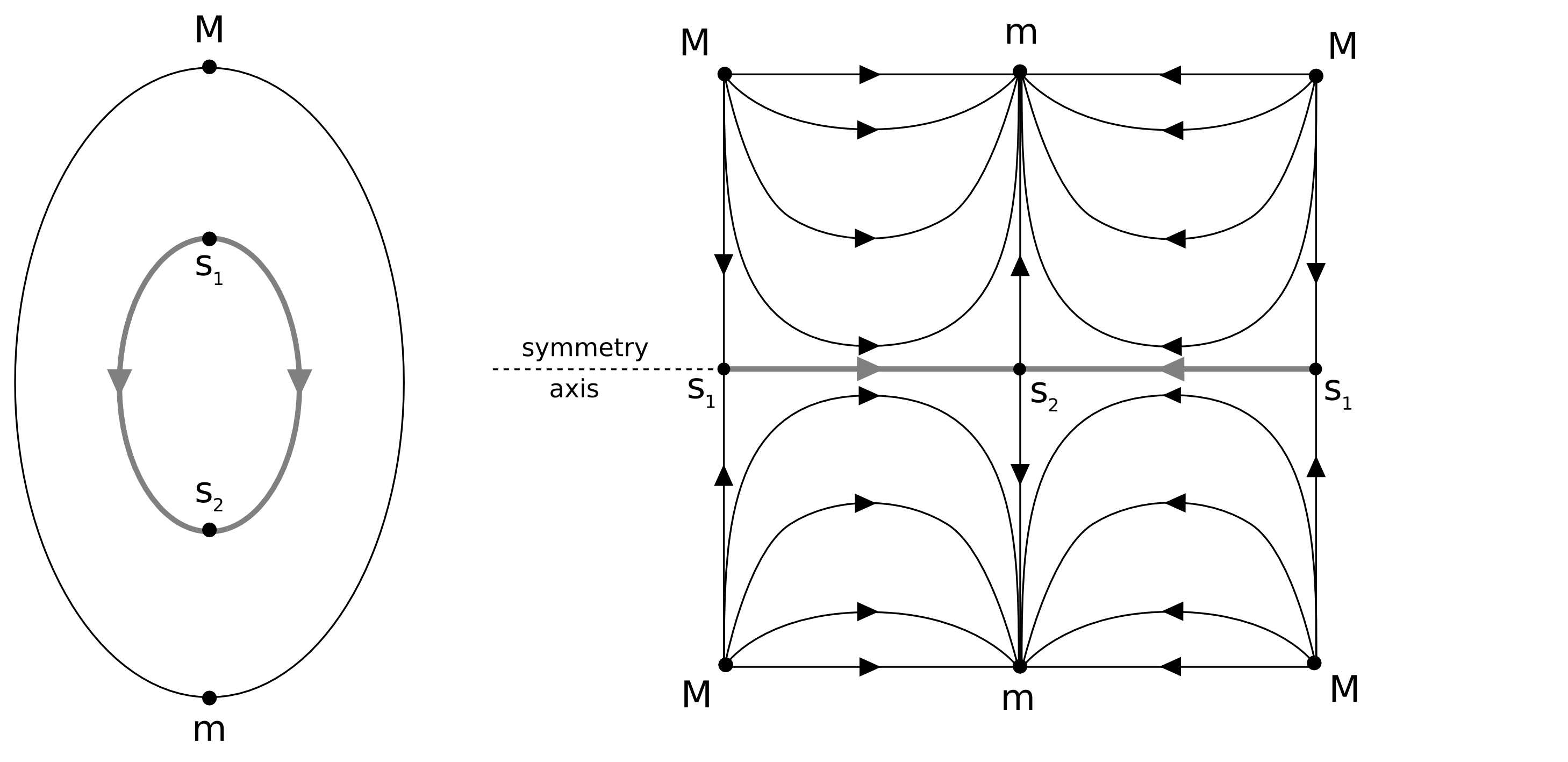}
\caption{\small{On the left we see the bagel cut open along the plane of the paper. The height function is along an axis in this plane through the critical points, it has its maximum at $M$, two saddles at $s_1,s_2$ and its minimum at $m$. On the right we `unwrap' the bagel and draw some anti-gradient flow lines. In both pictures, we draw in grey the saddle-saddle connection in one of the components of the fixed-point set of the action. Such a saddle-saddle connection can not be destroyed by a $\Z_2$-symmetric $C^2$-small perturbation.}}
\end{center}
\end{figure}

\begin{figure}\label{fig5}
\begin{center}
\includegraphics[width=110mm]{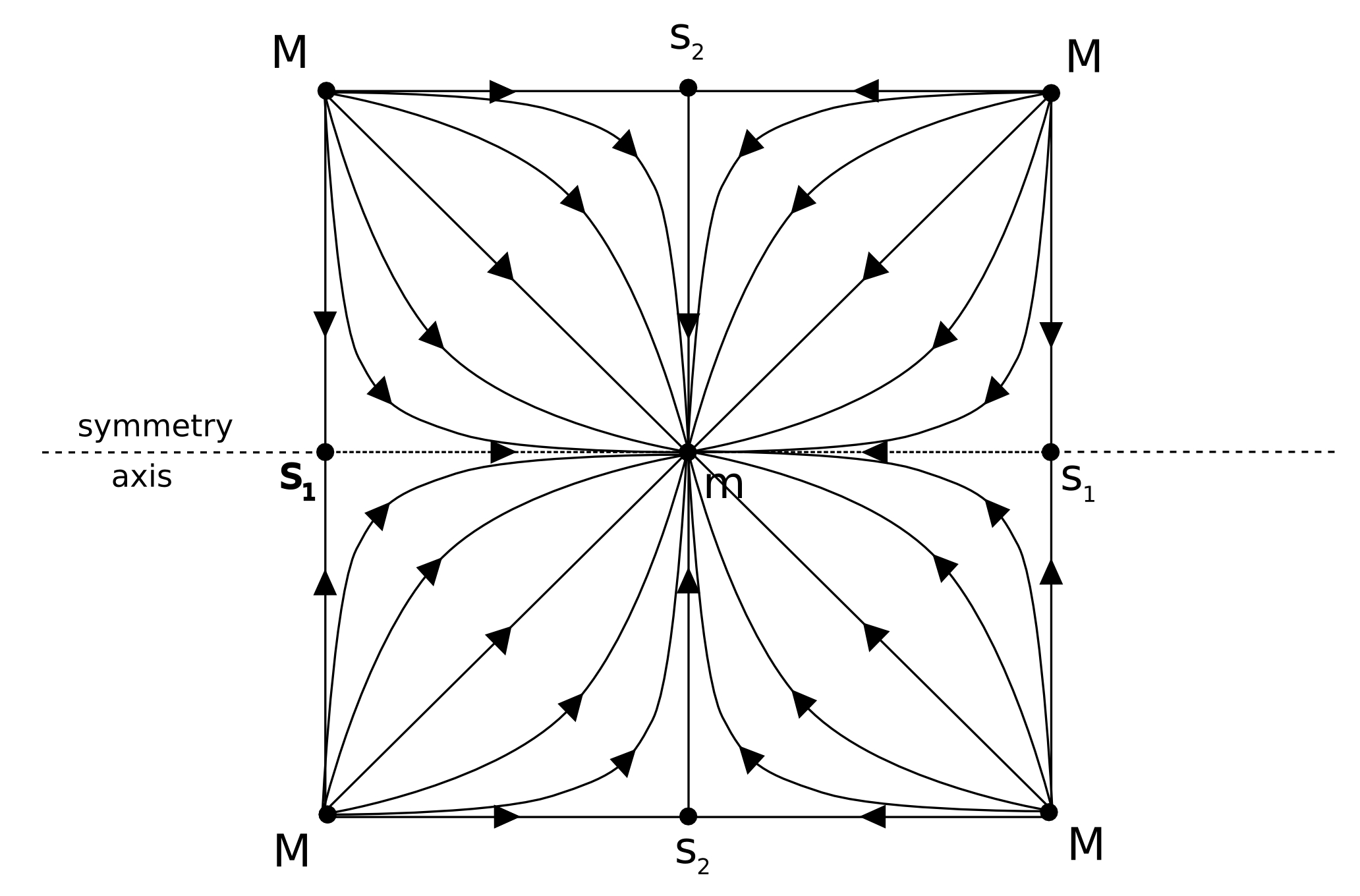}
\caption{{\small The homology of the subcomplex of $\Z_2$-invariant chains has one generator in degrees $1$ and $0$, thus agreeing with the $\Z_2$-equivariant homology.}}
\end{center}
\end{figure}

\subsubsection{Morse homological description of $\Z_k$-equivariant homology}\label{sssec_global_equiv_hom}

Here we compute invariant Morse homology in a simple example to illustrate the fact that invariant and equivariant homologies coincide for finite-cyclic group actions, with $\Q$-coefficients. The proof of this fact was given in~\cite[appendix]{GHHM}, but only now with Theorem~\ref{main2} we get to know that invariant Morse-Smale pairs exist in closed manifolds.

Consider the same $\Z_2$-action on the $2$-torus from Figure~4, but the symmetric Morse function is the one shown in Figure~5. It forms with the obvious flat metric a symmetric Morse-Smale pair. There is again a maximum $M$, two saddles $s_1,s_2$ and a minimum $m$. Let us orient the unstable $2$-disk of $M$ by the canonical orientation of the plane, and the unstable $0$-disk of $m$ by $+1$. The unstable manifold of $s_1$ is oriented to the right, and that of $s_2$ is oriented downwards. The Morse differential vanishes: $\partial M = s_2-s_2 + s_1-s_1=0$, $\partial s_2 = \partial s_1 = m-m = 0$. By the definition of action, $1\in\Z_2$ acts as $1\cdot M=-M$, since reflection reverses orientations of the plane, $1\cdot s_1=s_1$ and $1\cdot s_2=-s_2$ since reflection preserves/reverts orientation on the horizontal/vertical axis, and $1\cdot m=m$ since $m$ is fixed. The subcomplex of invariant chains vanishes in degree $2$, is generated by $s_1$ in degree $1$ and by $m$ in degree $0$. Since the differential vanishes, the homology of this subcomplex has a generator in degrees $0$ and $1$, agreeing with $\Z_2$-equivariant homology.

\section{Local Morse homology and the discrete action functional}\label{sec_properties}

In this section we establish basic properties of invariant and non-invariant local Morse homology groups of discrete action functionals, which are in complete analogy to properties of local Floer homology and local contact homology.

\subsection{Definition and invariance of local Morse homology with symmetries}\label{ssec_invariant}

In the absence of symmetries the material discussed here is absolutely standard and well-known. Aiming at the case where a finite-cyclic group symmetry is present, we start with a discussion of the non-symmetric case.

Let $U$ be an isolating small open neighborhood for $(f,x)$, where $f$ is a smooth function defined on some manifold $X$ without boundary and $x$ is an isolated critical point of~$f$. If a Riemannian metric $\theta$ on $X$ is fixed then one can find an arbitrarily $C^2$-small (even $C^\infty$-small) perturbation $(f',\theta')$ of $(f,\theta)$ which is Morse-Smale on~$U$. For each $j$ one considers the vector space $\CM_j(f',\theta',U)$ over $\Q$ freely generated by the critical points of $f'$ which lie on $U$ and have Morse index equal to $j$. A differential $\partial^{(f',\theta',U)}$ on the graded vector space $\CM_*(f',\theta',U)$ is defined by counting negative $(f',\theta')$-gradient trajectories contained in $U$ connecting critical points in $U$ of index difference one. This differential depends on choices of orientations of the unstable manifolds. We denote the associated homology groups by $\HM_*(f',\theta',U)$. The dependence on the choice of orientations is not made explicit in the notation.

It turns out that for two such pairs $(f',\theta')$, $(f'',\theta'')$ which are $C^2$-close enough to $(f,\theta)$, and Morse-Smale on $U$, there exist chain maps
\begin{equation}\label{chain_map_generic}
\left( \CM_*(f',\theta',U) , \partial^{(f',\theta',U)} \right) \to \left( \CM_*(f'',\theta'',U) , \partial^{(f'',\theta'',U)} \right)
\end{equation}
induced by a certain piece of extra data. Such chain maps can be defined as so-called Floer continuation maps. It follows from the particular way that Floer continuation maps are defined, that two chain maps~\eqref{chain_map_generic} associated to a fixed pair of pairs $(f',\theta')$, $(f'',\theta'')$, together with corresponding choices of orientations, are chain homotopic. Hence the map induced on homology
\begin{equation}\label{floer_continuation_maps_hom}
\Phi_{(f',\theta')}^{(f'',\theta'')} : \HM_*(f',\theta',U) \to \HM_*(f'',\theta'',U)
\end{equation}
does not depend on the extra data. Moreover, when $(f'',\theta'')=(f',\theta')$ and orientations are chosen to be equal, the map on homology is the identity, and these maps make diagrams such as
\begin{equation}\label{canonical_maps_commute}
\xymatrix{
\HM_{*}(f',\theta',U) \ar[d]_{\Phi_{(f',\theta')}^{(f''',\theta''')}} \ar[r]^{\Phi_{(f',\theta')}^{(f'',\theta'')}} & \HM_{*}(f'',\theta'',U) \ar[ld]^{\Phi_{(f'',\theta'')}^{(f''',\theta''')}} \\
\HM_*(f''',\theta''',U)
}
\end{equation}
commutative. All this holds for all pairs $(f',\theta'),(f'',\theta''),(f''',\theta''')$ on a fixed and small $C^2$-neighborhood of $(f,\theta)$, and is proved using a compactness-gluing argument which is ubiquitous in Floer theory, see the book of Schwarz~\cite{schwarz_book} for details.

\begin{remark}\label{rmk_equiv_relations_vector_spaces}
Suppose one is given a collection $\{V_\lambda\}_{\lambda\in\Lambda}$ of vector spaces, and for each pair $(\lambda_1,\lambda_2)\in\Lambda\times\Lambda$ an isomorphism $\Psi_{\lambda_1}^{\lambda_2}:V_{\lambda_1}\to V_{\lambda_2}$ such that $\Psi_\lambda^\lambda = id_{V_\lambda}$ and $\Psi_{\lambda_2}^{\lambda_3} \circ \Psi_{\lambda_1}^{\lambda_2} = \Psi_{\lambda_1}^{\lambda_3}$. Then on $\sqcup_{\lambda\in\Lambda}V_\lambda$ there is an equivalence relation defined by $v_1\sim v_2$ if, and only if $v_2=\Psi_{\lambda_1}^{\lambda_2}(v_1)$, where $v_i\in V_{\lambda_i}$. The associated quotient space, denoted by $V$, has the structure of a vector space such that for each $\lambda$ the quotient projection restricts to an isomorphism $V_\lambda \stackrel{\sim}{\to} V$. 
\end{remark}

The vector spaces $\HM_*(f',\theta',U)$ and maps $\Phi_{(f',\theta')}^{(f'',\theta'')}$, where $(f',\theta')$ and $(f'',\theta'')$ vary on a fixed $C^2$-small neighborhood of $(f,\theta)$, fit in the discussion of Remark~\ref{rmk_equiv_relations_vector_spaces}. We get the local Morse homology $\HM_*(f,\theta,x)$. As the notation suggests, this is independent of $U$. This is easily proved by noting that a small perturbation forces critical points and connecting trajectories to be contained in arbitrarily small neighborhoods of the origin.

Consider $\F = \{f_\lambda\}_{\lambda\in\Lambda}$ a smooth family of smooth functions defined on some manifold $M$, where $\lambda$ varies on the parameter space $\Lambda$. Suppose that $p\in M$ is a common critical point of all $f_\lambda$. One says that $p$ is a {\it uniformly isolated} critical point of $\F$ if there exists a neighborhood of $p$ which is an isolating neighborhood of $(f_\lambda,p)$ for all $\lambda$. Invariance properties in local Morse homology without symmetries are well-known, we summarize them in the statement below which can be proved in a standard fashion using Floer-type continuation maps. In particular, the independence of local Morse homology with respect to the metric follows as a corollary.

\begin{proposition}\label{prop_invariance_non_invariant}
Let $\F = \{f_s\}_{s\in[a,b]}$ be a smooth family of smooth real-valued functions defined on a manifold without boundary, and let $x$ be a uniformly isolated critical point of $\F$. Then for any family of metrics $\{\theta_s\}_{s\in[a,b]}$ there is a special family of so-called continuation isomorphisms
\[
\Theta(\F)_{s_0}^{s_1}:\HM(f_{s_0},\theta_{s_0},x) \to \HM(f_{s_1},\theta_{s_1},x)
\]
parametrized by $a\leq s_0\leq s_1\leq b$, satisfying $\Theta(\F)_{s_0}^{s_2} = \Theta(\F)_{s_1}^{s_2} \circ \Theta(\F)_{s_0}^{s_1}$ for all $a\leq s_0\leq s_1\leq s_2\leq b$. Moreover, $\Theta(\F)_a^b$ depends only on the homotopy class of $\F$ keeping endpoints fixed. 
\end{proposition}

It follows from the above statement and constructions that the local Morse homology of $(f,x)$, denoted as
\begin{equation}
\HM_*(f,x)
\end{equation}
is defined independently of choices of metrics, and stays constant under deformations $(f_s,x)$ through families where $x$ is a uniformly isolated critical point of $\{f_s\}$. This concludes our discussion of the non-invariant case.

We are interested in a version of Proposition~\ref{prop_invariance_non_invariant} under finite cyclic group symmetries. However, we do not have the transversality statement with symmetries necessary to use Floer-theoretic methods to show that continuation maps~\eqref{floer_continuation_maps_hom} are $\Z_k$-equivariant when the local Morse-Smale pairs are $\Z_k$-symmetric. We need another approach. Let us start by recalling the well-known interplay between local Morse homology and more classical local critical groups {\it \`a la} Gromoll-Meyer~\cite{GM}.

Consider a pair $(f,\theta)$. Let an isolated critical point $x$ of $f$ be given. Choose an  open, relatively compact, isolating neighborhood $U$ for $(f,x)$. 

\begin{definition}[Gromoll-Meyer pairs]\label{def_GM_pairs}
A {\it Gromoll-Meyer pair} for $(f,x)$ (in $U$) is a pair $W_- \subset W \subset U$ of closed subsets of $U$ with the following property. There exists a $C^2$-neighborhood $\mathcal N_1$ of $(f,\theta)$ such that for every $(f',\theta')\in\mathcal N_1$ which is Morse-Smale on $U$ one finds an isomorphism
\begin{equation}\label{invariance_map_local_hom}
\tilde\Psi_{(f',\theta',U)} : H_*(W,W_-) \stackrel{\sim}{\to} \HM_*(f',\theta',U).
\end{equation}
Moreover, these maps and the continuation maps~\eqref{floer_continuation_maps_hom} satisfy
\begin{equation}\label{local_continuation_and_pairs}
\Phi_{(f',\theta')}^{(f'',\theta'')} = \Psi_{(f'',\theta'',U)} \circ \left( \Psi_{(f',\theta',U)} \right)^{-1}.
\end{equation}
Here $H_*$ stands for singular homology with rational coefficients.
\end{definition}

Following Conley~\cite{conley}, see also~\cite{salamon} for the instructive case of closed manifolds, one constructs (non-uniquely) Gromoll-Meyer pairs. This construction will be revised in Appendix~\ref{app_invariance}, along with a proof of the proposition below. The maps~\eqref{invariance_map_local_hom} satisfying~\eqref{local_continuation_and_pairs} induce a canonical isomorphism $$ \HM(f,x) \simeq H(W,W_-). $$

\begin{proposition}\label{prop_GM_pairs}
Assume further that the ambient manifold is equipped with a $\Z_k$-action and that $f,\theta,x,U$ are $\Z_k$-invariant. Then there exists a Gromoll-Meyer pair $(W,W_-)$ in $U$ where both sets $W,W_-$ are $\Z_k$-invariant, and a $C^2$-neighborhood $\mathcal{N}_1$ of $(f,\theta)$ such that for every $(f',\theta')\in\mathcal N_1$ which is Morse-Smale on $U$  and $\Z_k$-invariant, the map~\eqref{invariance_map_local_hom} is $\Z_k$-equivariant.
\end{proposition}

\begin{remark}
Theorem~\ref{main1} guarantees that the above statement does not concern an empty set of pairs.
\end{remark}

Equivariance of~\eqref{invariance_map_local_hom} is to be understood as follows. The sets $W,W_-$ are $\Z_k$-invariant. If $a:X\to X$ is the diffeomorphism given by the action of $1\in\Z_k$ then on $C_*(W,W_-)$ we consider the $\Z_k$-action by chain maps generated by $a_*$. This induces a $\Z_k$-action on $H_*(W,W_-)$. On $\CM_*(f_1,\theta_1,U)$ we consider the $\Z_k$-action by chain maps described in~\ref{sssec_inv_MH}. Similarly, this induces a $\Z_k$-action on $\HM_*(f_1,\theta_1,U)$. Equivariance of~\eqref{invariance_map_local_hom} is meant with respect to these actions. 

\begin{remark}
As observed above, singular homology of a Gromoll-Meyer pair computes local Morse homology. In Appendix~\ref{app_invariance}, Proposition~\ref{prop_GM_pairs} will be proved by first recalling the construction of $W,W_-,\tilde\Psi_{(f',\theta',U)}$ and then showing that $\tilde\Psi_{f',\theta',U}$ is $\Z_k$-equivariant when the local Morse-Smale pair $(f',\theta')\in\mathcal N_1$ is $\Z_k$-symmetric.
\end{remark}

\begin{remark}
Let $(C_*,\partial)$ be any chain complex over
$\Q$ with a $\Z_k$-action by chain maps. This action induces a $\Z_k$-action on homology $H_*(C,\partial)$. Let $H_*(C,\partial)^{\Z_k} \subset H_*(C,\partial)$ be the subspace of invariant homology classes, and let $(C^{\Z_k}_*,\partial)$ be the subcomplex of invariant chains. There is a natural isomorphism $H_*(C,\partial)^{\Z_k} \simeq H_*(C^{\Z_k},\partial)$. In fact, on $(C_*,\partial)$ we have an averaging chain map $A$ defined by 
\[
Ac = \frac{1}{k} ( c + 1\cdot c + \dots + (k-1)\cdot c )
\]
Then $C^{\Z_k}_* = {\rm im}\ A$. The induced averaging map on homology, still denoted by $A$, also satisfies $H_*(C,\partial)^{\Z_k} = {\rm im}\ A$. Using these operators one checks that the map $H_*(C^{\Z_k},\partial) \to H_*(C,\partial)$ induced by inclusion of complexes is injective, and its image coincides with the image of $A$, i.e. with $H_*(C,\partial)^{\Z_k}$. This provides the desired isomorphism. Naturality in the category of chain complexes over $\Q$ with a $\Z_k$-action by chain maps is again easy to check. $H_*(C,\partial)^{\Z_k}$ is called {\it $\Z_k$-invariant homology}.
\end{remark}


Let $W_-$, $W$, $\mathcal N_1$ be given by Proposition~\ref{prop_GM_pairs}. It follows from this proposition 
that we have induced maps
\[
\overline\Phi_{(f',\theta')}^{(f'',\theta'')} : \HM_*(f',\theta',U)^{\Z_k} \to \HM(f'',\theta'',U)^{\Z_k}
\]
on invariant homology, and that the spaces $\HM_*(f',\theta',U)^{\Z_k}$ and the maps $\overline\Phi_{(f',\theta')}^{(f'',\theta'')}$ again fit in the discussion of Remark~\ref{rmk_equiv_relations_vector_spaces}, where $(f',\theta')$, $(f'',\theta'')$ belong to $\mathcal{N}_1$, are Morse-Smale on $U$ and $\Z_k$-invariant. The induced vector space, denoted by $\HM(f,\theta,x)^{\Z_k}$, is of course isomorphic to $H(W,W_-)^{\Z_k}$. As the notation suggests, this is independent of $U$. A consequence of Proposition~\ref{prop_GM_pairs} is the following analogue of Proposition~\ref{prop_invariance_non_invariant} in the presence of a $\Z_k$-action.

\begin{proposition}\label{prop_invariance_with_symmetries}
Let $\F = \{f_s\}_{s\in[a,b]}$ be a smooth family of $\Z_k$-invariant functions defined on a manifold without boundary equipped with a $\Z_k$-action, and let $x$ be a fixed point of the action, which is also a uniformly isolated critical point of $\F$. For any family $\{\theta_s\}_{s\in[a,b]}$ of $\Z_k$-invariant metrics there exist isomorphisms $$ \Theta^{\Z_k}(\F)_{s_0}^{s_1}:{\HM}(f_{s_0},\theta_{s_0},x)^{\Z_k} \to {\HM}(f_{s_1},\theta_{s_1},x)^{\Z_k} $$ parametrized by $a\leq s_0\leq s_1\leq b$, satisfying $\Theta^{\Z_k}(\F)_{s_0}^{s_2} = \Theta^{\Z_k}(\F)_{s_1}^{s_2} \circ \Theta^{\Z_k}(\F)_{s_0}^{s_1}$ for all $a\leq s_0\leq s_1\leq s_2\leq b$. Moreover, $\Theta^{\Z_k}(\F)_a^b$ depends only on the homotopy class of $\F$ keeping endpoints fixed.
\end{proposition}

Let ${\rm Met}^{\Z_k}$ denote the space of $\Z_k$-invariant metrics on $X$. Given $\theta',\theta''\in{\rm Met}^{\Z_k}$ one considers the family $(1-t)\theta'+t\theta''$, $t\in[0,1]$, of $\Z_k$-invariant metrics which, by Proposition~\ref{prop_invariance_with_symmetries}, yields an isomorphism $\HM(f,\theta',x)^{\Z_k}\to \HM(f,\theta'',x)^{\Z_k}$. These spaces and maps fit into the discussion of Remark~\ref{rmk_equiv_relations_vector_spaces} because ${\rm Met}^{\Z_k}$ is convex. Hence we obtain what we call {\it invariant local Morse homology}
\begin{equation*}
\HM(f,x)^{\Z_k}.
\end{equation*}
Again by the proposition, this is invariant under deformations $(f_s,x)$ through families of $\Z_k$-invariant functions keeping $x$ as a uniformly isolated critical point.

Before moving on to Hamiltonian dynamics, we study the relationship between the {\it continuation maps} from the above propositions and {\it direct sum maps} defined as follows. Let $x\in X$, $y\in Y$ be isolated critical points of $f:X\to\R$, $g:Y\to \R$ where $X,Y$ are manifolds without boundary. Fix metrics $\theta,\alpha$ on $X,Y$ and $U,V$ isolating neighborhoods of $x,y$ respectively. Let $(f',\theta')$, $(g',\alpha')$ be $C^2$-small perturbations of $(f,\theta)$, $(g,\alpha)$ which are Morse-Smale on $U,V$ respectively. Then $(f'\oplus g',\theta'\oplus\alpha')$ is Morse-Smale on $U\times V$ and there exists a chain isomorphism
\begin{equation}\label{isom_tensor_prod_chain}
\begin{aligned}
&(\CM(f'\oplus g',\theta'\oplus\alpha',U\times V),\partial^{\rm Morse}) \\
&= (\CM(f',\theta',U),\partial^{\rm Morse}) \otimes (\CM(g',\alpha',V),\partial^{\rm Morse})
\end{aligned}
\end{equation}
if the orientation of unstable manifold of the critical point $(x,y)$ of $f'\oplus g'$ is the product of the chosen orientations of the unstable manifolds of $x$ and of $y$.
Since we use rational coefficients, the K\"unneth formula yields
\begin{equation}\label{isom_kunneth_formula}
\HM(f\oplus g,(x,y)) \stackrel{\sim}{\to} \HM(f,x) \otimes \HM(g,y).
\end{equation}
If $y$ is a non-degenerate critical point of index $\mu$ then we have an isomorphism
\begin{equation}\label{isom_consideration_no_symm}
\Xi : \HM_*(f\oplus g,(x,y)) \stackrel{\sim}{\to} \HM_{*-\mu}(f,x)
\end{equation}
referred to as a {\it direct sum map}.

In the presence of $\Z_k$-symmetry we claim that~\eqref{isom_consideration_no_symm} is either $\Z_k$-equivariant or $\Z_k$-anti\-equivariant. To see this, we assume that $X,Y$ are equipped with $\Z_k$-actions, and that $f,\theta,x,U$ and $g,\alpha,y,V$ are $\Z_k$-invariant. On $X\times Y$ we consider the induced diagonal $\Z_k$-action. Let us denote the diffeomorphisms generating the actions on $X$, $Y$ or $X\times Y$ all by $1$, with no fear of ambiguity. The definition of the $\Z_k$-actions on the chain complexes in~\eqref{isom_tensor_prod_chain} explained in Section~\ref{sssec_inv_MH}, together with the particular choices of orientations of unstable manifolds which make~\eqref{isom_tensor_prod_chain} valid, imply that under the isomorphism~\eqref{isom_kunneth_formula} the generator $1_*$ on the left-hand side corresponds to $1_*\otimes 1_*$ on the right-hand side. Now assume again that $y$ is a non-degenerate critical point of $g$ of index $\mu$. The linear $\Z_k$-action on $T_yY$ preserves the negative eigenspace $E_-$ of the Hessian $D^2g(y)$, and there are two cases: either it preserves orientation on $E_-$, or reverses it. If it preserves the orientation then $1_*$ acts on $\HM(g,y)$ as the identity and $\Xi$ in~\eqref{isom_consideration_no_symm} satisfies
\begin{equation*}
\Xi \circ 1_* = 1_* \circ \Xi.
\end{equation*}
If it reverses then $1_*$ acts on $\HM(g,y)$ as minus the identity and $\Xi$ in~\eqref{isom_consideration_no_symm} satisfies
\begin{equation*}
\Xi \circ 1_* = - (1_* \circ \Xi).
\end{equation*}
Summarizing we get the following statement: If $1\in\Z_k$ preserves orientations on $E_-$ then $\Xi$ is $\Z_k$-equivariant and induces an isomorphism
\begin{equation}\label{isom_consideration_with_symm}
\Xi^{\Z_k} : \HM_*(f\oplus g,(x,y))^{\Z_k} \stackrel{\sim}{\to} \HM_{*-\mu}(f,x)^{\Z_k}.
\end{equation}
In fact we know slightly more: if $1\in\Z_k$ reverses orientations on $E_-$ then $\Xi$ induces an isomorphism between $\HM_*(f\oplus g,(x,y))^{\Z_k}$ and the $-1$-eigenspace of the $\Z_k$-action on $\HM_{*-\mu}(f,x)$ (which perhaps deserves to be called {\it anti-invariant local Morse homology} and might eventually find dynamical applications).

It follows from the definitions that if $\{f_s\}_{s\in[a,b]}$ is a family having $x$ as a uniformly isolated critical point, and $\{g_s\}_{s\in[a,b]}$ is a family having $y$ as a non-degenerate critical point, then $\{f_s\oplus g_s\}$ has $(x,y)$ as a uniformly isolated critical point and 
\begin{equation}\label{diagram_commutes_continuation_1}
\xymatrix{
\HM_*(f_a,x) \ar[r]^{\Theta} & \HM_*(f_b,x)  \\
\HM_{*+\mu}(f_a\oplus g_a,(x,y)) \ar[u]^{\Xi} \ar[r]^{\Theta} & \HM_{*+\mu}(f_b\oplus g_b,(x,y)) \ar[u]^{\Xi}
}
\end{equation}
commutes, where $\Theta$ are continuation maps. In the presence of symmetry there is a symmetric version of the above commutative diagram
\begin{equation}\label{diagram_commutes_continuation_2}
\xymatrix{
\HM_*(f_a,x)^{\Z_k}  \ar[r]^{\Theta^{\Z_k}} & \HM_*(f_b,x)^{\Z_k}  \\
\HM_{*+\mu}(f_a\oplus g_a,(x,y))^{\Z_k} \ar[u]^{\Xi^{\Z_k}} \ar[r]^{\Theta^{\Z_k}} & \HM_{*+\mu}(f_b\oplus g_b,(x,y))^{\Z_k} \ar[u]^{\Xi^{\Z_k}}
}
\end{equation}
provided that the generator of the induced linear $\Z_k$-action on $T_yY$ preserves the orientation on the negative space of the Hessian of $g_s$ at $y$, for all $s$.

\subsection{Isolated periodic points and their local invariants}

The link between local Hamiltonian dynamics and local Morse theory can be achieved through generating functions. This is a classical subject that goes back to Poincar\'e~\cite{poincare}.

\subsubsection{Generating functions}\label{sssec_gen_functions}

Let $\varphi$ be a germ of symplectic diffeomorphism defined near the fixed point $0$ in $(\R^{2m},\omega_0)$. We use $(x,y)$ to indicate the Lagrangian splitting $\R^{2m} = (\R^m\times0)\oplus (0\times \R^m) \simeq \R^m\times\R^m$. If
\begin{itemize}
\item[(Gen1)] $\R^{2m} = (\R^m\times 0) \oplus d\varphi(0)(0\times \R^m)$
\end{itemize}
holds then, denoting $(\bar x,\bar y) = \varphi(x,y)$, the map $(x,y) \mapsto (x,\bar y)$ defines a diffeomorphism near the origin. Hence we can use $(x,{\bar y})$ as independent coordinates, and consider the $1$-form $\eta := (y-{\bar y})dx + (\bar x-x)d{\bar y}$ in $(x,{\bar y})$-space. It is closed because $\varphi$ is symplectic. Hence there is a primitive $S=S(x,{\bar y})$ near the origin. Such a germ of function $S$ is called a generating function for $\varphi$ and
\begin{itemize}
\item[(Gen2)] For all $(X,Y),(x,y)$ near the origin
\begin{equation*}
\varphi(x,y) = (X,Y) \ \ \ \Leftrightarrow \ \ \ \left\{ \begin{aligned} X-x &= \nabla_2S(x,Y) \\ y-Y &= \nabla_1S(x,Y) \end{aligned} \right.\ .
\end{equation*}
\end{itemize}
Note that $S$ is determined up to an additive constant. There are many other types of generating functions, see~\cite[chapter 9]{McDSal}. In this work the term {\it generating function} refers to those defined as above. 

\begin{lemma}\label{lemma_formula_hessian}
Suppose that the germ $\varphi$ satisfies (Gen1). If $S$ is the generating function as in (Gen2) then $D^2S(0)$ and $d\varphi(0)$ are related by $$ d\varphi(0)-I = -J_0 \ D^2S(0) \ dT(0) $$ where $T$ is the map $T(x,y)=(x,Y)$ with $Y$ defined by $(X,Y)=\varphi(x,y)$, and $J_0$ is the matrix
\[
J_0 = \begin{pmatrix} 0 & -I \\ I & 0 \end{pmatrix}.
\]
Thus, $dT(0)$ provides a linear isomorphism between $\ker (d\varphi(0)-I)$ and $\ker D^2S(0)$.
\end{lemma}

\begin{proof}
Using (Gen2) we get
\begin{equation*}
\begin{aligned}
d\varphi(0)-I &=
\begin{pmatrix}
\nabla_xX-I & \nabla_yX \\
\nabla_xY & \nabla_yY-I
\end{pmatrix} \\
&= \begin{pmatrix}
\nabla_{21}S+\nabla_{22}S\nabla_xY & \nabla_{22}S\nabla_yY \\
-\nabla_{11}S-\nabla_{12}S\nabla_xY & -\nabla_{12}S\nabla_yY
\end{pmatrix}
\end{aligned}
\end{equation*}
where partial derivatives of $S$, $X$ and $Y$ are evaluated at the origin. Simple inspection shows that the last matrix is equal precisely to $-J_0D^2S(0)dT(0)$.
\end{proof}

\subsubsection{Discrete action functionals as generating functions}

Consider a $1$-periodic Hamiltonian $H_t$ defined near $0\in\R^{2n}$ satisfying $dH_t(0)=0$, and choose $N$ large enough such that the local diffeomorphisms $\psi_i$ defined in~\eqref{germs_small_steps} are $C^1$-small and hence satisfy (Gen1). Note that $\psi_i$ is $N$-periodic in $i$. There is a unique $N$-periodic sequence of germs $S_i$ as in (Gen2) normalized by $S_i(0)=0$.


Consider the discrete action functional $\A_{H,k,N}$ defined as in~\eqref{def_discrete_action} and the symplectic diffeomorphism $\Phi_{H,k,N}$ defined near the origin in $\R^{2nkN}$ by
\begin{equation}\label{maps_big_Phi}
\begin{aligned}
& \Phi_{H,k,N}(z_1,\dots,z_{kN}) = (Z_1,\dots,Z_{kN}) \\
& Z_1 = \psi_{kN}(z_{kN}), \ Z_i = \psi_{i-1}(z_{i-1}) \ \forall i=2,\dots,kN .
\end{aligned}
\end{equation}

\begin{lemma}\label{lemma_action_as_gen_function}
$\Phi_{H,k,N}$ satisfies (Gen1) and $\A_{H,k,N}$ is a generating function for $\Phi_{H,k,N}$ as in (Gen2). Moreover, the nullity of $0\in\R^{2nkN}$ as a critical point of $\A_{H,k,N}$ is equal to $\nu(H,k) = \dim \ker d\varphi_H^k(0)-I$. In particular, the origin in $\R^{2nkN}$ is a non-degenerate critical point of $\A_{H,k,N}$ if, and only if, $d\varphi_H^k(0)-I$ is invertible.
\end{lemma}

\begin{proof}
Consider $x_i,y_i,X_i,Y_i \in \R^n$ defined by $z_i=(x_i,y_i)$, $Z_i=(X_i,Y_i)$. Set 
\[
\begin{aligned}
& x=(x_1,\dots,x_N), \ y=(y_1,\dots,y_N) \\ 
& X=(X_1,\dots,X_N), \ Y=(Y_1,\dots,Y_N) .
\end{aligned}
\]
Note that
\begin{equation}\label{formula_differential_big_phi}
\begin{aligned}
&d\Phi_{H,k,N}(0,\dots,0)\cdot(\delta z_1,\dots,\delta z_{kN}) = (\delta Z_1,\dots,\delta Z_{kN}) \ \ \text{and} \\
&\delta Z_i = d\psi_{i-1}(0)\cdot\delta z_{i-1}.
\end{aligned}
\end{equation}
Hence $\Phi_{H,k,N}$ satisfies (Gen1) since so does each $\psi_i$ by our choice of $N$. 

With these formulas we see that $1$ is an eigenvalue of $d\Phi_{H,k,N}$ if, and only if, it is an eigenvalue of $d\varphi^k_H(0)$, in which case their geometric multiplicities coincide.

Now we wish to show that
\begin{equation}\label{big_gen_func_identities}
( X, Y) = \Phi_{H,k,N}( x, y) \ \Leftrightarrow \left\{ \begin{aligned} &  X -  x = \nabla_2\A_{H,k,N}( x, Y) \\ &  y -  Y = \nabla_1\A_{H,k,N}( x, Y) \end{aligned} \right.
\end{equation}
holds for all $( x, y),( X, Y)$ close enough to the origin in $\R^{2nkN}$. But $( X, Y) = \Phi_{H,k,N}( x, y)$ if, and only if, $Z_i=\psi_{i-1}(z_{i-1})$, which happens precisely when
\begin{equation*}
\begin{aligned}
X_i - x_{i-1} &= \nabla_2S_{i-1}(x_{i-1},Y_i) \\
y_{i-1} - Y_i &= \nabla_1S_{i-1}(x_{i-1},Y_i).
\end{aligned}
\end{equation*}
Using the first of these equations we get
\begin{equation}\label{X_i-x_i}
\begin{aligned}
X_i-x_i 
&= X_i - x_{i-1} + x_{i-1} - x_i \\
&= x_{i-1} - x_i + \nabla_2S_{i-1}(x_{i-1},Y_i).
\end{aligned}
\end{equation}
Now since
\begin{equation*}
\A_{H,k,N}( x, Y) = \sum_{i=1}^{kN} x_i(Y_{i+1}-Y_i) + S_i(x_i,Y_{i+1})
\end{equation*}
where the index $kN+1$ is to be replaced by $1$, we get
\begin{equation}\label{nablaY_iA_H,k,N(vec x,vec Y)}
\nabla_{Y_i}\A_{H,k,N}( x, Y) = x_{i-1} - x_i + \nabla_2S_{i-1}(x_{i-1},Y_i).
\end{equation}
Combining~\eqref{X_i-x_i} and~\eqref{nablaY_iA_H,k,N(vec x,vec Y)} for every $i$ we obtain the first equation in the right hand side of~\eqref{big_gen_func_identities}. The second equation in the right hand side of~\eqref{big_gen_func_identities} is obtained analogously. We have proved the first assertion in the statement of the lemma. 

To prove the other assertions one uses Lemma~\ref{lemma_formula_hessian} to identify the kernel of the Hessian of $\A_{H,k,N}$ at the origin with the $1$-eigenspace of $d\Phi_{H,k,N}(0)$, and then one uses the above mentioned fact that this eigenspace is linearly isomorphic to $\ker d\varphi^k_H(0)-I$.
\end{proof}

\subsubsection{Definition of local invariants, with or without symmetry}

As above, consider a germ $H=H_t$ of a $1$-periodic Hamiltonian near $0 \in \R^{2n}$ satisfying $dH_t(0)=0$ for all $t$, and denote $\varphi = \varphi_H^1$. We say that $N\geq1$ is {\it adapted to $H$} if
\begin{equation}\label{N_adapted}
t_0<t_1, \ t_1-t_0 \leq (2N)^{-1} \ \Rightarrow \ \varphi_H^{t_1} \circ (\varphi_H^{t_0})^{-1} \ \text{satisfies (Gen1).}
\end{equation}
Assuming~\eqref{N_adapted}, the sequence of germs $\psi_i = \varphi_H^{i/N}\circ(\varphi_H^{(i-1)/N})^{-1}$ satisfies (Gen1) and we find generating functions $S_i$ for $\psi_i$ as in (Gen2), normalized by $S_i(0)=0$. The sequences $\psi_i$ and $S_i$ are $N$-periodic in $i\in\Z$. Fix $k\in\N$ and consider the discrete action functional $\A_{H,k,N}$ defined in~\eqref{def_discrete_action}. As observed before, the function $\A_{H,k,N}$ is invariant under the $\Z_k$-action generated by the shift map~\eqref{shift_map_discrete_action}.

Assume that $0\in\R^{2n}$ is an isolated fixed point of $\varphi^k$. Thus $0\in\R^{2nkN}$ is an isolated critical point of $\A_{H,k,N}$ and with the use of Theorem~\ref{main1} we can define the local $\Z_k$-invariant Morse homology groups $\HM(\A_{H,k,N},0)^{\Z_k}$ as explained in Section~\ref{sssec_inv_MH}. One may also consider the usual non-invariant local Morse homology $\HM(\A_{H,k,N},0)$.

\begin{lemma}[Inflation isomorphisms]\label{lemma_inflation}
If $N$ is adapted to $H$ as in~\eqref{N_adapted} then there exists an isomorphism
\begin{equation}\label{inflation_map_non_symmetric}
\mathscr{I}_N : \HM_{*}(\A_{H,k,N},0) \to \HM_{*+nk}(\A_{H,k,N+1},0) \\
\end{equation}
and an isomorphism
\begin{equation}\label{inflation_map_symmetric}
\mathscr{I}^{\Z_k}_N: \HM_{*}(\A_{H,k,N},0)^{\Z_k} \to \HM_{*+2nk}(\A_{H,k,N+2},0)^{\Z_k}.
\end{equation}
\end{lemma}

\begin{proof}
We only prove~\eqref{inflation_map_symmetric} since~\eqref{inflation_map_non_symmetric} is analogous and easier. The functional $\A_{H,k,N}$ is defined near the origin in $\R^{2nkN} \simeq (\R^{2n})^{Nk}$, where a typical point is $(z_1,\dots,z_{kN})$, $z_i = (x_i,y_i)$. It was constructed using the $N$-periodic sequence $S_i$ of generating functions near the origin in $\R^{2n}$. Now consider the $(N+2)$-periodic sequence obtained by inserting the germ $0$ twice between positions $\lambda N$ and $\lambda N+1$, $\lambda\in\Z$. This gives
$$ 
(\dots S_{N-1},S_N,0,0,S_1,S_2,\dots,S_{N-1},S_N,0,0,S_1,S_2 \dots) .
$$ 
For each $\lambda=1,\dots,k$ consider two extra variables $w^1_\lambda = (u^1_\lambda,v^1_\lambda)$ and $w^2_\lambda = (u^2_\lambda,v^2_\lambda)$. Now we have $k(N+2)$ variables belonging to $\R^{2n}$ and we can consider the function $\A^+$ near the origin in $\R^{2nk(N+2)}$ of the form $\A^+ = \A' + \A''$ where 
\[
\begin{aligned}
\A' = \sum_{\lambda=1}^k \sum_{i=1}^{N-1} x_{(\lambda-1)N+i}(&y_{(\lambda-1)N+i+1}-y_{(\lambda-1)N+i}) \\
&+ S_i(x_{(\lambda-1)N+i},y_{(\lambda-1)N+i+1})
\end{aligned}
\]
and
\[
\A'' = \sum_{\lambda=1}^k x_{\lambda N} (v^1_\lambda-y_{\lambda N}) + S_N(x_{\lambda N},v^1_\lambda) + u^1_\lambda(v^2_\lambda-v^1_\lambda) + u^2_\lambda(y_{\lambda N+1}-v^2_\lambda)
\]
where $kN+1$ is to be identified with $1$.

We claim that
\begin{equation}\label{inflation_homology}
\begin{aligned}
& \HM_*(\A_{H,k,N},0)^{\Z_k} \simeq  \HM_{*+2nk}(\A^+,0)^{\Z_k} .
\end{aligned}
\end{equation}
To this end consider new variables
\begin{equation*}
\begin{aligned}
& \zeta^1_\lambda := v^2_\lambda - v^1_\lambda \\
& \xi^1_\lambda := x_{\lambda N} - u^1_\lambda \\
& \zeta^2_\lambda := y_{\lambda N+1} - v^2_\lambda \\
& \xi^2_\lambda := u^1_\lambda - u^2_\lambda \\
& (\lambda=1,\dots,k) \\
\end{aligned}
\end{equation*}
where $kN+1$ is to be replaced by $1$. This gives a new set of independent variables
\[
(x_1,\dots,x_{kN},y_1,\dots,y_{kN},\xi^1_1,\dots,\xi^1_k,\xi^2_1,\dots,\xi^2_k,\zeta^1_1,\dots,\zeta^1_k,\zeta^2_1,\dots,\zeta^2_k)
\]
with respect to which $\A'$ keeps the same form, and $\A''$ takes the form
\[
\begin{aligned}
&\A'' = \\
&= \sum_{\lambda=1}^k \left( \begin{aligned} & x_{\lambda N}(y_{\lambda N+1}-\zeta^1_\lambda-\zeta^2_\lambda-y_{\lambda N}) + S_N(x_{\lambda N},y_{\lambda N+1}-\zeta^1_\lambda-\zeta^2_\lambda) \\ & + (x_{\lambda N}-\xi^1_\lambda)\zeta^1_\lambda + (x_{\lambda N}-\xi^1_\lambda-\xi^2_\lambda)\zeta^2_\lambda \end{aligned} \right) \\
&= \sum_{\lambda=1}^k - \xi^1_\lambda\zeta^1_\lambda - \xi^2_\lambda\zeta^2_\lambda - \xi^1_\lambda\zeta^2_\lambda + x_{\lambda N}(y_{\lambda N+1}-y_{\lambda N}) + S_N(x_{\lambda N},y_{\lambda N+1}-\zeta^1_\lambda-\zeta^2_\lambda).
\end{aligned}
\]
For $s\in[0,1]$ define a family $\A^+_s = \A' + \A''_s$ where
\[
\begin{aligned}
&\A''_s = \\
&= \sum_{\lambda=1}^k - \xi^1_\lambda\zeta^1_\lambda - \xi^2_\lambda\zeta^2_\lambda - s\xi^1_\lambda\zeta^2_\lambda + x_{\lambda N}(y_{\lambda N+1}-y_{\lambda N}) + S_N(x_{\lambda N},y_{\lambda N+1}-s\zeta^1_\lambda-s\zeta^2_\lambda).
\end{aligned}
\]

We claim that the origin in $\R^{2nk(N+2)}$ is a uniformly isolated critical point of the family~$\A^+_s$. To see this first compute partial derivatives with respect to $x$
\begin{equation}\label{partial_x_1}
\begin{aligned}
& i\in \{1,\dots,N-1\} \Rightarrow \\
& \nabla_{x_{(\lambda-1)N+i}}\A^+_s = \nabla_{x_{(\lambda-1)N+i}} \A' \\
& = y_{(\lambda-1)N+i+1}-y_{(\lambda-1)N+i} + \nabla_1S_i(x_{(\lambda-1)N+i},y_{(\lambda-1)N+i+1})
\end{aligned}
\end{equation}
and
\begin{equation}\label{partial_x_2}
\begin{aligned}
& \nabla_{x_{\lambda N}}\A^+_s = \nabla_{x_{\lambda N}}\A''_s \\
& = y_{\lambda N+1} - y_{\lambda N} + \nabla_1S_N(x_{\lambda N},y_{\lambda N+1}-s\zeta^1_\lambda -s\zeta^2_\lambda)
\end{aligned}
\end{equation}
where $\lambda \in \{1,\dots,k\}$. Now we compute partial derivatives with respect to $y$-variables
\begin{equation}\label{partial_y_1}
\begin{aligned}
& i\in \{2,\dots,N\} \Rightarrow \\
& \nabla_{y_{(\lambda-1)N+i}}\A^+_s \\
&= x_{(\lambda-1)N+i-1}-x_{(\lambda-1)N+i} + \nabla_2S_{i-1}(x_{(\lambda-1)N+i-1},y_{(\lambda-1)N+i})
\end{aligned}
\end{equation}
and
\begin{equation}\label{partial_y_2}
\begin{aligned}
& \nabla_{y_{(\lambda-1) N+1}}\A^+_s \\
& = x_{(\lambda-1)N} - x_{(\lambda-1)N+1} + \nabla_2S_N(x_{(\lambda-1)N},y_{(\lambda-1)N+1}-s\zeta^1_{\lambda-1}-s\zeta^2_{\lambda-1})
\end{aligned}
\end{equation}
where $\lambda$ is to be taken modulo $k$.
Finally we compute partial derivatives with respect to $\xi^1_\lambda,\zeta^1_\lambda,\xi^2_\lambda,\zeta^2_\lambda$
\begin{equation}\label{partial_xi_zeta}
\begin{aligned}
& \nabla_{\xi^1_\lambda}\A^+_s = -\zeta^1_\lambda -s\zeta^2_\lambda \\
& \nabla_{\xi^2_\lambda}\A^+_s = -\zeta^2_\lambda \\
& \nabla_{\zeta^1_\lambda}\A_s^+ = -\xi^1_\lambda - s\nabla_2S_N(x_{\lambda N},y_{\lambda N+1}-s\zeta^1_\lambda-s\zeta^2_\lambda) \\
& \nabla_{\zeta^2_\lambda}\A_s^+ = -\xi^2_\lambda -s\xi^1_\lambda - s\nabla_2S_N(x_{\lambda N},y_{\lambda N+1}-s\zeta^1_\lambda-s\zeta^2_\lambda).
\end{aligned}
\end{equation}
In formulas~\eqref{partial_x_1}-\eqref{partial_xi_zeta} all indices appearing as subscripts of $x,y$ are to be taken modulo $kN$, and the index $\lambda$ is to be taken modulo $k$. Let $U$ be a neighborhood of the origin in $\R^{2nkN}$ satisfying
\[
z \in U, \ \nabla\A_{H,k,N}(z)=0 \ \Rightarrow \ z=0.
\]
We claim that $U\times\R^{4nk}$ is a uniformly isolating neighborhood for the origin in $\R^{2nk(N+2)}$ with respect to the family $\{\A^+_s\}$. At a critical point in $U\times \R^{4nk}$, the first two identities in~\eqref{partial_xi_zeta} give $\zeta^1_\lambda = \zeta^2_\lambda=0$ for all $\lambda$. Substituting into \eqref{partial_x_1}-\eqref{partial_y_2}, we find that $(z_1,\dots,z_{kN})$ is a critical point of $\A_{H,k,N}$ in $U$. Hence $z_1=\dots=z_{kN}=0$. Since $\nabla S_i(0,0)=0$ for all $i$ we obtain $\xi_\lambda=0$ for all $\lambda$ from the last two identities in~\eqref{partial_xi_zeta}. Finally note that the family $\A^+_s$ is $\Z_k$-invariant under the corresponding shift map which takes the block $(z_{(\lambda-1)N+1},\dots,z_{\lambda N},(\xi^1_\lambda,\zeta^1_\lambda),(\xi^2_\lambda,\zeta^2_\lambda))$ to $(z_{\lambda N+1},\dots,z_{(\lambda+1)N},(\xi^1_{\lambda+1},\zeta^1_{\lambda+1}),(\xi^2_{\lambda+1},\zeta^2_{\lambda+1}))$. Moreover, $\A^+_1 = \A^+$ and
\begin{equation}\label{intermediate_action}
\A^+_0 = \A_{H,k,N} - \sum_{\lambda=1}^k (\xi^1_\lambda\zeta^1_\lambda+\xi^2_\lambda\zeta^2_\lambda).
\end{equation}
Note that the quadratic form in the right hand side of the equation above has $2nk$ positive and $2nk$ negative eigenvalues, its negative space is the diagonal 
\[
\{\xi_\lambda^1=\zeta_\lambda^1, \ \ \xi_\lambda^2=\zeta_\lambda^2; \ \ \lambda=1,\dots,k\}
\]
on $\R^{4nk} = \R^{2nk} \times \R^{2nk}$, which is a $2nk$-dimensional vector space since it is the product of $k$ diagonals in $\R^{4n}$. In general, if $V$ is an $m$-dimensional space the cyclic shift on $V^k$ is orientation preserving when $m$ is even. In our case we get that the shift on the negative eigenspace of the quadratic form $- \sum_{\lambda=1}^k (\xi^1_\lambda\zeta^1_\lambda+\xi^2_\lambda\zeta^2_\lambda)$ is orientation preserving. Thus Proposition~\ref{prop_invariance_with_symmetries} and the considerations leading to~ \eqref{isom_consideration_with_symm} together imply~\eqref{inflation_homology}. Our claim is proved.

Next we claim that there is a continuation isomorphism
\begin{equation}\label{second_step_inflation}
\HM_*(\A^+,0)^{\Z_k} \simeq \HM_*(\A_{H,k,N+2},0)^{\Z_k}.
\end{equation}
To see this, consider a smooth function $\beta:[0,1]\to[0,1]$ satisfying:
\begin{itemize}
\item $\beta'\geq0$
\item $\beta(t)=\frac{(N+2)t}{N}$ on $[0,\frac{N-1}{N+2}]$
\item $\beta\equiv 1$ on $[\frac{N}{N+2},1]$.
\end{itemize}
Now consider $\beta_\tau(t)=(1-\tau)t+\tau\beta(t)$ and $H^\tau_t = \beta_\tau'(t)H_{\beta_\tau(t)}$. Note that $H=H^0$ and $\A^+ = \A_{H^1,k,N+2}$. It follows from~\eqref{N_adapted} that there is a well-defined smooth family of discrete action functionals $\A_{H^\tau,k,N+2}$, $\tau\in[0,1]$. Since critical points of $\A_{H^\tau,k,N+1}$ are in 1-1 correspondence with fixed points of the ($\tau$-independent) local diffeomorphism $\varphi^k$, we conclude that the origin in $\R^{2nk(N+2)}$ is a uniformly isolated critical point of the family $\A_{H^\tau,k,N+2}$. Hence~\eqref{second_step_inflation} follows from Proposition~\ref{prop_invariance_with_symmetries} because $\A_{H^\tau,k,N+2}$ is a $\Z_k$-invariant family. This provides the desired isomorphisms between corresponding $\Z_k$-invariant local Morse homologies. For the non-invariant version of this argument one uses Proposition~\ref{prop_invariance_non_invariant} instead of Proposition~\ref{prop_invariance_with_symmetries}.

The proof is complete, but we end by noting that both in the invariant and non-invariant case, the corresponding inflation map is a composition of continuation maps and a direct sum map.
\end{proof}

\begin{remark}
\label{rmk:symmetric x non_symmetric inflation map}
Note that in the symmetric inflation map \eqref{inflation_map_symmetric} it is crucial to take $N+2$ instead of $N+1$ for the non-symmetric map \eqref{inflation_map_non_symmetric}. The point here is that the quadratic form $ \sum_{\lambda=1}^k (\xi^1_\lambda\zeta^1_\lambda+\xi^2_\lambda\zeta^2_\lambda)$ in the right hand side of \eqref{intermediate_action} is defined on $\R^{4nk}$ with negative eigenspace isomorphic to $\R^{2nk}$. This ensures that the cyclic shift is orientation preserving for every $k$ and $n$. If we took $N+1$, the quadratic form in the right hand side of \eqref{intermediate_action} would be defined on $\R^{2nk}$ and given by $\sum_{\lambda=1}^k \xi_\lambda\zeta_\lambda$ (the extra variables would be $\xi_\lambda,\zeta_\lambda$ instead of $\xi^1_\lambda,\zeta^1_\lambda,\xi^2_\lambda,\zeta^2_\lambda$). The negative eigenspace of this quadratic form is isomorphic to $\R^{nk}$ and therefore the cyclic shift would be orientation reversing if $n$ is odd and $k$ is even.
\end{remark}

The inflation isomorphisms from Lemma~\ref{lemma_inflation} allow us to consider the directed system of graded groups $\{\HM_{*+nkN}(\A_{H,k,N},0)\}$ indexed by the positive integers $N\geq 1$ which are adapted to $H$ as in~\eqref{N_adapted}. The homomorphism 
\[
\HM_{*+nkN}(\A_{H,k,N},0) \to \HM_{*+nk(N+j)}(\A_{H,k,(N+j)},0)
\]
is, by definition, the grading preserving isomorphism $\mathscr{I}_{N+j-1} \circ \dots \circ \mathscr{I}_{N}$. In the presence of $\Z_k$-symmetry, the directed system is $\{\HM_{*+nk2N}(\A_{H,k,2N},0)^{\Z_k}\}$ indexed by the integers $N\geq1$ such that $2N$ is adapted to $H$. The homomorphism 
\[
\HM_{*+nk2N}(\A_{H,k,2N},0)^{\Z_k} \to \HM_{*+nk2(N+j)}(\A_{H,k,2(N+j)},0)^{\Z_k}
\]
is now given by $\mathscr{I}_{2(N+j-1)} \circ \dots \circ \mathscr{I}_{2(N+1)} \circ \mathscr{I}_{2N}$ ($j$ factors). 

\begin{definition}[Local invariants]\label{def_local_invariants}
The direct limits
\[
\begin{aligned}
\hloc_*(H,k,0) &= \lim_{N \to \infty} \HM_{*+nkN}(\A_{H,k,N},0) \\
\hloc^{\rm inv}_*(H,k,0) &= \lim_{N \to \infty} \HM_{*+nk2N}(\A_{H,k,2N},0)^{\Z_k}
\end{aligned}
\]
are called the {\it non-invariant} and {\it invariant local homologies of $(H,k,0)$}, respectively, which are always well-defined provided $0$ is an isolated fixed point of $\varphi_H^k$.
\end{definition}

\begin{remark}
Before moving on, we note that the isomorphisms $\mathscr{I}_N$ and $\mathscr{I}_N^{\Z_k}$ are compositions of a direct sum map with two continuation maps, according to the nomenclature established in Section~\ref{ssec_invariant}. For instance, inspecting the proof of Lemma~\ref{lemma_inflation} in the symmetric case, we find a non-degenerate quadratic form $Q$ on $\R^{4nk}$ with $2nk$ negative eigenvalues, and a family of functions $\{\A^+_s\}_{s\in[0,1]}$ having $(0,0) \in \R^{2nk(N+2)} \simeq \R^{2nkN} \times \R^{4nk}$ as a uniformly isolated critical point such that $\A^+_0 = \A_{H,k,N} \oplus Q$, see~\eqref{intermediate_action}, and $\A^+_1=\A_{H^1,k,N+2}$. Here $H^1$ is the final point of a family of Hamiltonians $\{H^\tau\}_{\tau\in[0,1]}$ such that $H^0=H$, the germ $\varphi^k_{H^\tau}$ is independent of $\tau$, and the $\Z_k$-action preserves orientations on the negative space of $Q$. Hence $\mathscr{I}_N^{\Z_k}$ is a composition
\begin{equation}\label{diagram_description_inflation}
\xymatrix{
\HM_*(\A_{H,k,N},0)^{\Z_k} \ar[r] & \HM_{*+2nk}(\A_{H,k,N}\oplus Q,(0,0))^{\Z_k} \ar[d] \\
& \HM_{*+2nk}(\A_{H^1,k,N+2},(0,0))^{\Z_k} \ar[d] \\
& \HM_{*+2nk}(\A_{H,k,N+2},(0,0))^{\Z_k}
}
\end{equation}
where the horizontal arrow is a direct sum map, while the vertical arrows are continuation maps. The non-invariant version of $\mathscr{I}_N$ has an analogous description.
\end{remark}

\subsubsection{Grading}\label{sssec_grading}

As before, $H_t$ is a smooth $1$-periodic family of germs of real-valued functions defined near $0\in\R^{2n}$ such that $dH_t(0)=0$ for all $t$. Assume that $0$ is an isolated fixed point of $\varphi^k_H$. The following important statement can be found in~\cite[Proposition~2.5]{mazz_SDM}.

\begin{proposition}[\cite{RSpath,mazz_SDM}]\label{prop_mazz_grading}
Let the family of germs $K_t$ be uniformly $C^2$-close to $H_t$. Then the Morse index of a critical point of $\A_{K,k,N}$ near $0\in\R^{2nkN}$ is equal to $\CZ+nkN$, where $\CZ$ denotes the Conley-Zehnder index of the corresponding $k$-periodic orbit of $\varphi^t_K$.
\end{proposition}

We also need the following general Morse-theoretical fact.

\begin{lemma}\label{lem_general_Morse_theo_fact}
Let the smooth function $f$ have an isolated critical point $p$ with Morse index $\mu_p$ and nullity $\nu_p$. Fix a relatively compact isolating neighborhood $U$ for $(f,p)$. If $f'$ is $C^2$-close enough to $f$ and all critical points of $f'$ in $U$ are non-degenerate, then all critical points of $f'$ in $U$ have Morse indices in $[\mu_p,\mu_p+\nu_p]$.
\end{lemma}

\begin{proof}[Sketch of proof]
The Hessians at critical points of $f'$ in $U$ are close to the Hessian of $f$ at $p$. Hence they have at least $\mu_p$ negative directions, and no more than $\mu_p+\nu_p$ negative directions.
\end{proof}

\begin{lemma}\label{lemma_gradings}
If $j\not\in [\CZ(H,k),\CZ(H,k)+\nu(H,k)]$ then $\hloc_j(H,k,0)$ and $\hloc^{\rm inv}_j(H,k,0)$ are trivial.
\end{lemma}

\begin{proof}
Let $N$ be adapted to $H$ as in~\eqref{N_adapted}. It is well-known that there exists a $1$-periodic $C^\infty$-small perturbation $K=K_t$ of $H$ such that all $k$-periodic orbits of $\varphi^t_K$ which bifurcate from $0$ are non-degenerate. These $k$-periodic orbits correspond to non-degenerate critical points of $\A_{K,k,N}$ that are close to $0\in\R^{2nkN}$. By Proposition~\ref{prop_mazz_grading} and Lemma~\ref{lemma_action_as_gen_function}, the Morse index and the nullity of $0$ as a critical point of $\A_{H,k,N}$ are equal to $\CZ(H,k)+nkN$ and $\nu(H,k)$, respectively. Lemma~\ref{lem_general_Morse_theo_fact} implies that critical points of $\A_{K,k,N}$ close to the origin have Morse indices in
\[
[\CZ(H,k)+nkN,\CZ(H,k)+\nu(H,k)+nkN].
\]
Hence $\hloc_j(H,k,0)=0$ if $j \not\in [\CZ(H,k),\CZ(H,k)+\nu(H,k)]$. Since $\hloc^{\rm inv}$ is a quotient of $\hloc$ the same conclusion must hold for $\hloc^{\rm inv}_*(H,k,0)$.
\end{proof}

\begin{lemma}\label{lemma_grading_tot_deg}
If $\hloc_{\Delta_\CZ(H,k)\pm n}(H,k,0)\neq 0$ then $0$ is a totally degenerate fixed point of $\varphi^k_H$. The same holds replacing $\hloc$ by $\hloc^{\rm inv}$.
\end{lemma}

\begin{proof}
We only need to prove the lemma for $\hloc$. As in the proof of Lemma~\ref{lemma_gradings}, consider $1$-periodic and $C^\infty$-small perturbations $K$ of $H$ such that the $k$-periodic orbits of $\varphi^t_K$ which bifurcate from $0$ are non-degenerate. If $\hloc_{\Delta_\CZ(H,k)+n}(H,k,0)\neq 0$ then for any such $K$, and large $N$, we find at least one critical point $z$ of $\A_{K,k,N}$ near $0\in\R^{2nkN}$ with Morse index equal to $\Delta_{\CZ}(H,K)+n+nkN$. Such a critical point corresponds to a $k$-periodic orbit of $\varphi^t_K$ whose Conley-Zehnder index we denote by $\CZ(z)$. By Proposition~\ref{prop_mazz_grading} the Morse index of $z$ is $\CZ(z)+nkN$. Thus $\CZ(z)=\Delta_{}(H,k)+n$. Now results from~\cite{SZ} imply that $0$ is a totally degenerate fixed point of $\varphi^k_H$. The case $\hloc_{\Delta_\CZ(H,k)-n}(H,k,0)\neq 0$ is identical.
\end{proof}

\subsubsection{Good and admissible iterations}\label{sssec_good_adm}

The notion of good and bad iterations is not only defined for periodic orbits, but also for symplectic matrices. We start this section with the following

\begin{remark}
Let $\ell$ be the number of eigenvalues of $M\in Sp(2n)$ in $(-1,0)$. If $\ell$ is even then all $k\in\N$ are {\it good} iterations of $M$. If $\ell$ is odd then only the odd iterations are good, the even ones are not and will be called {\it bad}.
\end{remark}

The following material is rather standard and included here for completeness as these results play a crucial role in this paper.

\begin{lemma}\label{lem_good_ite_paths}
Let the continuous path $\phi:\R\to Sp(2n)$ satisfy $\phi(t+1)=\phi(t)\phi(1)$, $\phi(0)=I$. If $k$ is a good and admissible iteration for $\phi(1)$ then $\CZ(t\in[0,k]\mapsto \phi(t))$ and $\CZ(t\in[0,1]\mapsto \phi(t))$ have the same parity.
\end{lemma}

\begin{proof}[Sketch of the proof]
This is well-known if $\det \phi(1)-I\neq 0$. If $\det \phi(1)-I= 0$ then the proof follows from the fact that there is a small perturbation $\tilde\phi$ of $\phi$, still satisfying $\tilde\phi(0)=I$, $\tilde\phi(t+1)=\tilde\phi(t)\tilde\phi(1)$, such that $\det( \tilde\phi(1)-I)\neq 0$, $k$ is good and admissible for $\tilde\phi(1)$ and
\[
\begin{aligned}
& \CZ(t\in[0,k]\mapsto \tilde\phi(t)) = \CZ(t\in[0,k]\mapsto \phi(t)) \\
& \CZ(t\in[0,1]\mapsto \tilde\phi(t)) = \CZ(t\in[0,1]\mapsto \phi(t)).
\end{aligned}
\]
The conclusion follows from the non-degenerate case applied to $\tilde\phi$.
\end{proof}

Now we can consider the case of periodic orbits.
Let $H_t$ be a $1$-periodic germ of Hamiltonian near $0\in\R^{2n}$ satisfying $dH_t(0)=0$.

\begin{definition}
Fix $m\in\N$. The number $k\in\N$ is an {\it admissible iteration} for $\varphi_H^m$ if it is admissible for $d\varphi_H^m(0)$. It is a {\it good iteration} for $\varphi_H^m$ if it is a good iteration for $d\varphi_H^m(0)$.
\end{definition}

\begin{lemma}\label{lemma_preserves_orient_Delta_perp}
Assume that $k$ is a good and admissible iteration for $\varphi_H^m$. Let $\Delta$ denote the $k$-diagonal in $\R^{2nmk2N}= (\R^{2nm2N})^k$. Then $D^2\A_{H,mk,2N}(0)$, with $0\in\R^{2nmk2N}$, preserves the splitting $\R^{2nmk2N}=\Delta\oplus \Delta^\bot$. Let $E_-\subset \Delta^\bot$ be the negative eigenspace of $D^2\A_{H,mk,2N}(0)|_{\Delta^\bot}$. Then the $\Z_{km}$-action on $\R^{2nmk2N}=(\R^{2n2N})^{mk}$ (by cyclic shift to the right) preserves orientations on $E_-$.
\end{lemma}

\begin{proof}
$D^2\A_{H,mk,2N}(0)$, $0\in\R^{2nmk2N}$, splits as a quadratic form into $Q\oplus Q^\bot$ according to the orthogonal splitting $\R^{2nmk2N} = \Delta\oplus \Delta^\bot$. This, as usual, just follows from the fact that the Euclidean gradient of $\A_{H,mk,2N}$ is tangent to $\Delta$ at points of $\Delta$. It remains to show that the $\Z_{km}$-action preserves orientations on the negative eigenspace $E_-$ of $Q^\bot$. Obviously $E_-$ is $\Z_{km}$-invariant because $\A_{H,km,2N}$ is $\Z_{km}$-invariant. 
Note that the $\Z_{km}$-action fixes no non-zero vector in $\Delta^\bot$. This follows from the fact that the fixed point set of the $\Z_k$-action induced by $m\in\Z_{km}$ is precisely $\Delta$. Hence, the $\Z_{km}$-action fixes no non-zero vector in $E_-$. The proof will be finished if we can show that $\dim E_-$ is even.

Let $\mu_{km}$ be the Morse index of $0\in\R^{2nmk2N}$ as a critical point of $\A_{H,km,2N}$, and $\mu_m$ be the Morse index of $0\in\R^{2nm2N}$ as a critical point of $\A_{H,m,2N}$. Let $\CZ^{(km)}$ be the Conley-Zehnder index of the path $t\in[0,km]\mapsto d\varphi^t_{H}(0)$, and $\CZ^{(m)}$ be the Conley-Zehnder index of the path $t\in[0,m]\mapsto d\varphi^t_{H}(0)$. Note that these paths might be degenerate, and we take the lower semi-continuous extension of the Conley-Zehnder index. The diagonal inclusion $\R^{2nm2N} \to\Delta$ identifies the negative eigenspace of $D^2\A_{H,m,2N}(0)$ at $0\in\R^{2nm2N}$ with the negative eigenspace of $Q$. Using~\cite[Proposition~2.5]{mazz_SDM} we get
\[
\dim E_- = \mu_{km} - \mu_m = \CZ^{(km)} - \CZ^{(m)} + nm(k-1)2N
\]
which is even precisely when $\CZ^{(km)} - \CZ^{(m)}$ is even. But the latter is even because $k$ is good and admissible (Lemma~\ref{lem_good_ite_paths}).
\end{proof}

\begin{remark}\label{rmk_dim_E_-}
We point out that the proof above reveals the following formula for $\dim E_-$:
\begin{equation*}
\dim E_- = \CZ^{(km)} - \CZ^{(m)} + nm(k-1)2N
\end{equation*}
where $\CZ^{(km)}$ is the Conley-Zehnder index of the path $t\in[0,mk]\mapsto d\varphi^t_{H}(0)$, and $\CZ^{(m)}$ is the Conley-Zehnder index of the path $t\in[0,m]\mapsto d\varphi^t_{H}(0)$. 
\end{remark}

\subsubsection{Effect of changing the Hamiltonian isotopy}\label{sssec_isotopy}

Let $H^s_t$, $s\in[0,1]$, be a smooth family of $1$-periodic Hamiltonians defined near $0\in\R^{2n}$ satisfying $dH^s_t(0)=0$ for all $s,t$. We will say that $N$ is {\it adapted to the family $H^s$} if
\begin{equation}\label{N_adapted_family}
\begin{aligned}
& s\in[0,1], \ t_0<t_1, \ t_1-t_0 \leq (2N)^{-1} \\ 
&\Rightarrow \ \varphi_{H^s}^{t_1} \circ (\varphi_{H^s}^{t_0})^{-1} \ \text{satisfies (Gen1).}
\end{aligned}
\end{equation}

\begin{lemma}\label{lemma_isomorphism_deformation}
Let $k\in\N$ and assume that $\varphi_{H^s}^k$ 
has $0$ as a uniformly (in $s$) isolated fixed point. Then there are isomorphisms
\begin{equation*}
\begin{aligned}
\hloc_*(H^0,k,0) &\to \hloc_*(H^1,k,0) \\
\hloc_*^{\rm inv}(H^0,k,0) &\to \hloc_*^{\rm inv}(H^1,k,0).
\end{aligned}
\end{equation*}
\end{lemma}

\begin{proof}
Let $N$ be adapted to the family $H^s$ as in~\eqref{N_adapted_family}. We claim that there are continuation isomorphisms 
\begin{equation}\label{continuation_isom_pre}
\begin{aligned}
\Theta : \HM_*(\A_{H^0,k,N},0) &\to \HM_*(\A_{H^1,k,N},0) \\
\Theta^{\Z_k} : \HM_*(\A_{H^0,k,2N},0)^{\Z_k} &\to \HM_*(\A_{H^1,k,2N},0)^{\Z_k}.
\end{aligned}
\end{equation}
We only work out $\Theta$, the map $\Theta^{\Z_k}$ is handled in a similar way.
Denote $\Phi_s=\Phi_{H^s,k,N}$ and $\A_s = \A_{H^s,k,N}$. By Lemma~\ref{lemma_action_as_gen_function}, (Gen1) holds for all $\Phi_s$ and $\A_s$ is generating function for the germ $\Phi_s$ in the sense of (Gen2). Our assumptions imply that the origin in $\R^{2nkN}$ is a uniformly isolated fixed point for $\Phi_s$ and consequently a uniformly isolated critical point for $\A_s$. Hence, by Proposition~\ref{prop_invariance_non_invariant}, we have $\HM_*(\A_0,0) \simeq \HM_*(\A_1,0)$ via a continuation map $\Theta$.

To conclude the proof in the non-invariant case we need to show that the first isomorphism~\eqref{continuation_isom_pre} makes the diagram
\[
\xymatrixcolsep{6pc}\xymatrix{
\HM_{*}(\A_{H^0,k,N},0) \ar[d]^{\Theta} \ar[r]^{\mathscr{I}_N} &\HM_{*+nk}(\A_{H^0,k,N+1},0) \ar[d]^{\Theta} \\
\HM_{*}(\A_{H^1,k,N},0) \ar[r]^{\mathscr{I}_N} & \HM_{*+nk}(\A_{H^1,k,N+1},0)
}
\]
commutative, where $\mathscr{I}_N$ are inflation maps given by Lemma~\ref{lemma_inflation}. This follows from the fact that $\mathscr{I}_N$ is a composition of direct sum maps and continuation maps as described in~\eqref{diagram_description_inflation}. These two kinds of isomorphisms commute with continuation maps, as explained in diagram~\eqref{diagram_commutes_continuation_1}.

The version of this argument with symmetries follows from applying Proposition~\ref{prop_invariance_with_symmetries} instead of Proposition~\ref{prop_invariance_non_invariant}, noting that the corresponding family of action functionals preserve $\Z_k$-symmetry, and the version of diagram~\eqref{diagram_description_inflation} with symmetries combined with diagram~\eqref{diagram_commutes_continuation_2}.
\end{proof}

Suppose now that we are given two $1$-periodic Hamiltonians $H_t,K_t$ near the origin in $\R^{2n}$, satisfying $dH_t(0)=dK_t(0)=0$ for all $t$, and assume that the time-$1$ germs coincide
\[
\varphi_H^1 = \varphi = \varphi_K^1.
\]
Then we find $G_t$ such that $\varphi_G^t \circ \varphi_H^t = \varphi_K^t$ and $\varphi_G^1=id$, i.e. $\{\varphi_G^t\}_{t\in[0,1]}$ is a loop of germs based at the identity. Normalizing Hamiltonians to vanish at $0$ we find $K = G\# H$ where 
\[
(G\#H)_t = G_t + H_t \circ (\varphi_G^t)^{-1}.
\]

\begin{lemma}\label{lemma_maslov_loop_deformation}
Fix $k\in\N$ and suppose that the Maslov index of $t\in[0,k] \mapsto d\varphi_G^t(0)$ vanishes. Then there exists a smooth family  $G^\tau$, $\tau\in[0,1]$, of $1$-periodic Hamiltonians satisfying $dG_\tau^t(0)=0$ for all $(\tau,t)$, $G^0=G$, $G^1=0$ and $\varphi_{G^\tau}^1=id$ for all~$\tau$.
\end{lemma}

\begin{proof}
For $\tau\in[0,1/2]$ consider the family of germs $\psi_{\tau,t}$ given by
\[
\psi_{\tau,t}(z) = 
\left\{ 
\begin{aligned} 
& \frac{1}{1-2\tau}\varphi_G^t((1-2\tau)z)  \ \ \ \text{if $\tau\in[0,1/2)$} \\ 
& d\varphi_G^t(0) \cdot z \ \ \ \text{if $\tau=1/2$}.
\end{aligned} 
\right.
\]
Note that $\psi_{\tau,t}$ is $1$-periodic in $t$ because so is $\varphi_G^t$. 
If $m$ is the Maslov index of $t\in[0,1] \mapsto d\varphi_{G}^t(0)$ then $km=0 \Rightarrow m=0$, by assumption. Hence we can continue $\psi_{\tau,t}$ to all $\tau\in[0,1]$ keeping $1$-periodicity in $t$ in such a way that $\psi_{1,t} = id$ for all $t$. Since we work locally, we can find smooth family of $1$-periodic (in $t$) Hamiltonians $G^\tau_t$ such that $dG^\tau_t(0)=0$, $\varphi_{G^\tau}^t = \psi_{\tau,t}$ and $G^1=0$, as desired.
\end{proof}

\begin{lemma}\label{lemma_change_of_isotopy}
Suppose that $0$ is an isolated fixed point of $\varphi^k$, for some $k\in\N$. Then 
\[
\begin{aligned}
\hloc_*(K,k,0) &\simeq \hloc_{*-2m}(H,k,0) \\
\hloc_*^{\rm inv}(K,k,0) &\simeq \hloc_{*-2m}^{\rm inv}(H,k,0)
\end{aligned}
\]
where $m$ is the Maslov index of the loop $t\in\R/k\Z \mapsto d\varphi_G^t(0)$ and the Hamiltonians $H,K$ and $G$ are related via $K=G\# H$.
\end{lemma}

\begin{proof}
We only work out the invariant case. The non-invariant case is simpler since it does not make use of the notion of good iterations; see the end of this proof.

We claim that there are isomorphisms
\begin{equation}\label{before_inflation}
\HM_*(\A_{K,k,2N},0)^{\Z_k} \simeq \HM_{*-2m}(\A_{H,k,2N},0)^{\Z_k}
\end{equation}
provided $N$ is large enough, and these commute with inflation maps. First consider $N\in \N$ large enough such that $2N$ is adapted to both $H$ and $K$ in the sense of~\eqref{N_adapted}.

Let us first assume that the spectrum of the symplectic matrix $d\varphi_H^1(0)=d\varphi_K^1(0)$ consists of eigenvalues of the form $e^{i2\pi\alpha}$, $\alpha\not\in\Q$. By Lemma~\ref{lemma_action_as_gen_function}, the origin in $\R^{2nk2N}$ is a non-degenerate critical point of both $\A_{H,k,2N}$ and $\A_{K,k,2N}$. According to~\cite[Theorem 4.1]{RSpath}, see also~\cite[Proposition 2.5]{mazz_SDM}, the Morse index of the origin as a critical point of $\A_{H,k,2N}$ is
\[
\CZ\left(t\in[0,k] \mapsto d\varphi_H^t(0)\right) + nk2N
\]
while the Morse index of the origin as a critical point of $\A_{K,k,2N}$ is
\[
\CZ\left(t\in[0,k] \mapsto d\varphi_K^t(0)\right) + nk2N
\]
where $\CZ$ denotes the Conley-Zehnder index of a path in $Sp(2n)$ starting at the identity and ending away from the Maslov cycle. By the properties of the Conley-Zehnder index we have
\[
\CZ\left(t\in[0,k] \mapsto d\varphi_K^t(0)\right) = \CZ\left(t\in[0,k] \mapsto d\varphi_H^t(0)\right) + 2m
\]
which would immediately give the desired conclusion in this case when group symmetries are not present. For the symmetric case we need to argue a bit more, because we still need to show that the $\Z_k$ actions on the negative eigenspaces of the Hessians of $\A_{H,k,2N}$ and of $\A_{K,k,2N}$ at the origin both preserve orientations. But this follows as a consequence of Lemma~\ref{lemma_preserves_orient_Delta_perp}.

To handle the general case, consider a $1$-periodic Hamiltonian $h_t$ defined near $0\in\R^2$ satisfying $dh_t(0)=0 $ for all $t$, such that $\varphi_h^1$ is an irrational rotation. In particular, $0\in\R^2$ is totally non-degenerate and all iterations are good and admissible. The integer $m$ is divisible by $k$ because $m/k$ is the Maslov index of the loop $t\in[0,1] \mapsto d\varphi^t_G(0)$ (recall that $\varphi_G^1=id$). Let $g_t$ be a $1$-periodic Hamiltonian, again defined near $0\in\R^2$, satisfying $dg_t(0)=0 $ for all $t$, $\varphi_g^1=id$ and the Maslov index of $t\in[0,1]\mapsto d\varphi_g^t(0)$ is $m/k$. Then the $1$-periodic Hamiltonians defined near $(0,0) \in \R^2 \oplus \R^{2n}$ as
\[
\begin{array}{ccc}
(g\#h) \oplus H & \text{ and } & h \oplus (G\#H) = h\oplus K
\end{array}
\]
have the same time-$1$ germ. It follows from the properties of the Maslov index that the Maslov index of the loop
\[
\begin{aligned}
t \in [0,1] &\mapsto d\varphi_{h\oplus(G\#H)}^t(0) \ \left( d\varphi_{(g\#h)\oplus H}^t(0)\right)^{-1} \\
&= [d\varphi_{h}^t(0) d\varphi_{h}^t(0)^{-1} d\varphi_{g}^t(0)^{-1}] \oplus [d\varphi_{G}^t(0) d\varphi_{H}^t(0) d\varphi_{H}^t(0)^{-1}] \\
&= d\varphi_g^t(0)^{-1} \oplus d\varphi_G^t(0)
\end{aligned}
\]
is equal to $(-m/k)+(m/k)=0$. By Lemma~\ref{lemma_maslov_loop_deformation}, there exists a family $\{\Lambda^\tau_t\}_{\tau\in[0,1]}$ of $1$-periodic Hamiltonians defined near $(0,0) \in \R^2 \oplus \R^{2n}$ such that $d\Lambda^\tau_t(0,0)=0$ for all $(\tau,t)$, satisfying $\Lambda^0 = h\oplus (G\#H)$ and $\Lambda^1 = (g\#h)\oplus H$ and, moreover, such that the family of germs $\varphi_{\Lambda^\tau}^1$ is independent of~$\tau$.

Increasing $N$, we can assume that it is adapted to the family $\Lambda^\tau$ as in~\eqref{N_adapted_family}. 
Proposition~\ref{prop_invariance_with_symmetries} implies that there is a continuation isomorphism
\[
\HM_*(\A_{\Lambda^0,k,2N},0)^{\Z_k} \simeq \HM_*(\A_{\Lambda^1,k,2N},0)^{\Z_k}
\]
where $0$ denotes the origin in $\R^{2(n+1)k2N}$. But $\Lambda^0,\Lambda^1$ are direct sums of Hamiltonians, from where it follows that
\[
\begin{aligned}
& \A_{\Lambda^0,k,2N} = \A_{h,k,2N} \oplus \A_{K,k,2N} \\
& \A_{\Lambda^1,k,2N} = \A_{g\#h,k,2N} \oplus \A_{H,k,2N}.
\end{aligned}
\]
The summand $\A_{h,k,2N}$ has the origin in $\R^{2k2N}$ as a non-degenerate critical point whose Morse index we denote by $i$. Then the same point is a non-degenerate critical point of $\A_{g\#h,k,2N}$ of Morse index $i+2m$. This is seen as in the first part of the proof for the totally non-degenerate case; it follows as a consequence of~\cite[Theorem~4.1]{RSpath}. We can finally compute, using direct sum maps and continuation maps as explained in Section~\ref{ssec_invariant},
\[
\begin{aligned}
\HM_*(\A_{K,k,2N},0)^{\Z_k} 
&\simeq \HM_{*+i}(\A_{h,k,2N}\oplus \A_{K,k,2N},(0,0))^{\Z_k} \\
&= \HM_{*+i}(\A_{\Lambda^0,k,2N},(0,0))^{\Z_k} \\
&\simeq \HM_{*+i}(\A_{\Lambda^1,k,2N},(0,0))^{\Z_k} \\
&= \HM_{*+i}(\A_{g\#h,k,2N} \oplus \A_{H,k,2N},(0,0))^{\Z_k} \\
&\simeq \HM_{*+i-(i+2m)}(\A_{H,k,2N},0)^{\Z_k} \\
&= \HM_{*-2m}(\A_{H,k,2N},0)^{\Z_k}
\end{aligned}
\]
where $0$ in each line denotes the origin in the appropriate space. This proves~\eqref{before_inflation}. Note that Lemma~\ref{lemma_preserves_orient_Delta_perp} and the properties of $h$ were strongly used to conclude that we can actually use the $\Z_k$-invariant versions of direct sum maps.

To conclude the proof we only need to show that the above chain of isomorphisms commutes with inflation maps, so we can take direct limits. This follows from diagram~\eqref{diagram_commutes_continuation_2} and complementary analogous diagrams stating that two kinds of direct sum maps commute with each other.
\end{proof}

\section{Iteration map and Persistence Theorems}\label{sec_shifting_lemma}

We outline the content of this section. In Section~\ref{ssec_equiv_GM_lemma} we prove an equivariant version of the celebrated Gromoll-Meyer splitting lemma~\cite{GM} which is crucial to our analysis of finite cyclic actions. In Section~\ref{ssec_persistence_no_symmetries} we prove the Persistence Theorem without symmetries, which is our discrete version of that of Ginzburg-G\"urel~\cite{GG} for local Floer homology. In Section~\ref{ssec_persistence_symmetries} we prove the invariant version of the Persistence Theorem which is the discrete version of a conjectural persistence theorem for local contact homology.

\subsection{Equivariant Gromoll-Meyer splitting lemma, and refinements}\label{ssec_equiv_GM_lemma}

First we review the proof of the classical splitting lemma when no group symmetries are present.

\begin{lemma}[Gromoll-Meyer Splitting Lemma]\label{lem_GM_non_equiv}
Let $f: \R^n = \R^{n_1} \times \R^{n_2} \to \R$ be a smooth function such that $(0,0)$ is an isolated critical point. Assume that $D^2f(0,0)$ preserves the splitting, and that $\ker D^2f(0,0) \subset \R^{n_1}\times\{0\}$. Then there is a neighborhood $U$ of $(0,0)$ and an embedding $\Psi: U \to \R^n$ fixing $(0,0)$ such that
$D\Psi(0,0)=I$ and
\[
f\circ \Psi(z_1,z_2) = g(z_1)+h(z_2),
\]
where $0\in\R^{n_2}$ is a non-degenerate critical point of $h$. In particular, we must have $D^2g(0)=D^2f(0,0)|_{\R^{n_1}\times \{0\}}$ and $D^2h(0) = D^2f(0,0)|_{\{0\}\times \R^{n_2}}$. Moreover, if $\nabla f$ is tangent to $\R^{n_1}\times\{0\}$ at points of $\R^{n_1}\times\{0\}$ then $\Psi$ can be arranged so that
\[
\begin{array}{ccc} 
g(z_1)=f(z_1,0) & \text{and} & h(z_2) = \frac{1}{2} \left< D^2f(0,0)\cdot (0,z_2),(0,z_2) \right>
\end{array}
\]
and $\Psi(z_1,0)=(z_1,0) $ for all $(z_1,0)\in U$.
\end{lemma}

\begin{remark}
One difference with the splitting lemma from~\cite{GM} is that we do not assume $\R^{n_1}\times\{0\}$ to be equal to the kernel of the Hessian, but only that it contains the kernel. Then, of course, the Hessian of $g$ at $0\in\R^{n_1}$ might not vanish.
\end{remark}

\begin{proof}
Before handling the general statement we prove a \\

\noindent {\bf Preliminary step.} There is a smooth family of embeddings $\Upsilon_s: V \to \R^n$, $s \in [0,1]$, defined on some neighborhood $V$ of $(0,0)$ satisfying $\Upsilon_0 = id$, $D\Upsilon_s(0,0)=I $ for all $s$, and such that for every $z_1$ close enough to $0\in\R^{n_1}$ the function $z_2 \mapsto f\circ\Upsilon_1(z_1,z_2)$ has a non-degenerate critical point at $0\in\R^{n_2}$.  \\

\noindent {\it Proof of the preliminary step.} Since $\ker D^2f(0,0) \subset \R^{n_1}\times\{0\}$ and $D^2f(0,0)$ respects the splitting $\R^{n_1} \times \R^{n_2}$, we know that $D_{22}f(0,0)$ is non-singular. Thus the equation $D_2f=0$ defines implicitly a smooth $\R^{n_2}$-valued function $\phi$ on a neighborhood of $0\in\R^{n_1}$ with the following property: a point $(z_1,z_2)$ near $(0,0)$ satisfies $D_2f(z_1,z_2)=0$ if, and only if, $z_2=\phi(z_1)$. Hence $\phi(0)=0$. Differentiating $D_2f(z_1,\phi(z_1))=0$ at $0\in\R^{n_1}$ we get $D_{21}f(0,0) + D_{22}f(0,0)D\phi(0)=0$. Since $D^2f(0,0)$ splits we have $D_{21}f(0,0)=0$, so $D\phi(0)=0$. This fact implies that we can take $\Upsilon_s(z_1,z_2) = (z_1,z_2+s\phi(z_1))$. \\

\noindent {\it Remark on the proof of the preliminary step.} If $\nabla f$ is tangent to $\R^{n_1}\times\{0\}$ at points of $\R^{n_1}\times\{0\}$, in other words $D_2f(z_1,0)\equiv0$ for $z_1\in\R^{n_1}$ close to $0$, then the above argument gives the trivial family $\Upsilon_s\equiv id$. \\

After the above preliminary step there is no loss of generality to assume, in addition to the hypothesis of the lemma, that for all $z_1$ close enough to $0\in \R^{n_1}$ the point $0\in\R^{n_2}$ is a non-degenerate critical point of $z_2 \mapsto f(z_1,z_2)$. In fact, this can be done after replacing $f$ by $f\circ\Upsilon_1$. And, as just observed, this leaves $f$ unchanged in the case $\nabla f$ is tangent to $\R^{n_1}\times\{0\}$.

The rest of the argument follows Gromoll and Meyer in \cite{GM} closely. For $(z_1,z_2)$ close to $(0,0)$ we find, using Taylor's formula, a smooth symmetric $n_2\times n_2$ matrix $H(z_1,z_2)$ such that $f(z_1,z_2) = g(z_1) + \left<H(z_1,z_2)z_2,z_2\right>$ where $g(z_1)=f(z_1,0)$ and $H$ satisfies $2H(z_1,0) = D_{22}f(z_1,0)$.

Set $B(z_1,z_2)=H(z_1,z_2)^{-1}H(0,0)$, which is also smooth. Then using the symmetry of $H$ we compute
\begin{equation*}
\begin{aligned}
B(z_1,z_2)^TH(z_1,z_2) &= H(0,0)^T(H(z_1,z_2)^{-1})^TH(z_1,z_2) \\
&= H(0,0) \\ &= H(z_1,z_2)B(z_1,z_2).
\end{aligned}
\end{equation*}
Let $\{c_k\}_{k\geq0}$ satisfy $\sqrt{1+x}=\sum_{k=0}^\infty c_kx^k$ for $x\sim 0$. Since $B(0,0)$ is the identity matrix of order $n_2$, on a neighborhood of $(0,0)$ the power series
\begin{equation*}
C(z_1,z_2) = \sum_{k=0}^\infty c_k(B(z_1,z_2)-I)^k
\end{equation*}
converges uniformly, together with all its derivatives, to a smooth function $C(z_1,z_2)$ satisfying $C^2=B$. Note that $C(0,0)=I$.  As proved above $B^TH=HB$, so the same holds for every polynomial in $B$, and hence also for $C$ since it is a uniform limit of polynomials in $B$. 

Setting $C_s = (1-s)I+sC$ we have $C_s(0,0)=I$ for all $s$, and $C_s(z_1,z_2)$ is invertible for all $s\in[0,1]$ and $(z_1,z_2)$ on a fixed small neighborhood of $(0,0)$. Now define
\begin{equation}
\Phi_s(z_1,z_2)=(z_1,C_s(z_1,z_2)^{-1}z_2).
\end{equation}
Here, as above, $s\in[0,1]$ and $(z_1,z_2)$ lies on a small neighborhood of $(0,0)$. For every $s$, $D\Phi_s(0,0)$ is the identity, so by the implicit function theorem there is a neighborhood $U$ of $(0,0)$ and a smooth family of embeddings $\Phi_s^{-1}:U \to \R^n$ which invert $\Phi_s$ near $(0,0)$. Finally define $\Psi_s := \Phi_s^{-1}$ and $\Psi := \Psi_1$. We claim that $\Psi$ is our desired embedding. Indeed,
\begin{equation*}
f_1(\zeta_1,\zeta_2) := f \circ \Phi_1^{-1}(\zeta_1,\zeta_2) = g(z_1) + \left<H(z_1,z_2)z_2,z_2\right>
\end{equation*}
where $\Phi_1(z_1,z_2)=(\zeta_1,\zeta_2)$. This means that $\zeta_1=z_1$ and $z_2=C(z_1,z_2)\zeta_2$, so substituting we get
\begin{equation}
\begin{aligned}
f_1(\zeta_1,\zeta_2) &= g(\zeta_1) + \left<H(z_1,z_2)C(z_1,z_2)\zeta_2,C(z_1,z_2)\zeta_2\right> \\
&= g(\zeta_1) + \left<C(z_1,z_2)^TH(z_1,z_2)C(z_1,z_2)\zeta_2,\zeta_2\right> \\
&= g(\zeta_1) + \left<H(z_1,z_2)C(z_1,z_2)^2\zeta_2,\zeta_2\right> \\
&= g(\zeta_1) + \left<H(z_1,z_2)B(z_1,z_2)\zeta_2,\zeta_2\right> \\
&= g(\zeta_1) + \left<H(0,0)\zeta_2,\zeta_2\right>
\end{aligned}
\end{equation}
as claimed.
\end{proof}

The next result shows that the embedding $\Psi$ can be chosen equivariant with respect to a linear isometric $\Z/k\Z$-action if the function $f$ is $\Z/k\Z$-invariant. This fact is well known to experts but, since it will be crucial in this work, we provide a detailed proof.

\begin{lemma}[Invariant version of the splitting lemma]\label{lemma_invariant_GM_splitting}
In the same setting of Lemma~\ref{lem_GM_non_equiv}, let $A = (A_1,A_2) \in O(n_1)\times O(n_2) \subset O(n_1+n_2)$ be an Euclidean isometry of $\R^{n_1} \times \R^{n_2} \simeq \R^{n_1+n_2}$, and assume that $f$ is $A$-invariant. Then all the conclusions of Lemma~\ref{lem_GM_non_equiv} hold with an embedding $\Psi$ which commutes with $A$.
\end{lemma}

\begin{proof}
We show that in each step of the above proof we can obtain $A$-equivariance. First we check the preliminary step: The splitting is respected by the isomorphism $A$ by assumption and we write $A(z_1, z_2)=(A_1z_1, A_2z_2)$. From $(Df\circ A)A=Df$ evaluated at the point $(z_1,\phi(z_1))$ we get
\[
D_2f(A_1z_1,A_2\phi(z_1))A_2 = D_2f(z_1,\phi(z_1)) = 0.
\]
By the uniqueness statement of the implicit function theorem we conclude that 
\begin{equation}\label{equiv_phi}
A_2\phi(z_1) = \phi(A_1z_1)
\end{equation}
holds for all $z_1$ close enough to $0 \in \R^{n_1}$. This is equivalent to the embeddings $(z_1,z_2) \mapsto (z_1,z_2+s\phi(z_1))$ being $A$-equivariant, and we are done showing that the {\it preliminary step} in the proof of Lemma~\ref{lem_GM_non_equiv} can be done equivariantly. This means, as before, that we can assume $D_2f(z_1,0)\equiv0$ for $z_1$ near $0\in\R^{n_1}$.

It is easy to see that in the formula $f(z_1,z_2) = g(z_1) + \left<H(z_1,z_2)z_2,z_2\right>$, the function $g$ is $A_1$-invariant. From this it follows that the second term must be $A$-invariant. The explicit form of $H(z_1,z_2)$ given by
\[
H(z_1,z_2) = \int_0^1\int_0^1 \tau D_{22}f(z_1,\lambda\tau z_2) \, d\lambda \, d\tau
\]
shows that, in fact, 
\[
A_2^TH(A_1z_1,A_2z_2)A_2 = A_2^{-1}H(A_1z_1,A_2z_2)A_2 = H(z_1,z_2).
\]
This follows simply from $A_2^T(D_{22}f\circ A)A_2 = D_{22}f$, which in turn follows from $A^T(D^2f\circ A)A=D^2f$. The same property holds for $B$ as can be seen by the following computation:
\[
\begin{aligned}
A_2^{-1}B(A_1z_1,A_2z_2)A_2 &= A_2^{-1}H(A_1z_1,A_2z_2)^{-1}H(0,0)A_2 \\
&= A_2^{-1}H(A_1z_1,A_2z_2)^{-1}A_2H(0,0) \\
&= (A_2^{-1}H(A_1z_1,A_2z_2)A_2)^{-1}H(0,0) \\
&= H(z_1,z_2)^{-1}H(0,0) \\
&= B(z_1,z_2).
\end{aligned}
\]
By construction, the same properties also hold for $C$ and $C_s$.

Now consider the function $\Phi_s(z_1,z_2)=(z_1,C_s(z_1,z_2)^{-1}z_2)$. This function is $A$-equivariant, since
\begin{equation}
\begin{aligned}
\Phi_s(A_1z_1,A_2z_2) &= (A_1z_1,C_s(A_1z_1,A_2z_2)^{-1}A_2z_2) \\
&= (A_1z_1,A_2C_s(z_1,z_2)^{-1}z_2) \\
&= A(z_1,C_s(z_1,z_2)^{-1}z_2) \\
&= A\Phi_s(z_1,z_2).
\end{aligned}
\end{equation}
Thus $\Psi_s := \Phi_s^{-1}$ is also $A$-equivariant.
\end{proof}


\begin{corollary}\label{lemma_local_MH_tangent_submfd}
Let $f$ be a smooth real-valued function defined near $0\in\R^n$ with an isolated critical point at $0$, and let $\Delta\subset\R^n$ be a linear subspace. Using the Euclidean metric to compute gradients, suppose that $p\in \Delta\Rightarrow \nabla f(p)\in \Delta$ and that $D^2f(0)|_{\Delta^\bot}$ is non-degenerate. Let us write $(x,y)$ for Euclidean coordinates with respect to the orthogonal splitting $\R^n = \Delta\oplus \Delta^\bot$. Then there exists a local diffeomorphism $\Phi$ fixing points in $\Delta$ near the origin such that 
\begin{equation}\label{quite_explict}
f\circ \Phi(x,y) = f(x,0) + \frac{1}{2} \left< D^2f(0,0)\cdot(0,y),(0,y) \right>
\end{equation}
for all $(x,y)$ close enough to zero. In particular, if $D^2f(0,0)|_{\Delta^\bot}$ has signature $(p,q)$ then there is an isomorphism
\[
\HM_*(f|_\Delta,0) \stackrel{\sim}{\to} \HM_{*+q}(f,0)
\]
of local Morse homologies. Moreover, if $A\in O(n)$ leaves $\Delta$ and $f$ invariant and generates a $\Z_\ell$-action, $\ell\in\N$, then $\Phi$ can be taken $\Z_\ell$-equivariant. Clearly $A$ leaves the negative space $E_-$ of $D^2f(0,0)|_{\Delta^\bot}$ invariant, and if $A|_{E_-}$ preserves orientations then there is an isomorphism of invariant local Morse homologies
\[
\HM_{*}(f|_\Delta,0)^{\Z_\ell} \stackrel{\sim}{\to} \HM_{*+q}(f,0)^{\Z_\ell}.
\]
\end{corollary}

\begin{remark}\label{rmk_explicit}
The isomorphisms of local Morse homologies in Corollary~\ref{lemma_local_MH_tangent_submfd}, non-invariant or invariant, are given by composing direct sum isomorphisms~\eqref{isom_consideration_no_symm}-\eqref{isom_consideration_with_symm} with the isomorphisms induced by changing $f\circ \Phi$ to $f$. Using the quite explicit form~\eqref{quite_explict} of $f\circ \Phi$, these maps on homology can be given obvious and explicit descriptions at the chain level using chain complexes of certain preferred perturbations.
\end{remark}


The particular case of the above corollary which is relevant for us is as follows. We consider two germs of smooth real-valued functions in different spaces. Let $f$ be such a germ near $0\in\R^{nm}$ with isolated critical point at the origin, and let $h$ be such a germ near $0\in\R^{nmk}$, again with an isolated critical point at the origin. 

Equip both spaces with the corresponding Euclidean metrics. Identifying $\R^{nmk}=(\R^{nm})^k$, let $\Delta$ be the $k$-diagonal linear subspace. Identifying $\R^{nmk}=(\R^n)^{mk}$ we can consider the $\Z_{mk}$-action generated by the cyclic shift to the right, which is in $O(nmk)$. Identifying $\R^{nm} = (\R^n)^m$ we can consider the $\Z_m$-action generated by the cyclic shift to the right, which is in $O(nm)$. Assume that $f$ is $\Z_m$-invariant and that $h$ is $\Z_{mk}$-invariant. In particular, $\nabla h$ is tangent to $\Delta$ at points in $\Delta$, and obviously so is $\nabla (\oplus^k f)=\oplus^k \nabla f$. Assume further that $h$ coincides with $\oplus^kf$ on $\Delta$, and that $D^2h(0)|_{\Delta^\bot}$ is non-degenerate; let $E_-\subset\Delta^\bot$ be its negative space and denote $q=\dim E_-$.

Let us use $(x,y)$ as Euclidean coordinates according to the splitting $\R^{nmk} = \Delta\oplus\Delta^\bot$. We claim that if the $\Z_{km}$-action preserves orientations on $E_-$ then Corollary~\ref{lemma_local_MH_tangent_submfd} provides an isomorphism
\begin{equation}\label{abstract_isom_inv_iterations}
\HM_*(h,(0,0))^{\Z_{km}} \simeq \HM_{*-q}(f,0)^{\Z_m}.
\end{equation}
On the right-hand side $0$ denotes the origin in $\R^{nm}$. In fact, Corollary~\ref{lemma_local_MH_tangent_submfd} provides a germ of diffeomorphism $\Phi$ near $0\in\R^{nmk}$ satisfying $\Phi|_\Delta=id|_\Delta$, $\Phi$ is $\Z_{km}$-equivariant, $D\Phi(0,0)=I$ and
\[
h\circ \Phi(x,y) = h(x,0) + \frac{1}{2} \left< D^2h(0,0)\cdot(0,y),(0,y) \right>
\]
near the origin. The quadratic form in the variable $y$ on the right-hand side is $\Z_{km}$-invariant and, by assumption, the $\Z_{km}$-action preserves orientations on its negative space $E_-$. Hence there is an isomorphism as in~\eqref{isom_consideration_with_symm}
\begin{equation}\label{first_equiv_isom}
\HM_*(h\circ\Phi,(0,0))^{\Z_{km}} \simeq \HM_{*-q}(h\circ \Phi|_{\Delta},0)^{\Z_{km}}.
\end{equation}
But the $\Z_{km}$-action restricted to $\Delta$ is not faithful, in fact, it is the $k$-th iteration of a $\Z_m$-action: the element $m\in\Z_{km}$ generates the identity on $\Delta$. Identifying $\R^{nm} = \Delta$ via the diagonal inclusion, it corresponds to the $\Z_m$-action on $\R^{nm}$. Note also that $h\circ\Phi|_\Delta = h|_\Delta = \oplus^k f|_\Delta$ and again the diagonal inclusion allows us to identify $\oplus^k f|_\Delta$ with $kf$. Hence we have isomorphisms
\[
\HM_{*-q}(h\circ \Phi|_{\Delta},0)^{\Z_{km}} \simeq \HM_{*-q}(h\circ \Phi|_{\Delta},0)^{\Z_m} \simeq H_{*-q}(f,0)^{\Z_m}.
\]
Composing these isomorphisms with the isomorphism~\eqref{first_equiv_isom} and the obvious isomorphism given by changing $h$ to $h\circ \Phi$, we obtain~\eqref{abstract_isom_inv_iterations}.

In the absence of cyclic group actions, a simpler argument provides an isomorphism
\begin{equation}\label{abstract_isom_noninv_iterations}
\HM_*(h,(0,0)) \simeq \HM_{*-q}(f,0).
\end{equation}

The proofs of the Persistence Theorem, with or without symmetries, will be given by showing that discrete action functionals fit into the above abstract framework.

\subsection{Persistence Theorem without symmetries}\label{ssec_persistence_no_symmetries}

Now let us go back to Hamiltonian dynamics and prove Theorem~\ref{thm_persistence_non_invariant}. Associated to a $1$-periodic germ $H=\{H_t\}$ near $0\in\R^{2n}$ satisfying $dH_t(0)=0$, $H_{t+1}=H_t$ for all $t$, we have a local Hamiltonian isotopy $\varphi_H^t$ defined by $d\varphi_H^t/dt=X_{H_t} \circ \varphi_H^t$, $\varphi_H^0=id$. This data defines the action functionals $\A_{H,k,N}$ for all $k\geq 1$, as in~\eqref{def_discrete_action}. Denote $\varphi = \varphi_H^1$. Let $N$ be adapted to $H$ in the sense of~\eqref{N_adapted}. Now consider maps $\Phi_{H,1,N}$ and $\Phi_{H,k,N}$ defined as in~\eqref{maps_big_Phi}. The first is defined near the origin in $\R^{2nN}$ while the second is defined near the origin in $\R^{2nkN}$. Let $\Delta \subset \R^{2nkN}$ be the $k$-diagonal, where we identify $\R^{2nkN} \simeq (\R^{2nN})^k$. We have a standard Lagrangian splitting $\R^{2nkN} \simeq \R^{nkN} \oplus \R^{nkN}$ with respect to which we write $\Phi_{H,k,N}(x,y)=(X,Y)$. Define $T(x,y)=(x,Y)$. By Lemma~\ref{lemma_action_as_gen_function} this map satisfies (Gen1), i.e., $T$ is a local diffeomorphism fixing the origin.

\begin{lemma}
The subspace $\Delta$ is invariant under $\Phi_{H,k,N}$ and $T$.
\end{lemma}

\begin{proof}
A point $(z_1,\dots,z_{kN}) \in \Delta$ is a sequence such that $(z_{(\lambda-1)N+1},\dots,z_{\lambda N})$ is independent of $\lambda \in \{1,\dots,k\}$. Denoting $\Phi_{H,k,N}(z_1,\dots,z_{kN}) = (Z_1,\dots,Z_{kN})$, it follows from~\eqref{maps_big_Phi} that $\Phi_{H,k,N}(\Delta) = \Delta$ locally near the origin. Denoting $z_i=(x_i,y_i)$ and $Z_i=(X_i,Y_i)$ we get that $(x_{(\lambda-1)N+1},\dots,x_{\lambda N})$ and $(Y_{(\lambda-1)N+1},\dots,Y_{\lambda N})$ are both independent of $\lambda$. Hence $T$ also preserves $\Delta$.
\end{proof}

\begin{lemma}
$\Delta$ is invariant under the Euclidean gradient flow of $\A_{H,k,N}$.
\end{lemma}

\begin{proof}
The Euclidean gradient of $\A_{}$ has components
\begin{equation*}
\begin{aligned}
& \nabla_{x_i}\A = y_{i+1}-y_i+D_1S_i(x_i,y_{i+1}) \\
& \nabla_{y_i}\A = x_{i-1}-x_i+D_2S_{i-1}(x_{i-1},y_i)
\end{aligned}
\end{equation*}
where here the index $i$ runs from $1$ to $kN$ with the convention that $0$ is to be identified with $kN$, while $kN+1$ is to be identified with $1$. From these formulas we see that $((\nabla_{x_{(\lambda-1)N+1}}\A,\nabla_{y_{(\lambda-1)N+1}}\A),\dots,(\nabla_{x_{\lambda N}}\A,\nabla_{y_{\lambda N}}\A))$ is independent of $\lambda \in \{1,\dots,k\}$ when computed at points of $\Delta$. In other words, the gradient of $\A$ is tangent to $\Delta$ at points of $\Delta$.
\end{proof}

\begin{corollary}
The subspaces $\Delta$ and $\Delta^\bot$ are orthogonal with respect to the symmetric matrix $D^2\A_{H,k,N}(0)$.
\end{corollary}

\begin{proof}
If $(u,v)$ are coordinates associated to the splitting $\Delta \oplus \Delta^\bot$ then from the previous lemma we get $\nabla_v\A(u,0) \equiv 0$. Thus $\nabla^2_{vu}\A$ vanishes along $\Delta$, and the conclusion follows from the symmetry of the Hessian.
\end{proof}

From now on, we assume that $k$ is an admissible iteration. The diagonal map $d:\R^{2nN} \to \Delta$ defined by $d(z)= (z,\dots,z)$ determines a map from $\ker d\Phi_{H,1,N}(0)-I$ to $\ker d\Phi_{H,k,N}(0)-I$ which is, in principle, only injective (even with no assumptions on $k$). Note that the image of $\ker d\Phi_{H,1,N}(0)-I$ under the map $d$ is contained in $\Delta$ since, obviously, $d$ takes values in $\Delta$. There are isomorphisms
\begin{equation*}
\begin{aligned}
\ker d\varphi(0)-I &\simeq \ker d\Phi_{H,1,N}(0)-I \\
\ker d\varphi^k(0)-I &\simeq \ker d\Phi_{H,k,N}(0)-I
\end{aligned}
\end{equation*}
given by
\begin{equation*}
\begin{aligned}
&\delta z \mapsto (\delta z,d\varphi^H_{1/N}\cdot \delta z,\dots,d\varphi^H_{(N-1)/N}(0)\cdot \delta z) \\
&\delta z \mapsto (\delta z,d\varphi^H_{1/N}\cdot \delta z,\dots,d\varphi^H_{(kN-1)/N}(0)\cdot \delta z)
\end{aligned}
\end{equation*}
respectively. Since $k$ is admissible, $\ker d\varphi(0)-I = \ker d\varphi^k(0)-I$. Hence 
the subspaces $\ker d\Phi_{H,1,N}(0)-I$ and $\ker d\Phi_{H,k,N}(0)-I$ have the same dimension. It follows that the image of $\ker d\Phi_{H,1,N}(0)-I$ under $d$ coincides precisely with $\ker d\Phi_{H,k,N}(0)-I$, in particular, $\ker d\Phi_{H,k,N}(0)-I$ is contained in $\Delta$. Now, the formula
\[
d\Phi_{H,k,N}(0)-I = -J_0 \ D^2\A_{H,k,N}(0) \ dT(0)
\]
given by Lemma~\ref{lemma_formula_hessian} tells us that the kernel of $D^2\A_{H,k,N}(0)$ is contained in $\Delta$. In fact, consider $V \in \ker D^2\A_{H,k,N}(0)$. By the above formula one gets $dT(0)^{-1}V\in \ker d\Phi_{H,k,N}(0)-I \subset \Delta$. Thus $V\in\Delta$ since $dT(0)(\Delta)=\Delta$. We can finally conclude that $D^2\A_{H,k,N}(0)|_{\Delta^\bot}$ is an isomorphism of $\Delta^\bot$.

Combining the above arguments with the fact that $\Delta$ is invariant under the Euclidean gradient flow of $\A_{H,k,N}$ and that $\A_{H,k,N}$ coincides with $\oplus^k\A_{H,1,N}$ along $\Delta$, we conclude that there is an isomorphism as in~\eqref{abstract_isom_noninv_iterations}
\[
\mathcal{I} : \HM_*(\A_{H,1,N},0) \to \HM_{*+q}(\A_{H,k,N},0)
\]
for $k\geq 1$ admissible. It is called the {\it iteration map}. Here, the index shift $q$ is the algebraic number of negative eigenvalues of $D^2\A_{H,k,N}(0)$ on $\Delta^\bot$. The Hessian $D^2\A_{H,k,N}(0)$ respects the splitting $\Delta \oplus \Delta^\bot$, and the diagonal isomorphism $\R^{2nN} \to \Delta$ identifies the negative space of $D^2\A_{H,1,N}(0)$ with the negative space of $D^2\A_{H,k,N}(0)|_\Delta$. It follows that $q$ is precisely the difference between the Morse index of $0\in\R^{2nkN}$ as a critical point of $\A_{H,k,N}$ and the Morse index of $0\in\R^{2nN}$ as a critical point of $\A_{H,1,N}$. By~\cite[Proposition~2.5]{mazz_SDM} we get
\begin{equation*}
q = \CZ^{(k)} - \CZ^{(1)} + n(k-1)N
\end{equation*}
where
\begin{equation*}
\begin{aligned}
& \CZ^{(k)} = \CZ(t\in[0,k]\mapsto d\varphi^t_H(0)) \\
& \CZ^{(1)} = \CZ(t\in[0,1]\mapsto d\varphi^t_H(0))
\end{aligned}
\end{equation*}
and $\CZ$ stands for the lower semi-continuous extension of the Conley-Zehnder index to degenerate paths.

The maps $\mathcal{I}$ above commute with the inflation maps constructed in Lemma~\ref{lemma_inflation}. More precisely, there is a commutative diagram
\begin{equation}
\xymatrixcolsep{6pc}\xymatrix{
\HM_*(\A_{H,1,N},0) \ar[d]^{\mathcal{I}} \ar[r]^{\mathscr{I}^{(1)}_N} & \HM_{*+nN}(\A_{H,1,N+1},0) \ar[d]^{\mathcal{I}} \\
\HM_{*+q}(\A_{H,k,N},0) \ar[r]^{\mathscr{I}^{(k)}_N} & \HM_{*+nN+q}(\A_{H,k,N+1},0)
}
\end{equation}
This follows since the various direct sum maps and continuation maps involved in the definitions of $\mathscr{I}_N^{(1)},\mathscr{I}^{(k)}_N$ and $\mathcal{I}$ commute with each other; see the end of Section~\ref{ssec_invariant} and Remark~\ref{rmk_explicit}. These diagrams and the formula above for $q$ (which depends on $N$!) in terms of Conley-Zehnder indices yield an iteration map
\[
\begin{array}{ccc}
\mathcal{I} : \hloc_*(H,1,0) \to \hloc_{*+s_k}(H,k,0) & \text{ with } & s_k = \CZ^{(k)} - \CZ^{(1)}.
\end{array}
\]
This concludes the proof of Theorem~\ref{thm_persistence_non_invariant}.



\subsection{Persistence Theorem with symmetries}\label{ssec_persistence_symmetries}

As before we study a $1$-periodic germ $H=\{H_t\}$ defined near $0\in\R^{2n}$ satisfying $dH_t(0)=0$ for all $t$. We fix $m\in\N$ such that $\varphi_H^m$ has an isolated fixed point at $0\in\R^{2n}$. Let $N\gg1$ be adapted to~$H$. Assume that $k$ is admissible and good for $\varphi^m_H$. Hence $0\in\R^{2n}$ is an isolated fixed point of $\varphi^{km}_H$. Identifying $\R^{2nkm2N} = (\R^{2nm2N})^k$, let $\Delta$ be the $k$-diagonal. Since the Euclidean gradient $\nabla\A_{H,km,2N}$ is tangent to $\Delta$ at points of $\Delta$, the Hessian $D^2\A_{H,km,2N}(0)$ at $0\in\R^{2nkm2N}$ splits as a quadratic form as $Q\oplus Q^\bot$ according to the splitting $\R^{2nkm2N} = \Delta\oplus\Delta^\bot$. Moreover $Q^\bot$ is non-degenerate because $k$ is admissible (see the non-invariant case above). By Lemma~\ref{lemma_preserves_orient_Delta_perp} the $\Z_{km}$-action preserves orientations on the negative space $E_-\subset \Delta^\bot$ of $Q^\bot$. Note that (Remark~\ref{rmk_dim_E_-}) the dimension of $E_-$ is
\[
\dim E_- = \CZ^{(km)} - \CZ^{(m)} + n(k-1)m2N
\]
where $\CZ^{(km)}$ and $\CZ^{(m)}$ are the (lower semi-continuous extensions of the) Conley-Zehnder indices of the paths $t\in[0,km]\mapsto d\varphi^t_H(0)$ and $t\in[0,m]\mapsto d\varphi^t_H(0)$, respectively. Set 
\[
s_{k,m} = \CZ^{(km)} - \CZ^{(m)}.
\]
We have checked all the prerequisites for applying the discussion following Corollary~\ref{lemma_local_MH_tangent_submfd} to conclude that we have an iteration map
\begin{equation*}
\HM_*(\A_{H,m,2N},0)^{\Z_m} \stackrel{\sim}{\to} \HM_{*+s_{k,m}+n(k-1)m2N}(\A_{H,km,2N},0)^{\Z_{km}}
\end{equation*}
as in~\eqref{abstract_isom_inv_iterations}, which is an isomorphism. As in the non-invariant case, such isomorphisms (note the dependence on $N$) commute with inflation maps of Lemma~\ref{lemma_inflation}. Hence we obtain the iteration map as an isomorphism
\begin{equation*}
\mathcal{I} : \hloc^{\rm inv}_*(H,m,0) \to \hloc^{\rm inv}_{*+s_{k,m}}(H,km,0).
\end{equation*}
Summarizing, we have proved Theorem~\ref{thm_persistence_invariant}.

\section{Discrete Chas-Sullivan product}\label{sec_chas_sullivan}

Here we construct a loop product of Chas-Sullivan type on non-invariant local homologies of discrete action functionals. In the analogy with local Floer homology this is to be thought of as the pair-of-pants product.

\subsection{Abstract products in local Morse homologies}\label{sec_products_local_Morse}

Let $(f_1,x_1),\dots,(f_k,x_k)$ be pairs consisting of functions and isolated critical points. This means we are given smooth manifolds without boundary $X_1,\dots,X_k$, smooth functions $f_i:X_i\to \R$, and isolated critical points $x_i\in X_i$ of $f_i$. For what follows there is no loss of generality to assume that $f_i(x_i)=0$ for all $i$. Let $h:X_1\times\dots\times X_k\to\R$ be smooth and such that $(x_1,\dots,x_k)$ is an isolated critical point of $h$. Finally, assume that $Z\subset X_1\times\dots\times X_k$ is a properly embedded co-oriented submanifold of codimension~$r$ such that $(x_1,\dots,x_k) \in Z$ and $h|_Z = f_1\oplus \dots\oplus f_k|_Z$. From this data we would like to construct a map
\begin{equation}\label{product_local_Morse}
\bullet^{(k)} : \HM_{i_1}(f_1,x_1) \otimes\dots\otimes \HM_{i_k}(f_k,x_k) \to \HM_{i_1+\dots+i_k-r}(h,(x_1,\dots,x_k)).
\end{equation}

For each $i$ let $U_i$ be an open, relatively compact, isolating neighborhood for $(f_i,x_i)$. Choose small open neighborhoods $D_i$ and $D_i'$ of $X_i\setminus U_i$ satisfying $\overline D_i \subset D'_i$. Note that
\[
\begin{aligned}
F = D_1 &\times X_2\times\dots\times X_k \cup X_1 \times D_2\times X_3\dots\times X_k \\
&\cup \dots \cup X_1\times\dots\times X_{k-1}\times D_k
\end{aligned}
\]
and
\[
\begin{aligned}
F' = D_1' &\times X_2\times\dots\times X_k \cup X_1 \times D_2'\times X_3\dots\times X_k \\
&\cup \dots \cup X_1\times\dots\times X_{k-1}\times D_k'
\end{aligned}
\]
are small open neighborhoods of $X_1\times\dots\times X_k \setminus U_1\times\dots\times U_k$ satisfying $\overline F\subset F'$.

With $\epsilon>0$ small set
\[
W_i = \{ f_i < \epsilon \} \cap U_i, \ \ \ \ W_{i,-} = (\{ f_i < -k\epsilon \} \cap U_i) \cup (W_i \cap D_i)
\]
and
\[
\begin{aligned}
C &= \{ f_1\oplus\dots\oplus f_k < k\epsilon \} \cap U_1\times\dots\times U_k \\
C_- &= (\{ f_1\oplus\dots\oplus f_k < -\epsilon \} \cap U_1\times\dots\times U_k) \cup (C \cap F) \\
C'_- &= (\{f_1\oplus\dots\oplus f_k <-\epsilon/2\}\cap U_1\times\dots\times U_k) \cup (C \cap F').
\end{aligned}
\]
The closure of $C_-$ with respect to $C$ is contained in $C'_-$.  Finally, set
\[
\begin{aligned}
\Lambda &= \{h<k\epsilon\}\cap U_1\times \dots\times U_k \\ 
\Lambda_- &= (\{h<-\epsilon/2\}\cap U_1\times\dots\times U_k) \cup (F' \cap \Lambda)
\end{aligned}
\]

Our constructions require a technical statement.

\begin{lemma}
Let $X$ be a smooth manifold without boundary, $f:X\to\R$ be a smooth function, $x\in X$ be an isolated critical point of $f$ such that $f(x)=0$, $U$ be an open relatively compact isolating neighborhood for $(f,x)$ and $\O$ be a small open neighborhood of $X\setminus U$. With $a,b>0$ consider
\[
W = \{ f<a \} \cap U \qquad W_- = (\{ f<-b \} \cap U) \cup (W\cap \O).
\]
Define $E=\{f\leq a\}\cap U$. If $a+b$ is small enough then there exist $E^1_-,E^0_-$ closed subsets of $U$ satisfying
\begin{itemize}
\item[a)] Their interiors $\dot E^1_-,\dot E^0_-$ satisfy $\dot E^1_- \subset W_- \subset \dot E^0_-$.
\item[b)] Both $(E,E^0_-)$ and $(E,E^1_-)$ are Gromoll-Meyer pairs for $(f,x)$ in $U$.
\item[c)] The inclusions of pairs $(\dot E,\dot E^j_-)\subset (E,E^j_-)$ induce isomorphisms on homology $H_*(\dot E,\dot E^j_-) \simeq H_*(E,E^j_-)$, for $j=0,1$.
\end{itemize}
In particular the inclusions $(\dot E,\dot E^1_-) \subset (W,W_-)$, $(W,W_-) \subset (\dot E,\dot E^0_-)$ induce maps
\begin{equation}\label{rough_maps_nbds}
\HM(f,x) \to H_*(W,W_-) \qquad H_*(W,W_-) \to \HM(f,x).
\end{equation}
\end{lemma}

\begin{remark}
Simple examples show that both maps in~\eqref{rough_maps_nbds} might not be isomorphisms. But the sets $E^0_-,E^1_-$ may be constructed in such a way that the first map is injective, and the second map is surjective.
\end{remark}

\begin{proof}
In Appendix~\ref{app_invariance} a Gromoll-Meyer pair in $U$ for $(f,x)$ is exhibited as follows. Choose a smooth bump function $\phi:X\to[0,1]$ such that $\phi\equiv0$ near $x$ and $1-\phi$ is compactly supported in $U$. Then the sets $\{f\leq a\} \cap U$, $\{ f-(a+b)\phi\leq-b\} \cap U$ form a Gromoll-Meyer pair provided $a,b>0$ are small enough. How small the numbers $a,b$ need to be depends on $f,x,U,\phi$.

We repeat this construction with two bump functions $\phi_0,\phi_1$ as above such that $\supp(1-\phi_1) \subset U\setminus\O$, and $\phi_0=0$ near $U\setminus\O$. Set
\begin{equation*}
E^1_- = \{ f-(a+b)\phi_1\leq-b\} \cap U \qquad E^0_- = \{ f-(a+b)\phi_0\leq-b\} \cap U
\end{equation*}
Assertion b) follows from the construction in Appendix~\ref{app_invariance}. Note that when $a+b$ is small enough then $a$ is a regular value of $f$ on $U$, and $-b$ is a regular value of both $f-(a+b)\phi_j$ on $U$, $j=0,1$.

It is simple to check that if $a+b$ is small enough then $E,E^1_-,E^0_-$ are smooth top dimensional (domains) with boundary of $U$. Their interiors are
\begin{equation*}
\begin{aligned}
& \dot E = \{f<a\}\cap U, \\
& \dot E^1_- = \{ f-(a+b)\phi_1<-b\} \cap U, \\
& \dot E^0_- = \{ f-(a+b)\phi_0<-b\} \cap U.
\end{aligned}
\end{equation*}
It follows that assertion c) holds. Note that the sets $E,E^1_-,E^0_-$ coincide outside of 
of $\supp(1-\phi_0)$.

Let $z\in W_-$. If $z\in\O$ then $z\in W\cap \O$ because $W_-\cap\O = W\cap\O$. In particular $f(z)<a$. But since $\phi_1(z)=1$, $f(z)<a$ is the same as $f(z)-(a+b)\phi_1(z)<-b$, so we get $z\in \dot E^1_-$. If $z\in U\setminus\O$ then $f(z)-(a+b)\phi_1(z) \leq f(z)<-b$. We proved that $W_- \subset \dot E^1_-$.

Let $z\in \dot E^0_-$. If $z\in U\setminus\O$ then $\phi_0(z)=0$, so that $-b>f(z)-(a+b)\phi_0(z) = f(z)$, and $z\in W_-$ in this case. If $z\in \O$ then we use that $a\geq(a+b)\phi_0(z)-b$ to conclude that the inequality $-b>f(z)-(a+b)\phi_0(z)$ implies $f(z)<a$. Then $z\in W\cap\O \subset W_-$ in this case. We proved that $E^0_- \subset W_-$. Assertion a) is proved.

Existence of the maps~\eqref{rough_maps_nbds} follows from a), b) and c) since the vector spaces $H_*(E,E^j_-)$ are canonically isomorphic to $\HM(f,x)$, see Definition~\ref{def_GM_pairs}.
\end{proof}


By the above lemma, we can find maps
\begin{equation}\label{rough_maps_tau_i}
\tau_i : \HM_*(f_i,x_i) \to H_*(W_i,W_{i,-})
\end{equation}
and a map
\begin{equation}\label{rough_map_Theta}
\Theta : H_*(\Lambda,\Lambda_-) \to \HM_*(h,(x_1,\dots,x_k))
\end{equation}
provided $\epsilon$ is small enough.

The product of pairs $(W_i,W_{i,-})$ is 
$$ 
\begin{aligned}
&\prod_i (W_i,W_{i,-}) \\ &= (W_1\times\dots\times W_k,W_{1,-}\times W_2\times\dots\times W_k \cup W_1\times W_{2,-}\times\dots\times W_k \cup \dots)
\end{aligned}
$$
and hence there are inclusions 
$$
\prod_i(W_i,W_{i,-}) \hookrightarrow (C,C_-) \hookrightarrow (C,C_-\cup(C\setminus Z)).
$$
We get a composite of two inclusions 
\[
\iota: \prod_i(W_i,W_{i,-}) \hookrightarrow (C,C_-\cup(C\setminus Z)).
\]
Composing with the cross-product $\times$ we get a map
\begin{equation}
\iota \circ \times^k : H_{i_1}(W_1,W_{1,-}) \otimes \dots \otimes H_{i_k}(W_k,W_{k,-}) \to H_{i_1+\dots+i_k}(C,C_-\cup(C\setminus Z)).
\end{equation}

Note that $Z\cap C$ is a properly embedded submanifold of the open set $C$, and we consider an open tubular neighborhood $N$ of $Z \cap C$ in $C$. This means that if $\nu$ denotes the normal bundle of $Z\cap C$ then there exists a diffeomorphism $N \simeq \nu$ which identifies the inclusion $Z\cap C \subset N$ with the inclusion of the zero section into $\nu$. This equips $N$ with the structure of an oriented rank $r$ vector bundle $\pi:N \to Z\cap C$. The set $C\setminus N$ is closed in $C$ and is contained in the open subset $C_- \cup (C\setminus Z)$ of $C$. Excision yields an isomorphism
\[
{\rm exc}: H_*(C,C_-\cup(C\setminus Z)) \to H_*(N,(N\cap C_-)\cup(N\setminus Z)).
\]

Consider the vertical saturation $\pi^{-1}(\pi(N\cap C_-)) \supset N\cap C_-$ and let $j$ denote the inclusion
\[
j: H_*(N,(N\cap C_-)\cup(N\setminus Z)) \to H_*(N,\pi^{-1}(\pi(N\cap C_-))\cup(N\setminus Z)).
\]
The relative Thom isomorphism theorem provides an isomorphism
\[
u \cap : H_*(N, \pi^{-1}(\pi(N\cap C_-)) \cup (N\setminus Z)) \to H_{*-r}(N,\pi^{-1}(\pi(N\cap C_-)))
\]
where $u \in H^r(N, \pi^{-1}(\pi(N\cap C_-)) \cup (N\setminus Z))$ is the Thom class. Obviously
\[
\pi : H_{*}(N,\pi^{-1}(\pi(N\cap C_-))) \to H_*(Z\cap C,\pi(N\cap C_-))
\]
is an isomorphism. If $N$ is small enough then $\pi(N\cap C_-) \subset Z\cap C'_-$. From this and from our crucial assumption $f_1\oplus\dots\oplus f_k|_Z = h|_Z$ it follows that 
there is an inclusion
\[
k : H_*(Z\cap C,\pi(N\cap C_-)) \hookrightarrow H_* \left( \Lambda,\Lambda_- \right). 
\]
The map~\eqref{product_local_Morse} is finally defined as the composition
\begin{equation*}
\bullet^{(k)}(a_1,\dots,a_k) = \Theta \circ k \circ \pi\circ (u\cap) \circ j \circ {\rm exc} \circ \iota (\tau_1(a_1) \times \dots \times \tau_k(a_k)).
\end{equation*}
It follows from this formula that $\bullet^{(k)}$ is $k$-multilinear.

This construction for $k=2$ yields a product
\begin{equation*}
\bullet : \HM_{i_1}(f_1,x_1) \otimes \HM_{i_2}(f_2,x_2) \to \HM_{i_1+i_2-r}(h,(x_1,x_2))
\end{equation*}
defined by
\begin{equation*}
a_1 \bullet a_2 = \bullet^{(2)}(a_1,a_2).
\end{equation*}
Let $M:X_1\times X_2 \to X_2\times X_1$ be the diffeomorphism $(q_1,q_2) \mapsto (q_2,q_1)$. Following the above construction with $f_1\oplus f_2,h,Z$ replaced by $f_2\oplus f_1,M_*h,M(Z)$ one defines a product $$ \bullet' : \HM_{i_2}(f_2,x_2) \otimes \HM_{i_1}(f_1,x_1) \to \HM_{i_1+i_2-r}(M_*h,(x_2,x_1)). $$ There is an obvious map $M_*:\HM(h,(x,y)) \to \HM(M_*h,(y,x))$. From functoriality properties given by the Eilenberg-Zilber theorem we get 
$$
 M_*(a\bullet b)=(-1)^{|a||b|} \ b\bullet'a .
 $$

\subsection{Products for discrete action functionals}\label{sec_products_local_action}

Fix a (germ of) $1$-periodic Hamiltonian $H$ near $0\in \R^{2n}$ such that $dH_t(0)=0 \ \forall t \in \R/\Z$. For $k_1,\dots,k_m \in \N$ fixed we assume that the constant trajectory $z(t) \equiv 0$ is isolated when seen as a $k_i$-periodic solution of Hamilton's equation $\dot z(t)=X_{H_t}(z(t))$ for all $i$, and also when seen as a $(k_1+\dots+k_m)$-periodic solution.

Take $N$ large and for each $1\leq i\leq N$ consider the generating function $S_i$ for the germ $\psi_i$ as in~\eqref{germs_small_steps} normalized by $S_i(0)=0$. Extend the family $\{S_i\}$ to a family $\{S_i\,;\, i\in \Z\}$ by $N$-periodicity. Consider discrete action functionals $\A_{H,k_i,N}$ and $\A_{H,k,N}$ where $k=k_1+\dots+k_m$, as defined in~\eqref{def_discrete_action}. It is important to note the drastic difference between $\A_{H,k_1,N} \oplus\dots\oplus \A_{H,k_m,N}$ and $\A_{H,k,N}$, where both are defined near the origin in $\R^{2nkN}$: at a first glance their formulas look the same, but these functionals are, in fact, very different because indices are taken with different periodicity conventions.

Identifying $\R^{2nk_iN} \simeq (\R^{2n})^{k_iN}$, we think of a point $(z_1,\dots,z_{k_iN})$ as a discrete loop with $z_1$ as the base point. Consider the linear manifold
\[
Z = \{ y_1 = y_{k_1N+1} = y_{(k_1+k_2)N+1} = \dots = y_{(k_1+\dots+k_{m-1})N+1} \} \subset \R^{2nkN}
\]
where $z_i=(x_i,y_i)$ denotes the standard Lagrangian splitting $\R^{2n} = \R^n \times \R^n$. Hence $Z$ has codimension $(m-1)n$ and its Euclidean orthogonal complement is
\[
Z^\bot = \left\{ \begin{aligned} & x_i=0 \ \forall i \\ & y_i=0 \ \forall i\not\in\{1,k_1N+1,(k_1+k_2)N+1,\dots,(k_1+\dots+k_{m-1})N+1\} \\ & y_1 + y_{k_1N+1} + y_{(k_1+k_2)N+1} + \dots + y_{(k_1+\dots+k_{m-1})N+1}=0 \end{aligned} \right\}.
\]
We co-orient $Z$ by identifying $Z^\bot \simeq (\R^n)^{m-1}$ via the isomorphism induced by projecting onto $(y_1,y_{k_1N+1},\dots,y_{k_{m-2}N+1})$ and pulling back the canonical orientation. The important observation is that
\begin{equation}\label{identity_action_functionals_product}
\A_{H,k_1,N} \oplus \dots \oplus \A_{H,k_m,N}|_Z = \A_{H,k,N}|_Z.
\end{equation}
\begin{proof}[Proof of~\eqref{identity_action_functionals_product}]
Set $\lambda_1=0$ and $\lambda_l = k_1+\dots+k_{l-1}$ when $2\leq l\leq m+1$. The formula for $\A_{H,k_1,N} \oplus \dots \oplus \A_{H,k_m,N}$ is
\begin{equation*}
\sum_{l=1}^m \left( \begin{aligned} &\left\{ \sum_{i=1}^{k_lN-1} x_{\lambda_lN+i}(y_{\lambda_lN+i+1}-y_{\lambda_lN_i}) + S_{\lambda_lN+i}(x_{\lambda_lN+i},y_{\lambda_lN+i+1}) \right\} \\ & + x_{\lambda_{l+1}N}(y_{\lambda_lN+1}-y_{\lambda_{l+1}N}) + S_{\lambda_{l+1}N}(x_{\lambda_{l+1}N},y_{\lambda_lN+1})\end{aligned} \right)
\end{equation*}
If the point lies in $Z$ then in the second term of each term inside the parenthesis we can replace $y_{\lambda_lN+1}$ by $y_{\lambda_{l+1}N+1}$ (identifying indices $\lambda_{m+1}N+1=kN+1$ and~$1$). This gives exactly the formula for $\A_{H,k,N}$.
\end{proof}

Thus we can apply the construction of the previous section, and the map~\eqref{product_local_Morse} yields a multilinear map 
\begin{equation*}
\bullet^{(m)}: \HM_{i_1}(\A_{H,k_1,N},0) \otimes\dots\otimes \HM_{i_m}(\A_{H,k_m,N},0) \to \HM_{i_1+\dots+i_m-(m-1)n}(\A_{H,k,N},0).
\end{equation*}
These operations interact in the right manner with direct sum maps and continuation maps from Section~\ref{sec_properties}. We get a multilinear map
\begin{equation*}
\bullet^{(m)}: \hloc_{i_1}(H,k_1,0) \otimes\dots\otimes \hloc_{i_m}(H,k_m,0) \to \hloc_{i_1+\dots+i_m-(m-1)n}(H,k,0).
\end{equation*}

As observed before, the case $m=2$ yields a product
\begin{equation}\label{product_action_asymmetric}
\hloc_{i_1}(H,k_1,0) \otimes \hloc_{i_2}(H,k_2,0) \to \hloc_{i_1+i_2-n}(H,k_1+k_2,0)
\end{equation}
defined by
\[
a\bullet b = \bullet^{(2)}(a,b).
\]
It is also important to notice that if we set $M$ to be the linear isomorphism of $\R^{2n(k_1+k_2)N}$ defined by
\[
M(z_1,\dots,z_{(k_1+k_2)N}) = (z_{k_1N+1},\dots,z_{(k_1+k_2)N},z_1,\dots,z_{k_1N})
\]
then
\begin{equation}\label{symmetry_coord_discrete_action}
\A_{H,k_1+k_2,N} \circ M = \A_{H,k_1+k_2,N}.
\end{equation}
Let us still denote by $\bullet$ the product $$ \hloc_{i_1}(H,k_2,0) \otimes \hloc_{i_2}(H,k_1,0) \to \hloc_{i_1+i_2-n}(H,k_1+k_2,0) $$ which, as the reader will notice, happens to be defined using the submanifold $M(Z)=Z$. Super-commutativity
\[
b \bullet a = (-1)^{|a||b|} a\bullet b
\] 
follows from the discussion at the end of Section~\ref{sec_products_local_Morse}, in analogy with loop space homology.

It remains to address associativity, which is a standard property in usual loop space homology and is also true in our context. Identify $$ V = \R^{2n(k_1+k_2+k_3)N} \simeq \R^{2nk_1N} \times \R^{2nk_2N} \times \R^{2nk_3N} $$ with the normal bundle of each of the linear subspaces
\[
\begin{aligned}
& Z'=\{y_1=y_{k_1N+1}\} \\
& Z''=\{y_{k_1N+1}=y_{(k_1+k_2)N+1}\}.
\end{aligned}
\]
These subspaces intersect transversely and both have codimension $n$. We get Thom classes $u' \in H^{n}(V,V\setminus Z')$, $u''\in H^{n}(V,V\setminus Z'')$. Associativity of the product~\eqref{product_action_asymmetric} will follow from the formula
\[
a_1 \bullet (a_2\bullet a_3) = \bullet^{(3)}(a_1,a_2,a_3) = (a_1\bullet a_2)\bullet a_3
\]
which, in turn, follows basically from the fact that $u' \cup u''$ is the Thom class of $Z'\cap Z''$ and from associativity of the cross-product.

\begin{remark}
The product~\eqref{product_action_asymmetric} plays the role of the pair-of-pants product in local Floer homology. In fact, in~\cite{HHM_prep} we will show that there are isomorphisms between ${\rm HF}(\varphi_H^k,0)$ and $\hloc(H,k,0)$ which intertwine the above product and the local pair-of-pants product.
\end{remark}

\subsection{Calculation of a special case}\label{ssec_special_case_prod}

Our goal here is to prove Proposition~\ref{prop_prod_SDM}. In the same set-up as in  Section~\ref{sec_products_local_Morse} above, consider the case where 
\begin{itemize}
\item[(a)] $X_i=\R^{d_i}$, $x_i=0$ is the origin in $\R^{d_i}$ and $\HM_{d_i}(f_i,0)\neq 0$ for all $i$.
\item[(b)] $Z\subset \R^{d_1}\times\dots\times\R^{d_k}$ is a linear subspace (of codimension $r$).
\item[(c)] There is a linear complement $L$ of $Z$ such that $\R^{d_1}\times\dots\times\R^{d_k} = Z \oplus L$, such that $h|_L$ has a local minimum at $(0,\dots,0)$.
\end{itemize}
It is a standard fact, shown in~\cite{Gi}, that $\HM_j(f_i,0)=0$ if $j\neq d_i$, $\HM_{d_i}(f_i,0)=\Q$. Moreover, each $f_i$ has a strict local maximum at $0\in\R^{d_i}$. We use this information and follow the notation in Section~\ref{sec_products_local_Morse}.

The pairs $(W_i,W_{i,-})$ can be chosen so that $W_i=U_i$ is a small open ball centered at the origin and $W_{i,-}$ is the complement in $W_i$ of an open ball centered at the origin of a smaller radius. The homology $H_*(W_i,W_{i,-})$ already computes local Morse homology in this case.

Let $B_1\supset B_2$ be small open balls in $Z$ centered at the origin of different radii, and let $B_1^\bot\supset B_2^\bot$ be small open balls in $Z^\bot$ centered at the origin of different radii. Since $(0,\dots,0)$ is a strict local maximum of $f_1\oplus\dots\oplus f_k$, we can use instead the pair $(C,C_-)$ defined as
\[
(C,C_-) = (B_1,B_1\setminus B_2)\times(B_1^\bot,B_1^\bot\setminus B_2^\bot). 
\]
The homology of this pair computes local Morse homology of $(f_1\oplus\dots\oplus f_k,(0,\dots,0))$. Then both pairs $\prod_i(W_i,W_{i,-})$ and $(C,C_-)$ have the exact same form: they are homeomorphic to a pair $(Q,Q_-)$ where $Q$ is an open ball centered at the origin in $\R^{d_1}\times\dots\times\R^{d_k}$ and $Q_-$ is the complement in $Q$ of an open ball centered at the origin with smaller radius. Since the critical points in question are strict local maxima, there are direct identifications 
\[
\begin{aligned}
& \HM(f_i,0) = H(W_i,W_{i,-}) \\
& \HM(f_1\oplus\dots\oplus f_k,(0,\dots,0)) = H(\prod_i(W_i,W_{i,-})) = H(C,C_-)
\end{aligned}
\]
with no need to consider the maps $\tau_i$~\eqref{rough_maps_tau_i} and $\Theta$~\eqref{rough_map_Theta} of Section~\ref{sec_products_local_Morse}.
The generators are clear from these descriptions. For instance, $H_{d_i}(W_i,W_{i,-})$ is generated by the class $e_i$ represented by a closed $d_i$-cell containing the origin in its interior and having boundary in $W_{i,_-}$. The same picture holds for the generators of $H(\prod_i(W_i,W_{i,-}))$ and $H(C,C_-)$ in degree $d_1+\dots+d_k$.

The tubular neighborhood $N$ in Section~\ref{sec_products_local_Morse} can be taken as $B_1\times B_3^\bot$ where $B_3^\bot$ is an open ball in $Z^\bot$ centered at the origin with radius much smaller than the radius of $B_2^\bot$. Then $Z\cap C = B_1$ and $\pi:N\to Z\cap C$ is the projection onto the first factor $B_1\times B_3^\bot \to B_1$. Since $N\cap C_-$ is already $\pi$-saturated, the map $j$ in Section~\ref{sec_products_local_Morse} is the identity. Let $e$ be a generator in $H_{d_1+\dots+d_k}(\prod_i(W_i,W_{i,-}))$. By the properties of the Thom class we get that
\[
\pi \circ (u\cap) \circ j \circ {\rm exc} \circ \iota : H_*(\prod_i(W_i,W_{i,-})) \to H_{*-r}(Z\cap C,\pi(N\cap C_-)) = H_{*-r}(B_1,B_1\setminus B_2)
\]
maps $e$ to a generator of $H_{d_1+\dots+d_k-r}(B_1,B_1\setminus B_2)$. Now, it follows from condition (c) that the map $k$ in Section~\ref{sec_products_local_Morse} satisfies $k(e)\neq 0$.

Let $e_i$ be a generator in $\HM_{d_i}(f_i,0)=\Q$. Identifying $\HM(f_i,0) = H(W_i,W_{i,-})$ we know that $e_1\times\dots\times e_k$ is a generator of $H_{d_1+\dots+d_k}(\prod_i(W_i,W_{i,-}))=\Q$. Summarizing, we have proved

\begin{lemma}\label{lemma_non_zero_product}
If (a), (b) and (c) hold then $\bullet^{(k)}(e_1,\dots,e_k)\neq0$.
\end{lemma}

Now we apply this lemma to action functionals.  

\begin{lemma}\label{lemma_non_zero_prod_action}
Let $k\in\N$ and $K_t$ be a $1$-periodic Hamiltonian defined near $0\in\R^{2n}$ such that $dK_t(0)=0$ for all $t$, $0$ is an isolated fixed point for $\varphi^1_K$ and $k$ is admissible for~$\varphi^1_K$. If $K$ is $C^2$-small enough and $\hloc_n(K,1,0)\neq 0$ then the following hold.
\begin{itemize}
\item[(i)] $\hloc(K,1,0)$ is supported in degree $n$ and $\hloc_n(K,1,0)=\Q$.
\item[(ii)] If $e$ generates $\hloc_n(K,1,0)$ then $\bullet^{(k)}(e,\dots,e)\neq 0$ in $\hloc_n(K,k,0)$.
\end{itemize}
\end{lemma}

\begin{proof}
Since $K$ is $C^2$-small we know that $N=1$ is adapted to $K$ in the sense of~\eqref{N_adapted}. If $S$ is a generating function for $\varphi^1_K$ then $\A_{K,1,1}=S$ and $\hloc_*(K,1,0) \simeq \HM_{*+n}(S,0)$ by definition. Note that $S$ is $C^2$-small. We get $\HM_{2n}(S,0)\neq 0$ by definition. Thus, the repeated pairs $(f_i,x_i) = (S,0)$ satisfy~(a). Moreover, (i) holds. For the subspace $Z\subset \R^{2nk}$ we take
\[
Z = \{ y_1=\dots=y_k \}
\]
where a point in $\R^{2nk} \simeq (\R^{2n})^k$ is $(z_1,\dots,z_k)$ with $z_i=(x_i,y_i)$. For the function $h$ we take $\A_{K,k,1}$ and note that $\A_{K,k,1}|_Z = S\oplus\dots\oplus S|_Z$ ($k$ factors). Finally, for the complement $L$ we take 
\[
L = \{ x_i = y_{i+1}-y_i \ (i\ {\rm mod}\ k), \ y_1+\dots+y_k=0 \}.
\] 
From the formula $\A_{K,k,1} = \sum_{i\mod k} x_i(y_{i+1}-y_i)+S(x_i,y_{i+1})$ we get 
\[
\A_{K,k,1}|_L = \sum_{i\mod k} |x_i|^2+S(x_i,y_{i+1}).
\]
Hence $\A_{K,k,1}|_L$ has a local minimum at the origin because $S$ is $C^2$-small. We have checked (a), (b) and (c) for the pairs $(\A_{K,1,1},0)$, $(\A_{K,k,1},(0,\dots,0))$ and the splitting $\R^{2nk}=Z\oplus L$. It follows from Lemma~\ref{lemma_non_zero_product} that if $e$ is the generator of $\HM_{2n}(\A_{K,1,1},0)$ then $\bullet^{(k)}(e,\dots,e)\neq 0$ in $\HM_{n+kn}(\A_{K,k,1},0)$. The proof is complete in view of the easily checked compatibility between the products defined in Section~\ref{sec_products_local_action} and continuation and direct-sum maps defined in Section~\ref{sec_properties}.
\end{proof}

\begin{proof}[Proof of Proposition~\ref{prop_prod_SDM}]
To prove this, first we claim that if $0$ is an SDM for $H$ then $1$ is the only Floquet multiplier of $\varphi^1_H$. This follows because, as proved in~\cite{SZ}, the Conley-Zehnder indices of the $1$-periodic orbits which bifurcate from $0$ as we perturb $H$ to a generic $H'$ lie on the interval $[\Delta_{\CZ}(H,1)-n,\Delta_{\CZ}(H,1)+n]$, but they must lie on the open interval $(\Delta_{\CZ}(H,1)-n,\Delta_{\CZ}(H,1)+n)$ if some Floquet multiplier is not equal to $1$. Since $\hloc_n(H,1,0)\neq0$ by assumption, it must be the case that all such generic perturbations produce $1$-periodic orbits with Conley-Zehnder index~$n$. This fact combined with $\Delta_{\CZ}(H,1)=0$ proves the claim. In particular, every $k\in\N$ is admissible for $\varphi^1_H$.

Now, any symplectic matrix having $1$ as the only eigenvalue is linearly symplectically conjugated to a matrix arbitrarily close to the identity; this is proved by Ginzburg in~\cite[Lemma~5.5]{Gi}. Hence, up to a linear symplectic change of coordinates, we may assume that $\varphi^1_H$ is arbitrarily $C^1$-close to $id$. In particular, we find an arbitrarily $C^2$-small germ $K$ of a $1$-periodic Hamiltonian defined near $0\in\R^{2n}$ such that $\varphi^1_K=\varphi^1_H$. Apply Lemma~\ref{lemma_change_of_isotopy} to get 
\[
\hloc_*(K,1,0) = \hloc_{*-2m}(H,1,0)
\]
where $m$ is the Maslov index of the loop $M:t\in\R/\Z \mapsto d\varphi^t_K(0)(d\varphi^t_H(0))^{-1}$. Note that $M(t)^{-1} d\varphi^t_K(0)= d\varphi^t_H(0)$. Since $K$ is $C^2$-small it follows that $-2m$ is close to $\Delta_{\CZ}(H,1)=0$, but $m$ is an integer and we conclude that $m=0$. In particular, $\hloc_n(K,1,0)\neq0$. The desired conclusion now follows from a direct application of Lemma~\ref{lemma_non_zero_prod_action}.
\end{proof}

\section{Preliminaries to transversality statements}\label{sec_prelim_transv}

\subsection{Finite cyclic group actions and invariant functions}\label{subsec_props} 

Let $(M,\theta)$ be a Riemannian manifold without boundary, possibly not compact. Let $k\geq 1$, and let $a$  be a $k$-periodic isometry of $(M,\theta)$ generating an action of $\Z_k$.

For each $p\in M$ one may consider $r=\min\{j \in \{1,\dots,k\} \mid a^j(p)=p\}$. Then $r$ divides $k$ and the isotropy at $p$ is isomorphic to a copy of $\Z_\ell$, $\ell = k/r$, embedded inside $\Z_k$ as $\{0,r,2r,\dots,(\ell-1)r\}$. Let $N$ denote the dimension of $M$. The $\Z_k$-action is linearizable in the sense that one finds a diffeomorphism between a $\Z_\ell$-invariant open neighborhood of $p$ and a Euclidean ball in $\R^N$ centered at the origin, mapping $p$ to the origin, that conjugates $a^r$ to some $A\in O(N)$ satisfying $A^\ell=I$.

For all $m\in\Z$ set
\begin{equation}
\label{set_F}
F_m = \fix(a^m).
\end{equation}
Since $\Z_k$-actions are linearizable  as explained above, each $F_m$ is a smooth submanifold and 
\[
T_xF_m = \ker (da^m|_x-I) \ \ \forall x\in F_m. 
\]

\begin{lemma}\label{lemma_intersection}
$F_i \cap F_j = F_{\gcd(i,j)}$ for all $i,j\geq 1$.
\end{lemma}

\begin{proof}
If $l$ divides $h$ then $F_h \supset F_l$. Thus $F_i \cap F_j \supset F_{\gcd(i,j)}$. Denote $m=\gcd(i,j)$, so that $i=pm$ and $j=qm$ for integers $p,q\geq 1$ satisfying $\gcd(p,q)=1$. Hence we can find $r,s\in \Z$ such that $pr + qs = 1$. Let $x\in F_i\cap F_j$ be arbitrary. Denoting $g = a^{m}$ we compute
\begin{equation*}
\begin{aligned}
x &= a^{ir}\circ a^{js}(x) \\
&= a^{mpr} \circ a^{mqs}(x) \\
&= g^{pr} \circ g^{qs}(x) = g^{pr+qs}(x) = g(x)= a^{m}(x).
\end{aligned}
\end{equation*}
Thus $x \in \fix(g) = F_{m}$.
\end{proof}

\begin{corollary}\label{coro_description_F_j}
For all $j\geq 1$ we have $F_j = F_{\gcd(j,k)}$.
\end{corollary}

The isotropy set $\iso \subset M$ is the set of points $x$ for which the isotropy group 
$$
\iso(x) = \{ j\in \Z/k\Z \mid a^j(x) = x \}
$$
is non-trivial. 
 
 This immediately implies
\begin{corollary}
The isotropy set can be written as $\iso = \bigcup_{d\in \div(k), d<k} F_d$.
\end{corollary}

Now we construct invariant cutoff functions near compact invariant sets.

\begin{lemma}\label{lemma_invariant_nbds}
Let $K\subset M$ be any invariant compact set and $V$ be any neighborhood of $K$. Then there exists an invariant smooth function $\phi:M\to [0,1]$ such that $\supp(\phi) \subset V$ and $\phi \equiv 1$ near $K$.
\end{lemma}

\begin{proof}
By compactness of $K$, for every open neighborhood $\mathcal U$ of $K$ there exists a neighborhood $U$ of $K$ such that all $j\cdot U$, $j\in\Z/k\Z$ are contained in a compact subset of $\mathcal U$. Hence $\cup_{j\in\Z/k\Z} j\cdot U$ is an open invariant neighborhood of $K$ contained in a compact subset of $\mathcal U$.

Hence we find $W,W'$ invariant open neighborhoods of $K$ such that $$ \overline{W'}\subset W\subset \overline{W} \subset V. $$ Take $h:M\to[0,1]$ any smooth function satisfying $\supp(h)\subset W$ and $h\equiv 1$ on $W'$. The average $\phi$ of $h$ over the group satisfies $\phi|_{W'}\equiv 1$ and $\supp(\phi) \subset W$.
\end{proof}

In particular, we have

\begin{corollary}
\label{cor_inv_smooth_nbd}
Any invariant compact subset $K\subset M$ has an arbitrarily small invariant, compact and smooth neighborhood. By a smooth neighborhood we mean a neighborhood with smooth boundary.
\end{corollary}

\begin{proof}
Take $\phi$ as given in Lemma~\ref{lemma_invariant_nbds} and consider $\phi^{-1}([c,1])$ where $c\in(0,1)$ is a regular value of $\phi$.
\end{proof}

Now we can also study the gradient of invariant functions near isotropy points.

\begin{lemma}\label{lemma_grad_tang}
If $f:V\to \R$ is an invariant smooth function then $\nabla^\theta f$ is tangent to $F_j$ for all $j$. In particular $\crit(f) \cap F_j = \crit(f|_{F_j})$ for all $j$.
\end{lemma}

\begin{proof}
Fix $x\in F_j$. Since $da^j|_x$ is a linear isometry of $(T_xM,\theta_x)$ we compute for any $u\in T_xM$ 
$$
\theta(\nabla f_x,u) = df_x\cdot u = df_{x} \cdot (da^j|_x)^{-1} \cdot u = \theta(da^j|_x \cdot \nabla f_x, u) .
$$
Since $u$ is arbitrary we conclude that $\nabla f_x \in \ker (da^j|_x-I) = T_xF_j$.
\end{proof}


\begin{lemma}
\label{lemma_invariant_normal_decreasing}
Let $f\colon F_j\to \R$ be a given smooth invariant function. Then there exists an open invariant neighborhood $V$ of $F_j$ in $M$ and an invariant function $\tilde f\colon V\to \R$ satisfying
\begin{itemize}
\item $\tilde f \equiv f$ on $F_j$.
\item $\crit(\tilde f) = \crit(f) \subset F_j$ and $W^s(x;\tilde f,\theta) \subset F_j$ for all $x\in \crit(\tilde f)$.
\item If $x\in \crit(f)$ is non-degenerate as a critical point of $f$, then $x$ is also non-degenerate as a critical point of $\tilde f$.
\end{itemize}
In the second property the stable manifold is taken with respect to $V$.
\end{lemma}

\begin{proof}
By Corollary~\ref{coro_description_F_j}, there is no loss of generality to assume that $j \in \div(k)$. Let $k=jm$ and write $t^k-1 = (t^j-1)Q(t)$ where 
\[
Q(t) = 1+t^j+t^{2j}+\dots+ t^{(m-1)j}
\]
is relatively prime with $t^j-1$.

Observe that $T_xM = \ker (da^k|_x-I)$ for all $x\in M$, and $T_xF_j = \ker (da^j|_x-I)$ for all $x\in F_j$. The smooth vector bundle $N_j$ over $F_j$ with fiber $N_j|_x = \ker Q(da|_x)$ over $x\in F_j$, satisfies $TF_j \oplus N_j = TM|_{F_j}$. Moreover, $N_j$ is $da$-invariant. Let $\exp$ be the exponential map associated to the $a$-invariant metric $\theta$. Then $\exp$ is well-defined on some open neighborhood $\mathcal{N}_j$ of the zero section of $N_j$. By perhaps shrinking $\mathcal{N}_j$, the map $\exp$ defines a diffeomorphism between $\mathcal{N}_j$ and an open neighborhood $\O_j$ of $F_j$ in $M$. Note that $\O_j$ is an invariant neighborhood, in fact, $a \circ \exp  = \exp \circ da$ on $\mathcal{N}_j$ because $a$ is a $\theta$-isometry. Now define $V:=\O_j$ and consider the projection $\pi:N_j\to F_j$ onto the base point to define $\tilde f$ by 
$$
\tilde f(\exp(v)) = f(\pi(v)) - \|v\|^2
$$
for all $v\in \mathcal{N}_j$. This function is smooth and invariant, and has all the desired properties.
\end{proof}

\subsection{Technical lemmas}

The main goal of this section is to establish some properties of stable and unstable manifolds of invariant functions.

\begin{lemma}\label{lemma_crucial1}
Let $(X,g)$ be a Riemannian manifold without boundary, and $h$ be a Morse function on $X$. Let $Y\subset X$ be a submanifold without boundary such that $\nabla^g h$ is tangent to $Y$. If $x\in \crit(h) \cap Y$ and $U$ is a neighborhood of $x$ in $X$ such that $U \cap W^s(x;h,g) \subset Y$, then $W^s(x;h,g) \subset Y$.
\end{lemma}

\begin{proof}
Denote by $\phi^t$ the flow of $-\nabla^g h$. By Definition~\ref{def_stable_unstable_mfds}, the stable manifold $W^s(x;h,g)$ is the set of points $p\in X$ such that $\phi^t(p)$ is defined for all $t\in[0,+\infty)$ and $\phi^t(p) \to x$ as $t\to+\infty$. For all such $p$ there exists $t_0>0$ such that for all $t\geq t_0$, we have $\phi^t(p) \in U \cap W^s(x;h,g) \subset Y$. Since $\nabla^g h$ is tangent to $Y$ we conclude that $\phi^t(p) \in Y$ for all $t\in[0,t_0]$ by uniqueness of solutions of ODEs.
\end{proof}

\begin{lemma}\label{lemma_crucial4}
Let $(X,g)$ be a Riemannian manifold without boundary, $Y\subset X$ be a submanifold without boundary, and $f:X\to \R$ be a smooth function such that $\nabla^g f$ is tangent to~$Y$. Let $x_0,x_1\in\crit(f)\cap Y$ be non-degenerate, and let $c:\R\to X$ be an anti-gradient trajectory of $f$ from $x_0$ to $x_1$ contained in $Y$. 
\begin{itemize}
\item[(i)] Assume that $W^s(x_1;f,g) \subset Y$. If points of $c(\R)$ are transverse intersection points of $W^s(x_1;f,g)$ with $W^u(x_0;f,g)$ in $X$, then they must also be transverse intersection points of $W^s(x_1;f|_Y,g|_Y)$ with $W^u(x_0;f|_Y,g|_Y)$ in~$Y$.
\item[(ii)] Assume that $W^s(x_i;f,g) \subset Y$, $i=0,1$. If points of $c(\R)$ are transverse intersection points of $W^s(x_1;f|_Y,g|_Y)$ with $W^u(x_0;f|_Y,g|_Y)$ in $Y$, then they are also transverse intersection points of $W^s(x_1;f,g)$ with $W^u(x_0;f,g)$ in $X$.
\item[(iii)] Assume that $D^2f(x_0)$ is negative definite on the $g$-orthogonal complement $(T_{x_0}Y)^g$ of $T_{x_0}Y$. Then $W^s(x_0;f,g) \subset Y$.
\end{itemize}
\end{lemma}

\begin{proof}
Let $p\in c(\R)$. Note that $\nabla^gf$ is tangent to $Y$ and therefore, we have $W^*(x_i;f|_Y,g|_Y) = W^*(x_i;f,g)\cap Y$, for $i\in\{0,1\}$ and $*=u$ or $*=s$. Consider 
\[
\begin{array}{cccc}
E^u=T_pW^u(x_0;f,g), & E^s=T_pW^s(x_1;f,g) & \text{and} & F=T_pY
\end{array}
\]
which are all linear subspaces of $T_pX$. 

First we prove (i). Since $$ W^s(x_1;f|_Y,g|_Y) = W^s(x_1;f,g) \cap Y = W^s(x_1;f,g) $$ we have $T_pX = E^s+E^u$ and $E^s\subset F$. Let $w\in F$ and write $w=e^s+e^u$ with $e^s\in E^s$, $e^u\in E^u$. Then $e^u=w-e^s \in F$ and $$ w\in (F\cap E^u) + E^s = (F\cap E^u) + (F\cap E^s). $$ This shows that $F \subset (F\cap E^u) + (F\cap E^s)$. In other words, $p$ is a transverse intersection point of $W^s(x_1;f|_Y,g|_Y)$ with $W^u(x_0;f|_Y,g|_Y)$ in $Y$, and i) is proved. 

Now we prove (ii). The important observation is that, in this case, $x_0$ is a transverse intersection point of $W^u(x_0;f,g)$ with $Y$ because $W^s(x_0;f,g) \subset Y$. By continuity of tangent spaces, $c(t)$ is a transverse intersection point of $W^u(x_0;f,g)$ with $Y$ provided $t\sim-\infty$. Hence the same is true for every $t$. In particular this holds at $p$. Consequently $T_pX = F+E^u$, and by assumption $F = (F\cap E^u) + E^s$, hence $T_pX = E^u+E^s$ as desired.

Item (iii) follows from uniqueness of the stable manifold at $x_0$ since $Y$ is invariant under the flow of $-\nabla^gf$ by Lemma~\ref{lemma_crucial1}.
\end{proof}

\subsection{A transversality lemma}\label{ssec_transv_lemma}
The following is one of the main technical tools in our constructions. It is a transversality statement which keeps track of the $\Z_k$-symmetry.

\begin{lemma}[Transversality lemma]\label{lemma_crucial3}
Let $\Z_k$ act smoothly by isometries on the smooth Riemannian manifold $(X,\theta_0)$ without boundary. Let $\iso\subset X$ denote the isotropy set, and let $f$ be a $\Z_k$-invariant smooth Morse function on $X$. Let $V_1$ be an open neighborhood of $\iso$ such that $\crit(f)\cap V_1\subset\iso$, and let $V_0\subset X\setminus \iso$ be an open neighborhood of $\crit(f)\setminus \iso$. Assume that:
\begin{itemize}
\item[(i)] $x\in\crit(f)\cap \iso \Rightarrow W^s(x;f,\theta_0) \subset\iso$.
\item[(ii)] $\{x,y\}\subset\crit(f) \cap \iso \Rightarrow W^u(x;f,\theta_0)\pitchfork W^s(y;f,\theta_0)$.
\end{itemize}
Then for every $\ell\geq1$ there exists a residual subset $\mathcal{R}$ of the set ${\rm Met}^\ell_{\Z_k}(V_0,\theta_0)$ of $\Z_k$-invariant metrics of class $C^\ell$ coinciding with $\theta_0$ on $X\setminus V_0$, equipped with the $C^\ell$-topology, with the following property: 
$$
 \theta\in\mathcal{R}\Rightarrow W^u(x;f,\theta)\pitchfork W^s(y;f,\theta) \text{ for all } (x,y)\in\crit(f)\times\crit(f) .
 $$
\end{lemma}

Lemma~\ref{lemma_crucial3} will be proved as a consequence of the following statement. See Section \ref{sssec_lem_crucial_3} for the proof of Lemma~\ref{lemma_crucial3}.

\begin{lemma}\label{lemma_transversality_step}
Let $(X,\theta_0)$ be a smooth Riemannian manifold without boundary where $\Z_k$ acts smoothly by isometries. Let $f:X\to\R$ be a $\Z_k$-invariant smooth Morse function, $\{x,y\}\subset\crit(f)$, $y\not\in\iso$, and let $V_0$ be an open neighborhood of $y$. Consider the set ${\rm Met}^\ell_{\Z_k}(V_0,\theta_0)$ of $C^{\ell}$ metrics ($\ell\geq1$) which are $\Z_k$-invariant and agree with $\theta_0$ on $X\setminus V_0$, equipped with the $C^\ell$-topology.

There exists a residual subset $\mathcal{R}_{x,y}\subset {\rm Met}^\ell_{\Z_k}(V_0,\theta_0)$ with the following property: if $\theta\in\mathcal{R}_{x,y}$ then $W^s(y;f,\theta)$ and $W^u(x;f,\theta)$ intersect transversely. An analogous statement holds if $x\not\in\iso$, with $V_0$ replaced by a neighborhood of $x$.
\end{lemma}


In the following we work towards the proof of Lemma~\ref{lemma_transversality_step}, which is given in Section~\ref{sssec_trans_symmetry}. Assume that $X,\theta_0,f,x,y,V_0$ are as in the statement of this lemma. We can assume $x\neq y$, otherwise there is nothing to prove.

\subsubsection{Functional analytic set-up}

From now on we fix an exponential map $\exp$ associated to a choice of smooth background metric $g$. The dimension of $X$ is denoted by~$n$. Let $\P(x,y)$ denote the set of $W^{1,2}_{\rm loc}$ functions $c:\R\to M$ satisfying
\begin{itemize}
\item $\lim_{t\to+\infty}c(t)=y$, $\lim_{t\to-\infty}c(t)=x$.
\item If we define $v_+:[a,+\infty)\to T_yX$ and $v_-:(-\infty,-a]\to T_xX$, with $a\gg1$,  by 
\begin{equation*}
\begin{aligned}
& c(t)=\exp(v^+(t)) \ \ t\in[a,+\infty) \\
& c(t)=\exp(v^-(t)) \ \ t\in(-\infty,-a]
\end{aligned}
\end{equation*}
then $v^+$ is $W^{1,2}$ on $[a,+\infty)$, and $v^-$ is $W^{1,2}$ on $(-\infty,-a]$, where we choose arbitrary identifications $T_xX\simeq\R^n$ and $T_yX\simeq \R^n$.
\end{itemize}
The set $\P(x,y)$ so defined is independent of the choice of $g$.

\begin{theorem}\label{thm_diff_structure}
$\P(x,y)$ admits the structure of a smooth, separable and Hausdorff Hilbert manifold, modeled on $W^{1,2}(\R,\R^n)$.
\end{theorem}

We do not provide all the analytical details of the proof of this standard theorem, see~\cite[Proposition~2.7]{schwarz_book} where the trajectory space is given an alternative but equivalent definition. However, we do describe the differentiable structure of $\P(x,y)$. The first step is to describe its topology. A sequence $c_j\in\P(x,y)$ is said to converge to $c\in\P(x,y)$ if
\begin{itemize}
\item $c_j\to c$ in $C^0_s(\R,X)$ and in $W^{1,2}_{\rm loc}(\R,X)$, where $C^0_s$ means strong $C^0$-topology.
\item Choosing $a\gg1$, if we define $$ \begin{array}{ccc} v^+_j,v^+:[a,+\infty)\to T_yX & \text{and} & v^-_j,v^-:(-\infty,-a]\to T_xX \end{array} $$ by
\begin{equation*}
\begin{aligned}
& \exp(v^+_j(t))=c_j(t), \ \exp(v^+(t))=c(t) \ \ \ \ t\in[a,+\infty) \\
& \exp(v^-_j(t))=c_j(t), \ \exp(v^-(t))=c(t) \ \ \ \ t\in(-\infty,-a]
\end{aligned}
\end{equation*}
then $v^+_j \to v^+$ in $W^{1,2}$ and $v^-_j\to v^-$ in $W^{1,2}$.
\end{itemize}
A set $\mathcal O \subset\P(x,y)$ is defined to be open if for every $c\in\mathcal O$ and every sequence $c_j\in\P(x,y)$ which converges to $c$ as above one finds $j_0$ such that $c_j\in\mathcal  O$ for all $j\geq j_0$. It is not hard to check that this is a topology which is metrizable and independent of $g$, and that $\P(x,y) \cap C^\infty(\R,X)$ is dense in $\P(x,y)$.

Now we turn to the description of the charts. Let $c\in\P(x,y)$. Then $c^*TX$ is a vector bundle of class $W^{1,2}_{\rm loc}$. A trivialization $\Psi:c^*TX \to \R\times \R^n$ is said to be {\it admissible} if one finds $U^x$ and $U^y$ open neighborhoods of $x,y\in X$ and smooth trivializations $\Phi^x:TU^x\to U^x\times \R^n$ and $\Phi^y:TU^y\to U^y\times \R^n$ such that
\begin{itemize}
\item If $t\gg1$ then $\Psi|_t$ and $\Phi^y|_{c(t)}$ coincide as linear isomorphisms $T_{c(t)}X\simeq \R^n$.
\item If $t\ll-1$ then $\Psi|_t$ and $\Phi^x|_{c(t)}$ coincide as linear isomorphisms $T_{c(t)}X\simeq \R^n$.
\end{itemize}
Using admissible trivializations one identifies $W^{1,2}(\R,\R^n)$ with a space of sections of $c^*TX$, denoted by $W^{1,2}(c^*TX)$. Also, one obtains a Hilbert structure on $W^{1,2}(c^*TX)$ by pulling back that of $W^{1,2}(\R,\R^n)$ via one of these identifications. The resulting Banachable space does not depend on the choice of admissible trivialization. The same procedure can be used to define $L^{2}(c^*TX)$, for $c\in \P(x,y)$. One can show that if $c\in\P(x,y)\cap C^\infty$ then the map
\begin{equation}\label{local_chart_pathspace}
\zeta \mapsto \exp(\zeta)
\end{equation}
is a homeomorphism between a neighborhood of the origin in $W^{1,2}(c^*TX)$ and a neighborhood of $c$ in $\P(x,y)$. This is a chart of the $C^\infty$-differentiable structure given by Theorem~\ref{thm_diff_structure}. Using that $\P(x,y)\cap C^\infty$ is dense in $\P(x,y)$ one proves that the images of such charts cover $\P(x,y)$. Moreover, the arguments from~\cite{eliasson} adapted to the non-compact domain $\R$ show that changes of coordinates are smooth. If $c\in\P(x,y) \cap C^{r}$, with $r\geq1$, then~\eqref{local_chart_pathspace} is a chart of the unique $C^r$-differentiable structure containing the above described $C^\infty$-differentiable structure.

The next step is to define a smooth Hilbert bundle $\mathcal E$ over $\P(x,y)$ with fibers modeled on $L^{2}(\R,\R^n)$. With $c\in\P(x,y)$, we define
\begin{equation*}
\begin{array}{ccccc} \mathcal{E}_c = L^{2}(c^*TX) & \text{ and } & \Pi: \mathcal{E} = \bigsqcup_{c\in\P(x,y)} \mathcal{E}_c \to \P(x,y) \end{array}
\end{equation*}
with $\Pi^{-1}(c) := \mathcal{E}_c$.

\begin{theorem}
$\Pi:\mathcal{E}\to \P(x,y)$ admits the structure of a smooth Hilbert bundle with fibers modeled on $L^{2}(\R,\R^n)$.
\end{theorem}

Again we do not provide full analytical details of the proof, see~\cite[Chapter~2]{schwarz_book} for the description of this Hilbert bundle; analytic details of the construction can be found in~\cite[appendix~A]{schwarz_book}. But we do describe the trivializations of $\mathcal{E}$. We use some of the constructions in~\cite{eliasson}. Let us denote by $K:TTX\to TX$ the connection map associated to $g$. In a local trivialization of $TTX$ induced by a chart of $X$ we have
\[
K(x,v,\delta x,\delta v) = (x,\delta v+\Gamma(x)(\delta x,v))
\]
where $\Gamma (x)(\delta x,v)^k= \sum_{ij}\Gamma^k_{ij}(x)(\delta x)^iv^j$; here $(\delta x)^i,v^j$ are local coordinates of tangent vectors and $\Gamma^k_{ij}$ are the Christoffel symbols of the Levi-Civita connection associated to $g$. Denoting by $\pi:TX\to X$ and $\tau :TTX\to TX$ the bundle projections, we have an isomorphism
\[
(\tau,d\pi,K) : TTX \to TX \oplus TX \oplus TX
\]
which turns out to be a vector bundle isomorphism if we see $TX \oplus TX \oplus TX$ as a vector bundle over the first component $TX$. Let $\mathcal{O}$ denote a neighborhood of the zero section of $TX$ where the map $\exp$ is defined, which will eventually be made smaller below. The derivative $d\exp:\tau^{-1}(\mathcal{O}) \to TX$ of $\exp:\mathcal{O}\to X$ conjugates under the above diffeomorphism to a map
\[
\nabla \exp :\mathcal{O} \oplus TX \oplus TX \to TX
\]
which is linear in the second and third components. In fact, $\nabla\exp(v,u,w) = J(1)$ where $J$ is the Jacobi field along $t\mapsto \exp(tv)$ satisfying $J(0)=u$ and $\frac{DJ}{dt}(0)=w$. We denote by $\nabla_1\exp$ and $\nabla_2\exp$ the maps 
$$
 \begin{array}{ccc} \nabla_1\exp(v) u = \nabla\exp(v,u,0) &\text{ and } & \nabla_2\exp(v) w = \nabla\exp(v,0,w) .\end{array} 
 $$
  These are linear maps $\nabla_i\exp(v):T_{\pi(v)}X \to T_{\exp(v)}X$. We take $\mathcal{O}$ small enough in such a way that these maps are linear isomorphisms.

Fixing $c\in\P(x,y)\cap C^\infty$ and the exponential chart $\zeta \mapsto \exp(\zeta)$ defined on a neighborhood of the origin in $W^{1,2}(c^*TX)$, we define a linear isomorphism
\begin{equation}\label{trivialization}
\begin{array}{ccc} \mathcal{E}_c = L^{2}(c^*TX) \to \mathcal{E}_{\exp(\zeta)} = L^{2}(\exp(\zeta)^*TX), & & \eta \mapsto \nabla_2\exp(\zeta)\eta .\end{array}
\end{equation}
Of course, here one has to prove that this indeed defines a linear isomorphism between the corresponding Hilbert spaces. This trivializes $\mathcal{E}$ on the domain of the chart, and one can show that the transition maps between such trivializations are smooth. The trivializations constructed in such a way over exponential charts~\eqref{local_chart_pathspace} centered at points in $\P(x,y) \cap C^{r}$ will be of class $C^r$, $r\geq1$.

In the following we denote ${\rm Met}^\ell_{\Z_k}(V_0,\theta_0)$ by ${\rm Met}_{\Z_k}^\ell$, for simplicity. With the $C^\ell$-topology this becomes a smooth Banach manifold. In fact, it is identified with an open set on the Banach space of symmetric $(2,0)$-tensors of class $C^\ell$ which vanish on $X\setminus V_0$, equipped with the $C^\ell$ norm. The projection
\[
\P(x,y)\times {\rm Met}_{\Z_k}^\ell \to \P(x,y)
\]
is smooth, and one can pull $\mathcal{E}$ back to a smooth bundle
\[
\mathcal{E}^{\rm univ} \to \P(x,y)\times{\rm Met}_{\Z_k}^\ell
\]
with fiber over $(c,\theta)$ given by $\mathcal{E}^{\rm univ}_{(c,\theta)} = \mathcal{E}_c$.

\subsubsection{Differential equation}

We shall now define a section 
\begin{equation*}
\begin{array}{ccc} s:\P(x,y)\times{\rm Met}_{\Z_k}^\ell \to \mathcal{E}^{\rm univ} & \text{by the equation} & s(c,\theta) = \dot c + \nabla^\theta f\circ c \end{array}.
\end{equation*}
We provide a precise description of this section. Fix $c\in\P(x,y)\cap C^r$, for some $r\geq1$. A neighborhood of $c$ is parametrized by a neighborhood of the origin in $W^{1,2}(c^*TX)$ via the chart $\zeta\mapsto \exp(\zeta)$ of class $C^r$. We compute
\begin{equation}
\begin{aligned}
\frac{d}{dt}\exp(\zeta) &= \nabla \exp(\zeta,\dot c,\nabla_t\zeta) \\
&= \nabla_1\exp(\zeta)\dot c + \nabla_2\exp(\zeta)\nabla_t\zeta \\
&= \nabla_2\exp(\zeta) (\nabla_t\zeta + \Theta(\zeta) \dot c)
\end{aligned}
\end{equation}
where $\nabla_t\zeta$ is the covariant derivative of $\zeta$ along $c$ and $\Theta:\mathcal{O} \to \mathcal{L}(TX)$ is defined by \begin{equation}\label{map_Theta}
\Theta(v) = \nabla_2\exp(v)^{-1} \circ \nabla_1\exp(v):T_{\pi(v)}X\to T_{\pi(v)}X.
\end{equation}
Consider the map
\begin{equation}\label{map_Z}
\begin{array}{ccc} Z:\mathcal{O}\times {\rm Met}^\ell_{\Z_k} \to TX, & & Z(v,\theta)=\nabla_2\exp(v)^{-1}\nabla^\theta f(\exp(v)). \end{array}
\end{equation}
One can show that $Z$ is smooth. Note that $Z(\cdot,\theta)$ denotes a smooth non-linear fiber-preserving map $\mathcal{O}\to TX$. Using the trivialization of $\mathcal{E}^{\rm univ}$ explained before, the section $s$ is represented by the map
\begin{equation}\label{local_formula_of_section}
F(\zeta,\theta) = \nabla_t\zeta + \Theta(\zeta)\dot c + Z(\zeta,\theta) .
\end{equation}
It can be proved that for some open neighborhood $\mathcal{U}$ of the origin in $W^{1,2}(c^*TX)$, $F$ defines a $C^r$-map
\[
F: \mathcal{U}\times{\rm Met}_{\Z_k}^\ell \to L^{2}(c^*TX).
\]
It follows that $s$ is a smooth section since we can take $r=\infty$ and cover $\P(x,y)$ by charts centered at points in $\P(x,y) \cap C^\infty$.

\begin{remark}
In order to study the differential of $F$ we introduce some notation taken from~\cite{eliasson}. Let $p:E\to M$ and  $p':E'\to M$ be smooth vector bundles over the same base, let $U\subset E$ be open, and let $h:U\to E'$ be a smooth map which is fiber-preserving in the sense that $p'\circ h=p$ on $U$. The fiber-derivative of $h$ is the smooth fiber-preserving map $Dh:U\to\mathcal{L}(E,E')$ characterized by
\[
Dh(e)u = \lim_{t\to0} \frac{h(e+tu)-h(e)}{t} \ \ \ \ \ \ \ \text{where} \ e\in U, \ u\in p^{-1}(p(e)).
\]
The limit in the right hand side is taken in the vector space $p'^{-1}(p(e))$.
\end{remark}

Using the above remark and some straightforward estimates one shows that the differential of the map $F$ is given by 
\begin{equation}\label{full_derivative_of_section}
DF(\zeta,\theta):(\eta,\xi) \mapsto \nabla_t\eta + (D\Theta(\zeta)\eta)\dot c + D_1Z(\zeta,\theta)\eta + D_2Z(\zeta,\theta)\xi
\end{equation}
Here $D\Theta$ denotes the derivative of $\Theta$ in the fiber-direction, $D_1Z$ denotes the derivative of $Z$ in the fiber-direction with respect to the first variable, and $D_2Z$ the derivative of $Z$ in directions tangent to ${\rm Met}_{\Z_k}^\ell$. The partial derivative $D_1F$ in the $W^{1,2}(c^*TX)$-direction is the bounded linear map
\[
D_1F(\zeta,\theta):W^{1,2}(c^*TX) \to L^{2}(c^*TX)
\]
given by
\begin{equation}
D_1F(\zeta,\theta)\eta=\nabla_t\eta+(D\Theta(\zeta)\eta)\dot c + D_1Z(\zeta,\theta)\eta.
\end{equation}

The following result is fundamental, but we state it here without a proof.

\begin{theorem}\label{thm_Fredholm_op}
The operator $D_1F(0,\theta)$ is a Fredholm operator. Its Fredholm index is ${\rm ind}(x)-{\rm ind}(y)$, where ${\rm ind}$ denotes the Morse index.
\end{theorem}

For a proof we refer to Schwarz~\cite{schwarz_book}.

\subsubsection{Transversality with $\Z_k$-symmetry}\label{sssec_trans_symmetry}

\indent
We start by investigating $D_2Z$ in more detail. Given $p\in X$ we have $Z(0_p,\theta)=\nabla^\theta f(p)$ where $0_p$ is the origin in $T_pX$ and $\nabla^\theta f$ is the $\theta$-gradient of $f$. Consider the space ${\rm Met}^\ell$ of $C^\ell$-metrics on $X$ which coincide with $\theta_0$ on $X\setminus V_0$. Then $T_\theta{\rm Met}^\ell$ is just the space of $(2,0)$-tensors of class $C^\ell$ which vanish on $X\setminus V_0$, and $T_\theta{\rm Met}^\ell_{\Z_k}$ is the space of those which are $\Z_k$-invariant. Moreover, the map $(v,\theta) \mapsto Z(v,\theta)$ is actually defined on the whole of $\mathcal{O}\times {\rm Met}^\ell$. The map $\xi \in T_\theta{\rm Met}^\ell \mapsto D_2Z(0,\theta)\xi$ assigns a vector field on $X$ to each $\xi$.

\begin{lemma}\label{lemma_Cinfty_linearity}
The equation $D_2Z(0,\theta)\phi\xi=\phi D_2Z(0,\theta)\xi$ holds for all $\xi\in T_\theta{\rm Met}^\ell$ and all $\phi\in C^\infty(X,\R)$.
\end{lemma}

\begin{proof}
This follows trivially from local representations. In fact, the vector field $q\in X \mapsto Z(0_q,\theta) \in T_qX$ is nothing but the $\theta$-gradient  of $f$. Here we denoted by $0_q$ the origin in $T_qX$. Choose local coordinates $x_1,\dots,x_n$ defined on some open subset of $X$. Use them to locally identify the metric with a field of symmetric matrices $\theta=[\theta_{ij}]$ of class $C^\ell$. Setting $\nabla f=(f_{x_1},\dots,f_{x_n})$, the vector field $q\mapsto Z(0_q,\theta)$ is represented as $\theta^{-1}\nabla f$ in these local coordinates. Any $\xi$ can be represented locally as a field $M$ of symmetric matrices, and the vector field $q\mapsto D_2Z(0_q,\theta)\xi$ is locally represented as $-\theta^{-1}M\theta^{-1}\nabla f$. From this the desired $C^\infty(X,\R)$ -linearity is obvious, since $\phi\xi$ is represented as $\phi M$ and $-\theta^{-1}(\phi M)\theta^{-1}\nabla f = -\phi \ \theta^{-1}M\theta^{-1}\nabla f$.
\end{proof}

\begin{lemma}
Given $p\in V_0\setminus (\iso\cup\crit(f))$, $v\in T_pX$, $v\neq0$ and $\theta\in{\rm Met}^\ell_{\Z_k}$, there exists $\xi\in T_\theta{\rm Met}^\ell_{\Z_k}$ such that $g_{p}(D_2Z(0_p,\theta)\xi,v)\neq0$. 
\end{lemma}

As before, here $0_p$ denotes the origin in $T_pX$.

\begin{proof}
As in the proof of the previous lemma, we can choose coordinates $x_1,\dots,x_n$ near $p\simeq(0,\dots,0)$ to write locally $Z(0,\theta)=\theta^{-1}\nabla f$. Here $\theta = [\theta_{ij}]$ is the local representation of the metric $\theta$ as a field of symmetric $n\times n$ matrix, and $\nabla f$ is the vector field $(f_{x_1},\dots,f_{x_n})$.

Let $v\in T_pX$ be a non-zero vector, represented as $w\in \R^n\setminus 0$ via the local coordinates. Denote $u=\nabla f(p)\neq0$. We claim that there is a symmetric $n\times n$ matrix $H$ satisfying $Hu=w$. This is obvious if $u=e_1:=(1,0,\dots,0)$, but if $u\neq e_1$ then we choose $R\in SO(n)$ satisfying $Ru=e_1$, $K$ symmetric satisfying $Ke_1=Rw$, and set $H=R^TKR$. Hence there is a field $M$ of symmetric matrices defined near $0$ satisfying $-\theta(0)^{-1}M(0)\theta(0)^{-1}u=w$. Cutting $M$ off with a cut-off function supported near zero, we get a vector $\xi_0 \in T_\theta{\rm Met}^\ell$ satisfying $D_2Z(0_p,\theta)\xi_0=v$, and which vanishes at other points of the $\Z_k$-orbit of $p$. Hence $g_p(D_2Z(0,\theta)\xi_0,v)=g_p(v,v)\neq0$. The desired $\xi$ is obtained by taking the $\Z_k$-average of $\xi_0$. It is crucial here that non-trivial  elements of $\Z_k$ move $p$ to a different point where $\xi_0$ vanishes. We have $g_p(D_2Z(0,\theta)\xi_0,v) = \frac{1}{k} g_p(D_2Z(0,\theta)\xi,v) \neq 0$ because $p\not\in \iso$ and the support of $\xi_0$ is a small neighborhood of $p$.
\end{proof}

\begin{lemma}\label{lemma_transv_technical}
Let $x,y\in \crit(f)$, $\theta\in {\rm Met}^\ell_{\Z_k}$, $c\in \P(x,y)\cap C^1$ and $\eta\in C^0(c^*TX)$ be given. Assume that 
there exists $t_0\in\R$ satisfying $c(t_0)\in V_0\setminus(\crit(f)\cup\iso)$, $\dot c(t_0)\neq 0$ and 
\[
\{ t\in \R \mid \text{$c(t)$ belongs to the $\Z_k$-orbit of $c(t_0)$} \} = \{t_0\}.
\]
Assume also that $\eta(t_0)\neq0$. Then there exists $\xi\in T_\theta{\rm Met}^\ell_{\Z_k}$ such that the function $$ t\in\R\mapsto g_{c(t)}(D_2Z(0_{c(t)},\theta)\xi,\eta(t)) \in \R $$ is everywhere non-negative, and positive at $t_0$.
\end{lemma}

\begin{proof}
By the previous lemma we find $\xi_1\in T_\theta{\rm Met}^\ell_{\Z_k}$ such that
\[
g_{c(t_0)}(D_2Z(0_{c(t_0)},\theta)\xi_1,\eta(t_0))> 0.
\]
In particular we find $\epsilon>0$ such that $g_{c(t)}(D_2Z(0_{c(t)},\theta)\xi_1,\eta(t))> 0$ for every $t\in (t_0-\epsilon,t_0+\epsilon)$. If $\phi$ is a smooth, non-negative, $\Z_k$-invariant, real-valued function on $X$ supported very near the $\Z_k$-orbit of $c(t_0)$ and satisfying $\phi(c(t_0))>0$, then the assumptions on $c$ imply that $\phi\circ c$ is a non-negative function with compact support contained in $(t_0-\epsilon,t_0+\epsilon)$. We choose such a $\phi$. Applying the $C^\infty$-linearity given by Lemma~\ref{lemma_Cinfty_linearity}, we get the formula
\[
g_{c(t)}(D_2Z(0_{c(t)},\theta)\phi\xi_1,\eta(t)) = (\phi\circ c)(t) \ g_{c(t)}(D_2Z(0_{c(t)},\theta)\xi_1,\eta(t)).
\]
The right hand side is a product of two functions, the first being non-negative and supported in the interval $(t_0-\epsilon,t_0+\epsilon)$, the second being positive in this interval. Setting $\xi=\phi\xi_1$ we get the desired conclusion.
\end{proof}


\begin{proposition}\label{prop_surjectivity_universal}
If $s(c,\theta)=0$ then $DF(0,\theta)$ is surjective.
\end{proposition}

\begin{proof}
Let $(c,\theta)$ be a zero of $s$. Then $c$ is of class $C^{\ell+1}$ since the identity $s(c,\theta)=0$ is equivalent to $\dot c+\nabla^\theta f(c)=0$ and $\nabla^\theta f$ is a vector field of class $C^\ell$. 

The first is to show that $c(t)$ satisfies the assumptions of Lemma~\ref{lemma_transv_technical}. Clearly, for every $t$ we have $c(t) \not\in \crit(f)$. Using the standing assumption that $y\not\in\iso$ we know that $c(t_0)\not\in\iso$ when $t_0\sim+\infty$. Since $x\neq y$, $c(t)$ is a non-constant trajectory of the flow of $-\nabla^\theta f$. In particular $\dot c(t)$ is continuous and does not vanish. We claim that if $t_0$ is close enough to $+\infty$ then 
\begin{equation}\label{crucial_assumption_previous_lemma}
\{ t\in \R \mid \text{$c(t)$ belongs to the $\Z_k$-orbit of $c(t_0)$} \} = \{t_0\}.
\end{equation}
If not we find $t_1\in\R$ and $j\in \Z_k$ such that $t_1\neq t_0$ and $j \cdot c(t_0) = c(t_1)$. Necessarily we must have $j\neq0$. It follows from uniqueness of solutions of ODEs that the identity $j\cdot c(t+t_0) = c(t+t_1)$ holds identically in $t\in\R$. The $\Z_k$-symmetry of the vector field $\nabla^\theta f$ was used. This can be rewritten as $j\cdot c(t) = c(t+t_1-t_0)$ for all~$t$. Consider the sequence $(mj)\cdot c(0)$ with $m\in\N$. Since $mj$ varies in $\Z_k$ we find that $\{(mj)\cdot c(0)\}_{m\geq1}$ is a finite set of points. However, it follows from $t_1-t_0\neq0$ that $\{c(m(t_1-t_0))\}_{m\geq1}$ is an infinite set of points. We get a contradiction from the identity $(mj)\cdot c(0)=c(m(t_1-t_0))$ established above. This proves~\eqref{crucial_assumption_previous_lemma}.

In the local chart the point $(c,\theta)$ gets represented as $(0,\theta)$ where $0$ denotes the zero section of $c^*TX$, and $s$ gets represented by a map $F$ given by the formula~\eqref{local_formula_of_section}. For simplicity, we write in this proof $W^{1,2}$ and $L^{2}$ instead of $W^{1,2}(c^*TX)$ and $L^{2}(c^*TX)$, respectively. 

The pairing
\begin{equation}
\left< \eta_1,\eta_2 \right> = \int_{-\infty}^{+\infty} g_{c(t)}(\eta_1(t),\eta_2(t)) \ dt
\end{equation}
is non-degenerate in $L^{2}$. The operator $D_1F(0,\theta):W^{1,2}\to L^{2}$ is Fredholm by Theorem~\ref{thm_Fredholm_op}. It follows that $DF(0,\theta)$ has a closed image. In view of the Hahn-Banach theorem, to show that $DF(0,\theta)$ is onto it suffices to prove that if $\eta\in L^{2}$ satisfies $\left<DF(0,\theta)(\zeta,\xi),\eta\right>=0$ for all $(\zeta,\xi)\in W^{1,2}\times T_\theta{\rm Met}^\ell_{\Z_k}$, then $\eta=0$. 

We proceed with this goal in mind. Let $\eta\in L^{2}$ be such a section. In particular $\left<D_1F(0,\theta)\zeta,\eta\right>=0$ for all $ \zeta\in W^{1,2}$. Let $D^*$ be the formal adjoint operator of $D_1F(0,\theta)$ with respect to the pairing $\left<\cdot,\cdot\right>$. In a trivialization of $c^*TX$, the operator $D^*$ has the form $\frac{d}{dt}+A(t)$ for some path of matrices $A(t)$ of class $C^\ell$. It follows that $\eta$ is a weak solution of $\eta'+A\eta=0$. Hence $\eta$ is $C^{\ell+1}$ and $\eta=0$ if, and only if, $\eta(t_0)=0$ for some $t_0$. It also follows that $\eta\in W^{1,2}$ but we do not need this fact in our particularly simple set-up. From now on we proceed indirectly assuming that $\eta\neq 0$. Then $\eta$ is smooth and $\eta(t)\neq 0$ for all $t$. Since $y\not\in\iso$ we can apply Lemma~\ref{lemma_transv_technical} with some $t_0$ such that $t_0\gg1$ (note that $c$ is smooth and $\dot c$ is nowhere vanishing) to find $\xi\in T_\theta{\rm Met}^\ell_{\Z_k}$ satisfying $\left<D_2F(0,\theta)\xi,\eta\right>>0$. Here we have used that $D_2F(0,\theta)\xi = \{t\mapsto D_2Z(0_{c(t)},\theta)\xi\}$ is a section of $c^*TX$. This contradiction shows that $\eta=0$, as desired.
\end{proof}

Note that $F$ is the local representative of $s$ near a zero $(c,\theta)$.
Since $DF(0,\theta)$ is the direct sum of the Fredholm operator $D_1F(0,\theta)$ with the bounded operator $D_2F(0,\theta)$ one concludes from~\cite[Lemma A.3.6]{Jcurves} that $DF(0,\theta)$ has a right inverse. By the implicit function theorem, the set
\begin{equation*}
\mathcal{M}^{\rm univ}(x,y) = \{\text{vanishing locus of }s\} \subset \P(x,y) \times {\rm Met}^\ell_{\Z_k}
\end{equation*}
is a smooth separable Banach manifold. Again by~\cite[Lemma A.3.6]{Jcurves}, the projection
\begin{equation*}
{\rm pr}_2: \mathcal{M}^{\rm univ}(x,y) \to {\rm Met}^\ell_{\Z_k}
\end{equation*}
is a (smooth) Fredholm map. We define
\begin{equation}
\mathcal{R}_{x,y} = \{ \text{regular values of ${\rm pr}_2$} \}
\end{equation}
and apply the Sard-Smale theorem to conclude that $\mathcal{R}_{x,y}$ is residual in ${\rm Met}^\ell_{\Z_k}$, see~\cite[Theorem A.5.1]{Jcurves}. Yet another application of~\cite[Lemma A.3.6]{Jcurves} tells us that if $\theta\in\mathcal{R}_{x,y}$ then 
\begin{equation*}
\mathcal{M}_\theta(x,y) = \{ c\in\P(x,y) \mid \dot c+\nabla^\theta f\circ c=0\}
\end{equation*}
is a smooth manifold of dimension $\ind(x)-\ind(y)$. To see this, just note that $\mathcal{M}_\theta(x,y) \times \{\theta\} = {\rm pr_2}^{-1}(\theta)$.


\begin{lemma}\label{lemma_vanishing_props_section}
If $\theta\in\mathcal{R}_{x,y}$, $c\in\mathcal{M}_\theta(x,y)$ and $\zeta\in T_c\mathcal{M}_\theta(x,y)$, then $\zeta$ is $C^{\ell+1}$ and $\zeta=0$ if, and only if, $\zeta(t_0)=0$ for some $t_0\in\R$.
\end{lemma}

\begin{proof}
Obviously $c$ is $C^{\ell+1}$ and $\zeta\in T_c\mathcal{M}_\theta(x,y)$ if, and only if, $\zeta\in \ker D_1F(0,\theta)$. In an admissible trivialization of $c^*TX$ the operator $D_1F(0,\theta)$ gets represented as $\frac{d}{dt}+S(t)$ for some path of matrices $S(t)$ of class $C^{\ell}$. In particular, $\zeta(t)$ solves a linear ODE weakly. Consequently, it is $C^{\ell+1}$, and it vanishes identically if, and only if, it vanishes at some point.
\end{proof}

Before proving Lemma~\ref{lemma_transversality_step} we review some basic facts of asymptotic analysis. In the next two lemmas, $\varphi^t$ denotes the flow of $-\nabla^\theta f$ and $\theta$ is an arbitrary metric of class $C^\ell$. Consider $\sigma>0$ defined by
\begin{equation}\label{sigma_spectrum}
\sigma = \min \{ |\lambda| : \lambda \in {\rm spec}({\rm Hess}_x(f)) \cup {\rm spec}({\rm Hess}_y(f)) \}
\end{equation}
and fix $0<\delta<\sigma$ arbitrarily.

\begin{lemma}\label{lemma_asymptotic_1}
If $p\in W^s(y;f,\theta)$ then, with $a\gg1$, the map $v:[a,+\infty)\to T_yX$ defined by $\varphi^t(p)=\exp(v(t))$ satisfies $e^{\delta t}|\partial^jv(t)|\to0$ as $t\to+\infty$, for all $0\leq j\leq \ell-3$. If $p\in W^u(x;f,\theta)$ then an analogous statement holds for $t\to-\infty$.
\end{lemma}

\begin{proof}
We work in a local coordinate system around $y$, where $y$ corresponds to $0\in\R^n$. We denote by $Y$ the local representation of $-\nabla^\theta f$. Then $Y(0)=0$ and for $t\gg1$, the curve $\varphi^t(p)$ is represented as $z(t)$ satisfying $\dot z=Y\circ z$, $z(t)\to0$ as $t\to+\infty$. If $p=y$, there is nothing to prove, so assume $p\neq y$. Thus $z(t)$ does not vanish.

Let $A=DY(0)$, $g=\frac{1}{2}|z|^2$,
\[
M(t) = \int_0^1 DY(\tau z(t))-A \ d\tau
\]
and $\Lambda(t)$ be the $t$-dependent symmetric bilinear form
\[
\Lambda(t) = \int_0^1\int_0^1 D^2Y(s\tau z(t))\tau \ dsd\tau.
\]
Note that $A^T=A$ and 
\[
M(t)u = \Lambda(t)(z(t),u) \ \ \ \ \forall u\in\R^n.
\]
Then $\dot z=Y(z)$ is rewritten as
\begin{equation}\label{ODE_coord}
\dot z(t)=Az(t)+ M(t)z(t) = Az(t) + \Lambda(t)(z(t),z(t)).
\end{equation}
Note that from the above assumptions and equations we get:
\begin{equation}\label{conclusions}
\left\{
\begin{aligned}
& \text{($\sup_t|\Lambda(t)|<\infty$ and $\lim_{t\to+\infty}\dot\Lambda(t) = 0$) \ $\Rightarrow$ \ $\lim_{t\to+\infty}\left(|M(t)|+|\dot M(t)|\right)=0$} \\
& \text{$|\dot z|\leq C|z|$ for some $C>0$ \ $\Rightarrow$ \ $\lim_{t\to+\infty}|\dot z(t)|=0$}.
\end{aligned}
\right.
\end{equation}

Differentiating, we obtain
\[
\dot g = \left<z,\dot z\right> = \left< z,(A+M)z \right>
\]
and
\[
\begin{aligned}
\ddot g &= \left< \dot z,(A+M)z \right> + \left< z,\dot Mz \right> + \left< z,(A+M)\dot z \right> \\
&= |Az|^2 + 2\left< Az,Mz \right> + |Mz|^2 + \left< z,\dot Mz\right> + |Az|^2 + \left< Az,M^Tz \right> \\
&+ \left< Mz,Az \right> + \left< Mz,M^Tz \right> .
\end{aligned}
\]
Using~\eqref{conclusions} we conclude that for all $\epsilon\in(0,\sigma)$ one finds $t_0$ such that if $t\geq t_0$ then the right hand side of the above equation is estimated from below as follows
\[
\ddot g(t) \geq 4(\sigma-\epsilon)^2 g(t) \text{ for all } t\geq t_0.
\]
Here $\sigma$ is the number defined in \eqref{sigma_spectrum}.

We claim that $\dot g(t)<0$ provided $t$ is large enough. To see this consider $h(t) = g+\frac{\dot g}{2(\sigma-\epsilon)}$. Then for $t\geq t_0$ we estimate
\[
\dot h = \dot g+ \frac{\ddot g}{2(\sigma-\epsilon)} \geq \dot g + 2(\sigma-\epsilon)g = 2(\sigma-\epsilon)h.
\]
Suppose $t_*>t_0$ satisfies $\dot g(t_*)\geq 0$. This implies that $h(t_*)>0$. If $h>0$ on $[t_*,T]$, then the inequality $\dot h\geq 2(\sigma-\epsilon)h$ implies that $$ h(T)\geq h(t_*)e^{2(\sigma-\epsilon)(T-t_*)}\geq h(t_*). $$ This shows that $h>0$ on $[t_*,+\infty)$, and $h(t)\geq h(t_*)e^{2(\sigma-\epsilon)(t-t_*)}$ for all $t\geq t_*$. Since $g(t)\to0$ as $t\to+\infty$, we must have $\dot g(t)\to+\infty$ as $t\to+\infty$, which then contradicts $\lim_{t\to+\infty}g(t)=0$. We have proved that $\dot g<0$ provided $t$ is large enough. With this in mind, we consider the function
\[
a = g - \frac{\dot g}{2(\sigma-\epsilon)}
\]
which is positive if $t$ is large enough. We estimate
\[
\dot a = \dot g - \frac{\ddot g}{2(\sigma-\epsilon)} \leq \dot g - 2(\sigma-\epsilon)g = -2(\sigma-\epsilon)a.
\]
Since $a$ is positive for $t$ large, we find $t_1\gg1$ such that 
\[
t\geq t_1 \ \Rightarrow \ g(t) < a(t) \leq a(t_1)e^{-2(\sigma-\epsilon)(t-t_1)}.
\]
In other words, for some $B_0>0$ we have $|z(t)| \leq B_0e^{-(\sigma-\epsilon)t}$ for $t\gg1$. Taking $\epsilon$ small enough so that $\delta < \sigma-\epsilon$ we get
\begin{equation*}
|z(t)| \leq B_0e^{-\delta t} \text{ for all } t\geq t_1.
\end{equation*}
Using~\eqref{ODE_coord} and the above estimate, we get $|\dot z(t)|\leq B_1e^{-\delta t}$ for $t$ large enough, with some $B_1>0$. Differentiating~\eqref{ODE_coord} and proceeding inductively, we get estimates $|\partial^jz(t)|\leq B_je^{-\delta t}$ for $t$ large enough, with some $B_j>0$. This concludes the proof in case $p\in W^s(y;f,\theta)$. The case $p\in W^u(x;f,\theta)$ is entirely analogous.
\end{proof}

If $p\in W^s(y;f,\theta)$ and $c(t) = \varphi^t(p)$, $t\in[0,+\infty)$, then we say that a trivialization $\Psi:c^*TX\to[0,+\infty)\times\R^n$ is {\it admissible} if on over some neighborhood $U$ of $y$ there exists a trivialization $\Phi:TU\to U\times\R^n$ such that $\Psi_t$ and $\Phi_{c(t)}$ give the same linear isomorphism $T_{c(t)}X\simeq\R^n$ whenever $t$ is large enough. If $p\in W^u(x;f,\theta)$ then we could consider $c(t) = \varphi^t(p)$, $t\in(-\infty,0]$, and define admissible trivializations analogously.

\begin{lemma}\label{lemma_asymptotic_2}
If $p\in W^s(y;f,\theta)$, $v\in T_pW^s(y;f,\theta)$ and $\eta(t)=d\varphi^t(p)v$ then, setting $c(t)=\varphi^t(p)$ and identifying $c^*TX\simeq [0,+\infty)\times \R^n$ via an admissible trivialization, the section $\eta$ gets represented as a map $u:[0,+\infty)\to\R^n$ satisfying $e^{\delta t}|\partial^ju(t)|\to0$ as $t\to+\infty$, for all $0\leq j\leq \ell-3$. An analogous statement holds for $p\in W^u(x;f,\theta)$ and $v\in T_pW^u(x;f,\theta)$.
\end{lemma}

\begin{proof}[Sketch of proof]
The estimates are analogous to the previous lemma. We outline the argument. It suffices to look at a point $p\in W^s(y;f,\theta)$ which lies on a small coordinate neighborhood of $y\simeq 0\in\R^n$. Writing $Y=-\nabla^\theta f$ and $z(t) = \varphi^t(p)$ locally, a solution $u(t)$ of the linearized flow along $z$ satisfies $\dot u=DY(z)u$. Note that such a coordinate system induces an admissible trivialization of $c^*TX$ along the positive end of $c$. Since $u$ is tangent to $W^s(y;f,\theta)$, we obtain $u(t)\to0$ as $t\to+\infty$. By the previous lemma, the matrix $D(t) = DY(z(t))$ satisfies $e^{t\delta}|\partial^j(D(t)-A)|\to0$ as $t\to+\infty$ for every $0\leq j\leq\ell-3$, where $A=DY(0)$. Plugging into the linear ODE satisfied by $u$ we obtain the desired conclusions.
\end{proof}

Choose $\theta\in\mathcal{R}_{x,y}$ and consider the (smooth) evaluation map
\begin{equation}\label{ev_map}
\begin{array}{ccc} {\rm ev}:\mathcal{M}_\theta(x,y) \to X & & {\rm ev}(c)=c(0). \end{array}
\end{equation}
The two lemmas above have the following consequence.

\begin{corollary}
If $\ell\geq3$ and $\theta\in\mathcal{R}_{x,y}$ then the following holds:
\begin{equation*}
\begin{aligned}
{\rm ev}(\mathcal{M}_\theta(x,y)) &= W^u(x;f,\theta)\cap W^s(y;f,\theta) \\
d\,{\rm ev}(c)(T_c\mathcal{M}_\theta(x,y)) &= T_{{\rm ev}(c)}W^u(x;f,\theta) \cap T_{{\rm ev}(c)}W^s(y;f,\theta).
\end{aligned}
\end{equation*}
The second identity holds for every $c\in\mathcal{M}_\theta(x,y)$.
\end{corollary}

\begin{proof}
The inclusion ${\rm ev}(\mathcal{M}_\theta(x,y)) \subset W^u(x;f,\theta)\cap W^s(y;f,\theta)$ is clear. For the other inclusion, consider $p\in W^u(x;f,\theta)\cap W^s(y;f,\theta)$. Setting $c(t) = \varphi^t(p)$ we obtain $c\in\P(x,y)$ since the exponential decay given by Lemma~\ref{lemma_asymptotic_1} immediately implies that the vector fields $v_\pm$ as in the definition of $\P(x,y)$ are of class $W^{1,2}$ on their respective domains. Hence $$ {\rm ev}(\mathcal{M}_\theta(x,y)) \supset W^u(x;f,\theta)\cap W^s(y;f,\theta) $$ since obviously $s(c,\theta)=0$. The first claim is proved.

For the second claim, the inclusion $\supset$ is again clear. For the other direction, note that $p={\rm ev}(c)$. Let $v\in T_pW^u(x;f,\theta) \cap T_pW^s(y;f,\theta)$. Then $\eta(t) = d\varphi^t(p)v$ is a section of $c^*TX$. The exponential decay given by Lemma~\ref{lemma_asymptotic_2} implies that $\eta\in W^{1,2}(c^*TX)$. Here we used that a norm on $W^{1,2}(c^*TX)$ is defined by identifying this space with $W^{1,2}(\R,\R^n)$ via admissible trivializations.

Let $\gamma(s)$ be a smooth curve defined for $|s|$ small such that $\gamma(0)=p$, $\dot\gamma(0)=v$. For every $L>0$ there exists $\epsilon>0$ such that if $|s|<\epsilon$ then there is a unique vector field $\zeta(s,t) \in T_{c(t)}X$ for $t\in[-L,L]$ satisfying $\varphi^t(\gamma(s))=\exp(\zeta(s,t))$. We claim that $\eta(t)=\frac{d}{ds}\zeta(s,t)|_{s=0}$ for all $ t\in[-L,L]$, where we see $s\mapsto \zeta(s,t)$ as a curve in the vector space $T_{c(t)}X$. To see this we compute
\begin{equation*}
\begin{aligned}
\eta(t) &= d\varphi^t(p)\dot\gamma(0) = \left.\frac{d}{ds}\right|_{s=0} \varphi^t(\gamma(s)) = \left.\frac{d}{ds}\right|_{s=0} \exp(\zeta(s,t)) \\
&= d\exp \ \left.\frac{d}{ds}\right|_{s=0} \zeta(s,t) = \nabla\exp \left( 0,0, \left.\frac{d}{ds}\right|_{s=0} \zeta(s,t) \right) = \left.\frac{d}{ds}\right|_{s=0} \zeta(s,t)
\end{aligned}
\end{equation*}
where we use that the base point of $\zeta(s,t)$ is $c(t)$ independent of $s$, and that $\zeta(0,t)$ vanishes.

The equation $\frac{d}{dt}\varphi^t(\gamma(s))+\nabla^\theta f(\varphi^t(\gamma(s)))=0$ for $t\in[-L,L]$, is equivalent to the equation $\nabla_t\zeta+\Theta(\zeta)\dot c + Z(\zeta,\theta)=0$ in view of the definition of the maps $\Theta$ and $Z$, see~\eqref{map_Theta}-\eqref{map_Z}. Differentiating with respect to $s$ and evaluating at $s=0$ we get
\[
\nabla_t\eta+D_1Z(0,\theta)\eta=0
\]
since $D\Theta(0)$ vanishes. Since $L$ can be taken arbitrarily large, we conclude that $\eta\in W^{1,2}(c^*TX)$ is a solution of $D_1F(0,\theta)\eta=0$, in other words, $\eta\in T_c\mathcal{M}_\theta(x,y)$. Since $v = \eta(0) = d{\rm ev}(c)\eta$ we get
\[
d\,{\rm ev}(c)(T_c\mathcal{M}_\theta(x,y)) \supset T_{{\rm ev}(c)}W^u(x;f,\theta) \cap T_{{\rm ev}(c)}W^s(y;f,\theta)
\]
which completes the proof.
\end{proof}

Using these lemmas, we can now prove the main results of this section.

\begin{proof}[Proof of Lemma~\ref{lemma_transversality_step}]
Choose $\theta\in\mathcal{R}_{x,y}$. By Lemma~\ref{lemma_vanishing_props_section} the map $$ d{\rm ev}(c):T_c\mathcal{M}_\theta(x,y) \to T_{c(0)}X $$ is injective. By the above corollary the following holds for all $p\in W^u(x;f,\theta)\cap W^s(y;f,\theta)$:
\[
\dim \ T_pW^u(x;f,\theta) \cap T_pW^s(y;f,\theta) = \ind(x)-\ind(y).
\]
Hence $T_pX = T_pW^u(x;f,\theta) + T_pW^s(y;f,\theta)$, as desired.
\end{proof}

\subsubsection{Proof of Lemma~\ref{lemma_crucial3}}\label{sssec_lem_crucial_3}
We can now finally prove the transversality lemma.
Define the desired set as 
$$
\mathcal{R} = \bigcap \left\{ \mathcal{R}_{x,y} \mid \{x,y\}\subset\crit(f),\ \{x,y\}\not\subset\iso \right\}. 
$$ 
By Lemma~\ref{lemma_transversality_step}, this is a residual subset of ${\rm Met}^\ell_{\Z_k}(V_0,\theta_0)$. From now on we choose $\theta\in\mathcal{R}$ arbitrarily and fix $x,y\in\crit(f)$. We consider two different cases. \\

\noindent {\it Case 1.} $\{x,y\} \not\subset\iso$.

In this case, Lemma~\ref{lemma_transversality_step} implies directly that $W^u(x;f,\theta)$ intersects $W^s(y;f,\theta)$ transversely. \\

\noindent {\it Case 2.} $\{x,y\}\subset\iso$.

As $\theta_0$ and $\theta$ coincide near $\{x,y\}$, there are neighborhoods $U^x,U^y$ such that 
\begin{equation}\label{nbds_Ux_Uy}
\begin{aligned}
W^u(x;f,\theta_0)\cap U^x &= W^u(x;f,\theta)\cap U^x \\
W^s(y;f,\theta_0)\cap U^y &= W^s(y;f,\theta)\cap U^y.
\end{aligned}
\end{equation}

We first show that $W^s(y;f,\theta_0) = W^s(y;f,\theta)$. The inclusion $W^s(y;f,\theta_0) \subset W^s(y;f,\theta)$ is obvious because $W^s(y;f,\theta_0) \subset\iso\subset V_1$ and $\theta$ coincides with $\theta_0$ on~$V_1$, by assumption. Choose $q\in W^s(y;f,\theta)$. For large positive times the $\theta$-antigradient flow maps $q$ to $W^s(y;f,\theta)\cap U^y \subset W^s(y;f,\theta_0) \subset\iso$. By uniqueness of solutions of ODEs we conclude $q\in W^s(y;f,\theta_0)$.

Now we claim that $W^u(x;f,\theta_0)$ and $W^u(x;f,\theta)$ coincide in a neighborhood of $W^u(x;f,\theta_0)\cap W^s(y;f,\theta_0)$. Consider any point $p_0$ in $W^u(x;f,\theta_0)\cap W^s(y;f,\theta_0)$. Then $\phi^t_{f,\theta}(p_0)=\phi^t_{f,\theta_0}(p_0)$ for all $ t\in\R$. One finds $T\gg1$ such that $\phi^{-T}_{f,\theta}(p_0) \in U^x$. By continuity of the flow, there is a neighborhood $N$ of $p_0$ in $W^u(x;f,\theta)$ such that $\phi^{-T}_{f,\theta}(q) \in U^x\cap W^u(x;f,\theta)$ and $\{\phi^t_{f,\theta}(q)\}_{t\in[-T,0]} \subset V_1$ for all $ q\in N$. By~\eqref{nbds_Ux_Uy} we get $\phi^{-T}_{f,\theta}(q) \in U^x\cap W^u(x;f,\theta_0)$. Since $\theta$ and $\theta_0$ coincide on $V_1$ we obtain
\[
\phi^t_{f,\theta}(\phi^{-T}_{f,\theta}(q)) = \phi^t_{f,\theta_0}(\phi^{-T}_{f,\theta}(q)) \in W^u(x;f,\theta_0) \ \ \text{ for all } t\in[0,T].
\]
Evaluating at $t=T$ yields $q\in W^u(x;f,\theta_0)$. We have proved that $N\subset W^u(x;f,\theta_0)$. The desired claim follows because both $W^u(x;f,\theta_0)$ and $W^u(x;f,\theta)$ are embedded submanifolds of the same dimension.

A point $$ p_0\in W^u(x;f,\theta)\cap W^s(y;f,\theta) $$ belongs to $W^s(y;f,\theta_0)$ since we already proved that $W^s(y;f,\theta_0) = W^s(y;f,\theta)$. By assumption, $\phi^t_{f,\theta_0}(p_0) \in \iso$ for all~$t$. Hence $\phi^t_{f,\theta}(p_0) = \phi^t_{f,\theta_0}(p_0)$ for all~$t$ and $p_0 \in W^u(x;f,\theta_0)$. This shows that $$ p_0 \in W^u(x;f,\theta_0)\cap W^s(y;f,\theta_0). $$ But, by assumption, $W^u(x;f,\theta_0)$ intersects $W^s(y;f,\theta_0)$ transversely. It follows that $W^u(x;f,\theta)$ intersects $W^s(y;f,\theta)$ transversely at $p_0$ since we proved before that $W^s(y;f,\theta_0) = W^s(y;f,\theta)$ and that $W^u(x;f,\theta_0)$ coincides with $W^u(x;f,\theta)$ in a neighborhood of $W^u(x;f,\theta_0)\cap W^s(y;f,\theta_0)$. Case~2 is complete.
\qed

\section{Local invariant Morse-Smale pairs for finite-cyclic group actions}
\label{sec:MS_local}

In this section we apply the transversality results from the previous section to prove Theorem~\ref{main1}. We start with some preliminaries, and then as a first step we reduce the problem to the case of totally degenerate critical points. These are then handled using an inductive construction on the strata of the isotropy set.

\subsection{Preliminaries}

Let $(M,\theta)$ be a smooth Riemannian manifold without boundary endowed with an action of a finite cyclic group of order $m$ generated by the isometry $$ a: M \to M. $$ There is no loss of generality to assume that it is a faithful action.

Let $p\in M$ and consider $h$ the minimal positive integer such that $a^h(p)=p$. Set $k=m/h \in \Z$. Thus $a^h$ generates a $\Z/k\Z$ action of which $p$ is a fixed point. Using the exponential map, one finds an $a^h$-invariant neighborhood $U$ of $p$ such that $a^h|_U$ is conjugated to the restriction of $da^h|_p: T_pM \to T_pM$ to a sufficiently small $\theta|_p$-ball around of the origin. Thus, since our analysis can be localized around $p$, we may assume that
\[
(M,p)=(\R^N,0),
\]
that the inner-product $\theta|_0$ is the Euclidean inner product $\theta_0$ of $\R^N$, and that the action is generated by a matrix
\[
A=da^h|_0\in O(N)
\]
satisfying $A^k=I$.

Let $D$ be a closed Euclidean ball centered at $0$. The interior of $D$ will be denoted by $\dot D$. Consider a smooth $A$-invariant function
\[
f:D\to \R
\]
having $0$ as its unique critical point.

As in Subsection~\ref{subsec_props} we write $$ F_j = \ker (A^j-I). $$ These subspaces will be referred to as the {\it linear isotropy manifolds}. Before proceeding we make some simple but important remarks about them. Considering the real Jordan decomposition of $A$, we have a splitting
\[
\R^N = \bigoplus_{\lambda^k=1} E_\lambda
\]
where $E_\lambda$ is the real generalized eigenspace of $\lambda$ if $\lambda$ is in the spectrum of $A$, or the trivial vector space if not. Since $A \in O(N)$ this is a $\theta_0$-orthogonal decomposition into $A$-invariant subspaces, where each $E_\lambda$ can be further orthogonally decomposed into $A$-invariant subspaces $E_\lambda = \oplus W$ as follows:
\begin{itemize}
\item If $\lambda \not\in \R$ then each $W$ satisfies $\dim W=2$, the $A$-action on $W$ is linearly conjugated to the action on $\C$ by multiplication by $\lambda$. 
\item If $\lambda \in \{1,-1\}$ then each $W$ satisfies $\dim W=1$, and either $A|_W$ is the identity or $A|_W$ is minus the identity.
\end{itemize}
In other words, the decomposition into the $E_\lambda$'s is the decomposition into {\it isotypical components}. For $j\in\div(k)$ note that
\begin{equation}
F_j = \bigoplus_{\lambda^j=1} E_\lambda
\end{equation}
from where we recover Lemma~\ref{lemma_intersection}: $F_j \cap F_i = F_{\gcd(i,j)}$. This can be used to prove

\begin{lemma}\label{lemma_relposition_istropy}
Let $i,j\in \div(k)$ and let $H$ be the $\theta_0$-orthogonal of $F_{\gcd(i,j)}$ inside $F_i$. Then $H \subset F_j^\bot$.
\end{lemma}

\begin{proof}
Setting $d=\gcd(i,j)$ this follows from formula $H=\bigoplus_{\lambda^i=1,\lambda^d\neq1} E_\lambda$.
\end{proof}

\subsection{Reduction to the totally degenerate case}

We explain why is it sufficient to prove Theorem~\ref{main1} in the case $p$ is a totally degenerate critical point.

Notice that, since $f$ is $A$-invariant and $A$ is a Euclidean isometry, the splitting $$ \R^N = \ker D^2f(0) \oplus (\ker D^2f(0))^\perp $$ is preserved by both $D^2f(0)$ and $A$. Here $(\ker D^2f(0))^\perp$ denotes the orthogonal complement of $\ker D^2f(0)$ with respect to the Euclidean metric.

Up to a transformation in $O(N)$ we may assume, without any loss of generality, that $N=N_1+N_2$, $\R^N \simeq \R^{N_1} \times \R^{N_2}$ and
\begin{equation*}
\begin{aligned}
\R^{N_1} \simeq \R^{N_1} \times 0 &= \ker D^2f(0,0) \\
\R^{N_2} \simeq 0 \times \R^{N_2} &= (\ker D^2f(0,0))^\perp
\end{aligned}
\end{equation*}
By our prior arguments, $A$ assumes the form
\begin{equation*}
A = \diag (A_1,A_2) \qquad A_j \in O(N_j)
\end{equation*}
In particular, $A_j^k$ is the identity in $\R^{N_j}$, i.e. $A_j$ generates a $\Z_k$-action on $\R^{N_j}$ by Euclidean isometries. By Lemma~\ref{lemma_invariant_GM_splitting} we find an embedding $$ \Psi:U \to \R^{N_1} \times \R^{N_2} $$ defined on some $A$-invariant neighborhood $U$ of $(0,0)$ such that:
\begin{itemize}
\item $\Psi$ is $A$-equivariant.
\item $\Psi(0,0)=(0,0)$ and $D\Psi(0,0)=I$.
\item $f\circ\Psi(z_1,z_2) = g(z_1) + h(z_2)$ where $0\in\R^{N_1}$ is a totally degenerate critical point of the $A_1$-invariant function $g$, and $0\in\R^{N_2}$ is a non-degenerate critical point of the $A_2$-invariant function $h$.
\end{itemize}
Thus, it suffices to prove Theorem~\ref{main1} for the isolated critical point $0\in\R^{N_1}$ of~$g$.

\subsection{The totally degenerate case}

Total degeneracy means that $D^2f(0)=0$. In this subsection we proceed assuming total degeneracy. With $d\in\div(k)$ we set
\begin{equation}\label{set_G_d}
G_d := \bigcup_{j\in\div(k),j\leq d} F_j.
\end{equation}
We will prove inductively the following family of claims indexed by $d\in\div(k)$:
\begin{itemize}
\item[($C_d$)] For every $\epsilon>0$ there is an $A$-invariant open neighborhood $V_d$ of $G_d$ and an $A$-invariant pair $(f_d,\theta_d)$ defined on $D$ with the following properties:
\begin{itemize}
\item[(i)] $f_d|_{V_d \cap D}$ is Morse and $\crit(f_d) \cap V_d \subset G_d \cap \dot D$.
\item[(ii)] $j\in \div(k), \ j\leq d, \ x\in \crit(f_d) \cap F_j \Rightarrow W^s(x;f_d,\theta_d) \subset F_j \cap \dot D$. 
\item[(iii)] $x,y\in\crit(f_d) \cap V_d \Rightarrow W^u(x;f_d,\theta_d) \pitchfork W^s(y;f_d,\theta_d)$.
\item[(iv)] $(f_d,\theta_d)$ is $\epsilon$-close to $(f,\theta)$ in $C^2(D)$.
\end{itemize}
\end{itemize}
In (ii) and (iii) stable and unstable manifolds are taken with respect to the open manifold $\dot D$. The desired conclusion follows from ($C_k$).

Before we give the details of the proof, let us give an outline. The initial step is to obtain the Morse-Smale condition on the fixed point set near the critical point. Since the action is trivial there, standard arguments ensure transversality; this is the content of Lemma~\ref{lemma_auxiliary_inital_step}. Using the total degenericity, the argument explained immediately after Lemma~\ref{lemma_auxiliary_inital_step} shows that we can create the necessary perturbation on the fixed point set of the $\Z_k$-action in such a way that the Hessian is negative definite (and small) in normal directions. Hence stable manifolds of critical points in the fixed point set are contained in the fixed point set.

The induction to prove claims $(C_d)$, where $d$ ranges over the divisors of $k$, is as follows. Let $d'<d$ be consecutive divisors of $k$. One does not touch the pair $(f_{d'},\theta_{d'})$ on the neighborhood $V_{d'}$ of $G_{d'}$. The induced $\Z_d$ action on $F_d$ has isotropy set $F_d \cap G_{d'}$, so we can achieve the Morse-Smale condition on $F_d$ by a small perturbation supported on the free part $F_d\setminus G_{d'}$ of this action. Here is where the analysis of Subsection~\ref{ssec_transv_lemma} plays a role. This does not modify what we already have on $V_{d'}$. Regularized distance functions now come into play to keep the perturbation on $F_d$ obtained so far and simultaneously create negative Hessian in directions normal to $F_d$ at critical points in $F_d\setminus G_{d'}$. Lemma~\ref{lemma_crucial4} plays a key role to compare transversality in $F_d$ with transversality in $\dot D$, and to ensure that stable manifolds of critical points in $F_d\setminus G_{d'}$ are contained in $F_d$. This concludes the idea of the proof.

Now we give the details of the proof. Throughout the argument below, norms of vectors and tensors, as well as distances between points and sets are measured with respect to the Euclidean metric, not to be confused with the Riemannian metric $\theta$ on $D$. We use the notation $\bot$ to denote Euclidean orthogonal complements of subspaces. We denote by $P_j:\R^N \to \R^N$ the Euclidean orthogonal linear projection onto $F_j$, and by $\iota_j:F_j \hookrightarrow \R^N$ the inclusion. Note that 
\begin{equation}\label{projections_commute_with_A}
P_jA=AP_j, \qquad \qquad \iota_jA=A\iota_j.
\end{equation}

\subsubsection{Starting the induction}

Here we prove ($C_1$). Let $c_1>0$ be fixed arbitrarily.

\begin{lemma}\label{lemma_auxiliary_inital_step}
For every $\delta_1>0$ we find a smooth function $h_1:F_1 \to \R$ and a smooth symmetric tensor $\lambda_1:F_1 \to F_1^* \otimes F_1^*$ with the following properties.
\begin{itemize}
\item[I.] The data $(h_1,\lambda_1)$ is compactly supported in $F_1 \cap \dot D$ and is $\delta_1$-close to $(0,0)$ in the $C^2(F_1\cap D)$-topology.
\item[II.] The pair $(\iota_1^*f+h_1,\iota_1^*\theta+\lambda_1)$ is Morse-Smale on $F_1 \cap \dot D$.
\item[III.] Any critical point $x$ of $f+h_1\circ P_1$ belongs to $\frac{1}{2}D$ and satisfies 
\[
|D^2(f+h_1\circ P_1)(x)|\leq \frac{c_1}{10}. 
\]
\end{itemize}
\end{lemma}

\begin{proof}
Properties I and II follow from the usual construction of local Morse homology explained in Section~\ref{sec_properties}. The necessary transversality results are, in fact, contained as a special case of the results in Section~\ref{sec_prelim_transv} for the trivial group action (empty isotropy). Let us give more details.

Consider the smooth compact manifold with boundary $F_1 \cap D$ equipped with the pair $(\iota_1^*f,\iota_1^*\theta)$. Since the $\theta$-gradient of $f$ is tangent to $F_1 \cap D$, the origin is the unique critical point of $\iota_1^*f$. It follows that we can choose $h_1$ $C^2$-small and supported near the origin in such a way that $\iota_1^*f+h_1$ is Morse in $F_1\cap D$ and all its critical points are very close to the origin. Using the results from Section~\ref{sec_prelim_transv}, with a trivial group action, we find $\lambda_1$ compactly supported in $F_1\cap\dot D$ and arbitrarily $C^2$-small, such that the Morse-Smale condition (Definition~\ref{def_MS}) is satisfied by the pair $(\iota_1^*f+h_1,\iota_1^*\theta+\lambda_1)$ in the smooth manifold without boundary $F_1\cap\dot D$.

As was remarked above, by the $C^2$-smallness of $h_1$ we can be sure that all critical points of $f+P_1^*h_1$ are contained in $\dot D$ and lie very close to the origin. Moreover, by the total degeneracy assumption on the unperturbed data, we can be sure that the Hessian of $f+P_1^*h_1$ at its critical points is very small, and can be made arbitrarily small if $\delta_1$ are small enough. Thus, property III can be achieved if $\delta_1$ is small enough.
\end{proof}

We now show the claim of ($C_1$). Consider the $A$-invariant pair $(f_1,\theta_1)$ defined by 
\begin{equation}\label{defining_f_1}
\begin{array}{ccc}
f_1(x) = f(x) + h_1\circ P_1(x) - \frac{c_1}{2}|(I-P_1)x|^2 & & \theta_1 = \theta + P_1^*\lambda_1
\end{array}.
\end{equation}
Note that $\crit(\iota_1^*f+h_1) = \crit(f_1) \cap F_1$ because both $f+h_1\circ P_1$ and $f_1$ are $A$-invariant, their gradients with respect to $A$-invariant Riemannian metrics must be tangent to $F_1$ at points of $F_1$, and coincide over $F_1$; see Lemma~\ref{lemma_grad_tang}. The bilinear form $D^2f_1(x)$ is negative definite on $F_1^\bot$ whenever $x\in\crit(f_1)\cap F_1$, in fact, at such a critical point we have an estimate
\begin{equation}\label{estimate_normal_hessian_C_1}
\begin{aligned}
D^2f_1(x)(u,u) &= D^2(f+h_1\circ P_1)(x)(u,u) - c_1|u|^2 \\
& \leq \frac{c_1}{10}|u|^2 - c_1|u|^2 = -\frac{9}{10}c_1|u|^2 \ \ \ \forall u \in F_1^\bot.
\end{aligned}
\end{equation}
Thus every $x\in\crit(f_1) \cap F_1$ is a non-degenerate critical point of $f_1$ and 
\[
W^s(x;f_1,\theta_1) \subset F_1 \cap \dot D
\]
in view of (iii) in Lemma~\ref{lemma_crucial4}. We have shown that (ii) in ($C_1$) holds for every $x\in \crit(f_1) \cap F_1$. 

Obviously the $C^2(D)$-norm of the difference between $(f,\theta)$ and $(f_1,\theta_1)$ is bounded from above in terms of $\delta_1,c_1$, and hence it can be made smaller than any $\epsilon$ because $\delta_1$ and $c_1$ can be taken arbitrarily small. In other words (iv) in ($C_1$) holds.

For $r>0$ consider $V_1 = \{x : |(I-P_1)x|<r\}$. Since all critical points in $\crit(f_1) \cap F_1$ are non-degenerate, they must be finite in number because they are contained in $\frac{1}{2}D$. Hence if $r$ is small enough then $V_1$ is a neighborhood of $F_1$ for which (i) in ($C_1$) holds. Finally, (iii) in ($C_1$) holds in view of II above and of item (ii) in Lemma~\ref{lemma_crucial4}.

\subsubsection{The inductive step}

Let $d'<d$ be two consecutive divisors of $k$. We assume by induction that ($C_{d'}$) holds, i.e., for all $\epsilon'>0$, there exist $V_{d'}$ and $(f_{d'},g_{d'})$ satisfying (i)-(iv) in ($C_{d'}$), where $\epsilon'$ is the quantifier in (iv): $(f_{d'},g_{d'})$ is $\epsilon'$-close to $(f,g)$ in the $C^2(D)$-topology. We always consider $\epsilon'$ small enough so that $\crit(f_{d'}) \subset \dot D$. 

\begin{remark}\label{rmk_tangencies_isotropy_strata}
Throughout it is important to keep in mind that the gradient of any $A$-invariant pair on $D$ or $\dot D$ is necessarily tangent to $F_j$, for all $j\in\div(k)$.
\end{remark}

By (i) in ($C_{d'}$) we have $\crit(f_{d'}) \setminus G_{d'} \subset \R^N \setminus V_{d'}$. The set
\[
\Omega := (F_d \cap \dot D) \setminus G_{d'}
\]
is an $A$-invariant open subset of $F_d$, and $\crit(f_{d'}) \cap \Omega$ is a compact subset of $F_d \setminus V_{d'}$. Note that $A$ induces a $\Z_d$-action on $F_d \cap \dot D$ which is free on $\Omega$ because the isotropy set of this action is precisely $G_{d'}\cap F_d \cap\dot D$.

On the quotient $\Omega/\Z_d$ the function $f_{d'}$ induces a smooth function $\widehat f_{d'}$ which has a compact critical set $(\crit(f_{d'}) \cap \Omega)/\Z_d$. Here we used that the gradient of $f_{d'}$ with respect to an $A$-invariant metric must be tangent to $F_j$, see Remark~\ref{rmk_tangencies_isotropy_strata}. In particular it follows that $\crit(f_{d'}) \cap \Omega=\crit(f_{d'}|_{\Omega})$. Hence there exists an arbitrarily $C^2$-small function $\widehat\alpha:\Omega/\Z_d \to \R$ supported on an arbitrarily small neighborhood of the compact set $(\crit(f_{d'}) \cap \Omega)/\Z_d$ such that $\widehat f_{d'} + \widehat\alpha$ is Morse on $\Omega/\Z_d$ and has finitely many critical points there. Pulling $\widehat\alpha$ back to $\Omega$ we obtain an $A$-invariant function $\alpha:F_d \cap \dot D\to\R$ supported near $\crit(f_{d'}) \cap \Omega$ such that 
\begin{equation}\label{Morse_perturbation_1}
f_{d'}|_{F_d\cap \dot D}+\alpha
\end{equation}
is a Morse function on~$\Omega$. 

The function~\eqref{Morse_perturbation_1}
is Morse on $V_{d'}\cap F_d\cap\dot D$ because $f_{d'}$ is assumed to be Morse on $V_{d'}\cap\dot D$ and the $\theta_{d'}$-gradient of $f_{d'}$ is tangent to $F_d$, see Remark~\ref{rmk_tangencies_isotropy_strata}. Moreover, $F_d\cap\dot D$ can be covered by two relatively open sets
\[
F_d \cap \dot D = \Omega \cup (V_{d'} \cap F_d \cap \dot D).
\]
It follows that~\eqref{Morse_perturbation_1} 
is Morse on $F_d \cap \dot D$ with finitely many critical points, all of which lie close to the origin. Shrinking $V_{d'}$ we may also assume that $\alpha$ vanishes on $\overline{V_{d'}} \cap F_d\cap D$. Hence
\begin{equation}\label{nice_critical_position}
\overline{V_{d'}} \cap F_d \cap \dot D \cap \crit(f_{d'}|_{F_d\cap\dot D}+\alpha) \subset G_{d'} \cap F_d \cap \dot D.
\end{equation}
Once again we used $A$-invariance and Remark~\ref{rmk_tangencies_isotropy_strata}.

The next and important step is to apply the Transversality Lemma~\ref{lemma_crucial3} to the open Riemannian manifold $(F_d \cap \dot D,\iota_d^*\theta_{d'})$ equipped with the $\Z_d$-action by isometries induced by $A|_{F_d}$, and the $\Z_d$-invariant Morse function~\eqref{Morse_perturbation_1}.
Note that $A$ generates a $\Z_d$-action on $F_d\cap\dot D$ with isotropy set $G_{d'}\cap F_d\cap\dot D$. Note also that $\alpha$ vanishes on $\overline{V_{d'}} \cap F_d\cap D$. Note that $f_{d'}|_{F_d\cap \dot D}+\alpha$ has finitely many critical points because this is a Morse function with a compact critical set.

Consider any
\[
x \in \crit(f_{d'}|_{F_d \cap\dot D}+\alpha) \cap F_j \cap F_d \cap\dot D = \crit(f_{d'}|_{F_d \cap\dot D}) \cap F_j \cap F_d \cap\dot D
\]
where $j\leq d'$ is some divisor of $k$. To simplify the notation we write
\[
\begin{aligned}
& W^s(x) = W^s(x;f_{d'}|_{F_d \cap\dot D},\iota_d^*\theta_{d'}) \\ 
& W^s_\alpha(x) = W^s(x;f_{d'}|_{F_d \cap\dot D}+\alpha,\iota_d^*\theta_{d'})
\end{aligned}
\]
where stable manifolds are taken with respect to the open manifold $F_d \cap\dot D$.

We need to check conditions (i) and (ii) of Lemma~\ref{lemma_crucial3}. To check (i) it suffices to prove that
\[
W^s_\alpha(x) \subset F_j\cap F_d\cap \dot D.
\]
Let 
\begin{equation*}
\text{$\phi_\alpha^t$ be the anti-gradient flow of $(f_{d'}|_{F_d \cap\dot D}+\alpha,\iota_d^*\theta_{d'})$ on $F_d\cap \dot D$}
\end{equation*}
and
\begin{equation*}
\text{$\phi^t$ be the anti-gradient flow of $(f_{d'}|_{F_d \cap\dot D},\iota_d^*\theta_{d'})$ on $F_d\cap \dot D$.}
\end{equation*}
Let $y\in W^s_\alpha(x)$. This means that $\phi_\alpha^t(y)$ is defined for $t\in[0,+\infty)$ and $\phi_\alpha^t(y) \to x$ as $t\to+\infty$. Since $\alpha$ vanishes on a neighborhood of $x$, we get that $W^s(x)$ coincides with $W^s_\alpha(x)$ near $x$. Hence there exists $\tau$ large such that
\[
\phi_\alpha^\tau(y) \in W^s(x) \subset W^s(x;f_{d'},\theta_{d'}) \subset F_j \cap F_d \cap\dot D.
\]
But $(f_{d'}|_{F_d\cap \dot D}+\alpha,\iota_d^*\theta_{d'})$ is $A$-invariant, in particular, the flow $\phi_\alpha^t$ leaves $F_j\cap F_d\cap \dot D$ invariant. Thus $y = \phi_\alpha^{-\tau} (\phi_\alpha^\tau(y)) \in F_j \cap F_d \cap \dot D$. We are done checking (i) in Lemma~\ref{lemma_crucial3}.

We next check the condition (ii). Let
\[
\begin{aligned} 
x,y \in & \ \crit(f_{d'}|_{F_d \cap \dot D}+\alpha)\cap G_{d'}\cap F_d\cap \dot D \\ 
& = \crit(f_{d'}|_{F_d \cap \dot D})\cap G_{d'}\cap F_d\cap \dot D \\ 
& \subset \crit(f_{d'}) \cap V_{d'} .
\end{aligned} 
\]
The last inclusion uses tangency to $F_d$ of the gradient of $f_{d'}$ with respect to invariant metrics, see Remark~\ref{rmk_tangencies_isotropy_strata}. We simplify notation by writing
\[
\begin{aligned}
& W^s(y) = W^s(y;f_{d'}|_{F_d \cap\dot D},\iota_d^*\theta_{d'}) \\ 
& W^s_\alpha(y) = W^s(y;f_{d'}|_{F_d \cap\dot D}+\alpha,\iota_d^*\theta_{d'}) \\
& W^u(x) = W^u(x;f_{d'}|_{F_d \cap\dot D},\iota_d^*\theta_{d'}) \\ 
& W^u_\alpha(x) = W^u(x;f_{d'}|_{F_d \cap\dot D}+\alpha,\iota_d^*\theta_{d'})
\end{aligned}
\]
and again we let $\phi_\alpha^t$ denote the anti-gradient flow of $f_{d'}|_{F_d \cap\dot D}+\alpha$ on $F_d\cap \dot D$, and $\phi^t$ denote the anti-gradient flow of $f_{d'}|_{F_d \cap\dot D}$ on $F_d\cap \dot D$, both taken with respect to the metric $\iota_d^*\theta_{d'}$.

We claim that
\begin{equation}\label{stable_alpha_coincide}
W^s(y) = W^s_\alpha(y).
\end{equation}
Note $W^s(y) \subset W^s(y;f_{d'},\theta_{d'})\subset F_m$ for some $m\leq d'$. Hence $W^s(y) \subset W^s_\alpha(y)$ since $\alpha$ vanishes near $F_m$. To prove the other inclusion, let $q\in W^s_\alpha(y)$. Taking $T$ large enough, $\phi^{T}_\alpha(q)$ gets very close to $y$. The manifolds $W^s_\alpha(y)$ and $W^s(y)$ coincide near $y$ since $\alpha$ vanishes near $y$. Hence $\phi^{T}_\alpha(q) \in W^s(y)$ if $T\gg1$. Then also $\phi^t(\phi^{T}_\alpha(q)) \in W^s(y) \subset F_m$ for all $t$ for which this is well-defined. But by uniqueness this coincides with $\phi^t_\alpha(\phi^{T}_\alpha(q))$ since $\alpha$ vanishes on $F_m$, hence this curve is defined for $t\in[-T,0]$. Plugging $t=-T$ we obtain $q\in W^s(y)$. This proves~\eqref{stable_alpha_coincide}.

Now we claim that
\begin{equation}\label{unstable_alpha_coincide_locally}
\begin{aligned}
& W^u(x) \cap U = W^u_\alpha(x) \cap U \\
& \text{for some neighborhood $U$ of $W^u(x)\cap W^s(y)$ in $F_d\cap \dot D$.}
\end{aligned}
\end{equation}
In fact, let $p\in W^u(x) \cap W^s(y)$. There exists a neighborhood $U^x$ of $x$ in $F_d\cap \dot D$ such that $W^u(x) \cap U^x = W^u_\alpha(x) \cap U^x$, since $\alpha$ vanishes near $x$. By continuity of the flow there exists $T\gg1$ and a neighborhood $W$ of $p$ in $W^u(x)$ such that $\phi^{-T}(W)\subset U^x$. Hence $\phi^{-T}(W)\subset W^u_\alpha(x)$. The trajectory $\phi^t(p)$, which is well-defined for all $t\in\R$, is contained in $G_{d'}$ since $W^s(y) \subset W^s(y;f_{d'},\theta_{d'})\subset F_m$ for some $m\leq d'$; this follows from (ii) in ($C_{d'}$). Since $\alpha$ vanishes on $V_{d'}$ we can invoke uniqueness to conclude that $\phi^t(p)=\phi^t_\alpha(p)$ for all~$t$. Thus, after further shrinking $W$, we may assume that for all $t\in[0,T]$ $\phi^t_\alpha$ is well-defined on $\phi^{-T}(W)$ and $\phi^t_\alpha(\phi^{-T}(W))\subset V_{d'}$. Since $\alpha$ vanishes on $V_{d'}$ we see that $\phi^t(z)=\phi^t_\alpha(z)$ for all $t\in[0,T]$ and $z\in \phi^{-T}(W)$, by uniqueness of solutions. Plugging $t=T$ we get
\[
W = \phi^T(\phi^{-T}(W)) = \phi^T_\alpha(\phi^{-T}(W))\subset W^u_\alpha(x).
\]
This concludes the proof of~\eqref{unstable_alpha_coincide_locally}.

Finally we note that 
\begin{equation}\label{unstable_stable_transverse_F_d}
W^u(x) \pitchfork W^s(y)
\end{equation}
follows from (iii) in ($C_{d'}$) together with (i) in Lemma~\ref{lemma_crucial4}.

Now we can finally explain why (ii) in Lemma~\ref{lemma_crucial3} follows from~\eqref{stable_alpha_coincide},~\eqref{unstable_alpha_coincide_locally} and~\eqref{unstable_stable_transverse_F_d}. Namely, we claim that
\begin{equation*}
W^u_\alpha(x) \pitchfork W^s_\alpha(y).
\end{equation*}
To see this consider a point $p$ in $W^u_\alpha(x) \cap W^s_\alpha(y)$. By~\eqref{stable_alpha_coincide} the latter set is equal to $W^u_\alpha(x) \cap W^s(y)$. The trajectory of $p$ is contained in $G_{d'}$ since $W^s(y) \subset G_{d'}$. It follows that $\phi^t(p)=\phi^t_\alpha(p)$ for all $t$. Hence $p\in W^u(x)$. We have shown that $p\in W^u(x) \cap W^s(y)$. It follows from~\eqref{unstable_alpha_coincide_locally} that locally near $p$ the manifolds $W^u(x)$ and $W^u_\alpha(x)$ coincide. From~\eqref{unstable_stable_transverse_F_d} we conclude that $p$ is a point where $W^u_\alpha(x)$ and $W^s_\alpha(y)$ meet transversely. 

As we have now verified the conditions of Lemma~\ref{lemma_crucial3}, we can now use it to find an $A$-invariant symmetric smooth tensor $\lambda:F_d\cap D \to F_d^* \otimes F_d^*$, compactly supported on $F_d\cap\dot D$ and  with an arbitrarily small $C^2$-norm, such that the pair $$ (f_{d'}|_{F_d\cap\dot D}+\alpha,\iota_d^*\theta_{d'}+\lambda) $$ is Morse-Smale on $F_d \cap \dot D$. In other words, the {\it preliminary pair}
\begin{equation}\label{preliminary_MS_pair}
(f_{d'}+P_d^*\alpha,\theta_{d'}+P_d^*\lambda)
\end{equation}
is $A$-invariant and restricts to a Morse-Smale pair on $F_d\cap \dot D$.
Lemma~\ref{lemma_crucial3} allows us to pick $\lambda$ in such a way that its support does not intersect $\overline{V_{d'}} \cap F_d \cap \dot D$.

One crucial remark that needs to be made at this point is that~\eqref{preliminary_MS_pair} keeps all properties (i)-(iv) of ($C_{d'}$). This is because it coincides with the pair $(f_{d'},\theta_{d'})$ on $V_{d'}$, perhaps after shrinking $V_{d'}$. In fact, $z\in G_{d'} \cap \dot D \Rightarrow P_dz\in G_{d'}\cap F_d\cap \dot D$ by Lemma~\ref{lemma_relposition_istropy}.

By Theorem~\ref{thm_existence_regularized_distance} we can find $A$-invariant regularized distance functions:
\begin{align*}
& \delta_0 \ \text{ the $A$-invariant regularized distance to } \ \overline{V_{d'}} \\
& \delta_1 \ \text{ the $A$-invariant regularized distance to } \ \overline{V_{d'}} \cup F_d.
\end{align*}
These functions are continuous and defined on all of $\R^N$, however $\delta_0$ is smooth on $\R^N\setminus \overline{V_{d'}}$ and $\delta_1$ is smooth on $\R^N\setminus(\overline{V_{d'}} \cup F_d)$. In fact, Theorem~\ref{thm_existence_regularized_distance} provides these distance functions without any mention to $A$-invariance, but then we can average over the group to obtain $A$-invariance. By the properties of regularized distance functions described in Theorem~\ref{thm_existence_regularized_distance}, there exists $M>0$ depending only on $N$ such that
\begin{equation}\label{ineq_derivatives_delta_01}
\begin{aligned}
& |\nabla\delta_0(x)| \leq M, \ |\nabla^2\delta_0(x)| \leq M/\delta_0(x) \ \ \forall x\in \R^N\setminus \overline{V_{d'}} \\
& |\nabla\delta_1(x)| \leq M, \ |\nabla^2\delta_1(x)| \leq M/\delta_1(x) \ \ \forall x\in \R^N\setminus (\overline{V_{d'}}\cup F_d). \\
\end{aligned}
\end{equation}
It follows that the equations
\begin{equation}\label{ineq_derivatives_delta_01_square}
\begin{aligned}
& |\nabla(\delta_i^2)|=|2\delta_i\nabla\delta_i| \leq 2M \delta_i \\
& |\nabla^2(\delta_i^2)|=|2\nabla\delta_i\otimes\nabla\delta_i+2\delta_i\nabla^2\delta_i| \leq 2M^2+2M
\end{aligned}
\end{equation}
hold pointwise in the respective domains. Moreover we have
\begin{equation}\label{ineq_delta_0_delta_1}
\delta_1 \leq c \ \dist(\cdot,\overline{V_{d'}} \cup F_d) \leq c \ \dist(\cdot,\overline{V_{d'}}) \leq c' \delta_0
\end{equation}
with suitable constants $c,c'>0$ depending only on $N$. In Appendix~\ref{app_reg_dist_functions} we prove Proposition~\ref{prop_refinement_thm_reg_function} which states that we can assume that
\begin{equation}\label{property_delta_1}
\delta_1 \ \text{ agrees with } \ \dist(\cdot,F_d) \ \text{ on a neighborhood of } \ F_d \setminus \overline{V_{d'}}.
\end{equation}

Now let us consider $\phi_\delta(s)=\phi(s/\delta)$ where $\phi:\R\to[0,1]$ vanishes identically near $(-\infty,0]$ and takes the constant value $1$ near $[1,+\infty)$. Consider the family of functions 
\[
g_\delta = \phi_\delta(\delta_0^2)\delta_1^2.
\]
Note that $g_\delta$ defines a smooth function on $\R^N$ in view of~\eqref{ineq_delta_0_delta_1} and~\eqref{property_delta_1}. Here smoothness of $\dist(\cdot,F_d)^2$ on $\R^N$ was used.

We claim that on points of $D$ all derivatives  of $g_\delta$ of order at most $2$ can be bounded uniformly and independently of $\delta$. Let us prove this claim. Setting $$ B := \max_{i=0,1} \{\|\delta_i\|_{L^\infty(D)}\} $$ we compute derivatives of $g_\delta$ using~\eqref{ineq_derivatives_delta_01} and~\eqref{ineq_derivatives_delta_01_square}. Note that $\phi_\delta'(\delta_0^2)\neq0 \Rightarrow \delta_0\leq \sqrt{\delta}$ since the support of $\phi'_\delta$ is contained in $[0,\delta]$. Using this and the previous estimates, for the first order derivatives we get
\[
\nabla g_\delta = \phi'_\delta(\delta_0^2) [\nabla (\delta_0^2)] \delta_1^2 + \phi_\delta(\delta_0^2)\nabla(\delta_1^2),
\]
which, combined with~\eqref{ineq_delta_0_delta_1}, implies that
\[
\|\nabla g_\delta\|_{L^\infty(D)} \leq \delta^{-1}\|\phi'\|_{L^\infty}2BM (c'\delta_0)^2 + 2BM \leq 2BM(1+|c'|^2\|\phi'\|_{L^\infty}).
\]
For second derivatives, we estimate similarly
\[
\begin{aligned}
\nabla^2g_\delta &= \phi''_\delta(\delta_0^2)[\nabla(\delta_0^2) \otimes \nabla(\delta_0^2)] \delta_1^2 \\
&+ \phi'_\delta(\delta_0^2) [\nabla^2(\delta_0^2)] \delta_1^2 + \phi'_\delta(\delta_0^2) \nabla(\delta_1^2)\otimes\nabla(\delta_0^2) \\
&+ \phi'_\delta(\delta_0^2) \nabla(\delta_0^2) \otimes \nabla(\delta_1^2) + \phi_\delta(\delta_0^2) \nabla^2(\delta_1^2) .
\end{aligned}
\]
Again this implies
\[
\begin{aligned}\|\nabla^2g_\delta\|_{L^\infty(D)} &\leq \|\phi''\|_{L^\infty}\delta^{-2} 2M\sqrt{\delta}2M\sqrt{\delta} (c')^2\delta + \|\phi'\|_{L^\infty}\delta^{-1}(2M^2+2M) (c')^2\delta \\
&+ \|\phi'\|_{L^\infty}\delta^{-1}2M\sqrt{\delta}\ 2Mc' \sqrt{\delta} + \|\phi'\|_{L^\infty}\delta^{-1}2M\sqrt{\delta}\ 2Mc'\sqrt{\delta} \\
&+ 2M^2+2M.
\end{aligned}
\]
Since powers of $\delta$ cancel in the estimate, we get the desired conclusion.

We are finally ready to conclude our induction step. Let $c_d>0$ and $\nu \in (0,1)$ be fixed arbitrarily. Then the pair~\eqref{preliminary_MS_pair} with all the properties established so far can be taken $C^2$-close enough to $(f_0,\theta_0)$ in such a way that all critical points $x$ lie very close to the origin, and the hessian of $f_{d'}+P_d^*\alpha$ at these points satisfies
\[
\begin{array}{ccc} |D^2(f_{d'}+P_d^*\alpha)(x)| \leq \nu c_d & & \|\lambda\|_{L^\infty} \leq\nu \end{array}.
\]
The choice of $(\alpha,\lambda)$ for this to be true depends on $(c_d,\nu)$. Now define the pair $(f_d,\theta_d)$ by 
\begin{equation}\label{final_pair_f_d_theta_d}
\begin{array}{ccc}
\theta_d := \theta_{d'} + P_d^*\lambda & & f_d := f_{d'}+P_d^*\alpha - \frac{c_d}{2} g_\delta
\end{array}
\end{equation}
where $\delta>0$ is chosen in such a way that $\phi_\delta(\delta_0^2)$ is equal to $1$ near critical points of $f_{d'}|_{F_d\cap \dot D}+\alpha$ belonging to $(F_d\cap \dot D)\setminus\overline{V_{d'}}$. Such a $\delta$ exists in view of~\eqref{nice_critical_position}. Note that the ratio $\nu/c_d$ can be chosen arbitrarily small. \\

This completes the construction of the inductive step and it remains to check the claimed properties that there exists an open $A$-invariant neighborhood $V_d$ of $G_d$ such that $V_d$ and the pair $(f_d,\theta_d)$ satisfy all properties (i)-(iv) of ($C_d$), with $\epsilon$ arbitrarily small. Let us verify these properties.

\subsubsection*{Property (iv) of claim ($C_d$)}

Note that (iv) in $(C_d)$ is clear since $\nu$ and $c_d$ can be chosen arbitrarily small as a consequence of our estimates on the $C^2$-norm of $g_\delta$ (derivatives up to second order of $c_dg_\delta$ can be bounded in terms of $c_d$ independently of~$\delta$).

\subsubsection*{Property (i) of claim ($C_d$)}

Note that $(f_d,\theta_d)$ coincides with $(f_{d'},\theta_{d'})$ near $G_{d'}$. In fact, $g_\delta$ vanishes on $\cl{V_{d'}}$, so $f_d$ coincides with $f_{d'}+P_d^*\alpha$ near $G_{d'}$. Let $z\in D$, and let $w\in G_{d'}$ satisfy $\dist(z,G_{d'})=|z-w|$. There exists $j\leq d'$ such that $w\in F_j$. Since $P_d$ commutes with $A^j$ we conclude that $P_dw\in F_j\cap F_d$. Since $$ |P_dz-P_dw|\leq |z-w|=\dist(z,G_{d'}) $$ we conclude that if $\dist(z,G_{d'})$ is small enough then $P_dz\in \cl{V_{d'}} \cap F_d \cap D$. It follows that $P_d^*\alpha$ and $P_d^*\lambda$ vanish at $z$, as we wanted to show.

If $\nu/c_d$ is small enough and $x$ is a critical point of $f_d$ in $(F_d\cap \dot D) \setminus \overline{V_{d'}}$ then the Hessian $D^2f_d(x)$ is negative-definite along the $\theta_d(x)$-orthogonal of $F_d$ at $x$:
\begin{equation}\label{neg_def_along_orthogonal}
\begin{aligned}
&x \in (\crit(f_d) \cap F_d\cap \dot D) \setminus \overline{V_{d'}} = \crit(f_{d'}|_{F_d\cap \dot D}+\alpha) \setminus \overline{V_{d'}}, \ u \bot F_d, \ u\neq0 \\
&\Rightarrow \ D^2f_d(x)(u,u) <0 .
\end{aligned}
\end{equation}
Here we strongly used that $\delta_1^2$ coincides with $\dist(\cdot,F_d)^2$ near $F_d\setminus \overline{V_{d'}}$. It also follows that critical points on $F_d \setminus \overline{V_{d'}}$ are non-degenerate. We can take 
\[
V_d := V_{d'} \cup \{y\in\R^N : |(I-P_d)y|<r\}
\]
with some $r>0$ small so that critical points on $V_d\setminus \overline{V_{d'}}$ must lie on $F_d$. Critical points of $f_d$ which lie on $\cl{V_{d'}}$ actually lie on $G_{d'}$ by construction since, as shown above, $f_d$ coincides with $f_{d'}$ near $G_{d'}$. This proves that (i) in ($C_d$) holds.

\subsubsection*{Property (ii) of claim ($C_d$)}

In the following, we will denote by $\phi_{d'}^t$ and $\phi^t_d$ the antigradient flows of $(f_{d'},\theta_{d'})$ and $(f_d,\theta_d)$, respectively. By (iii) in Lemma~\ref{lemma_crucial4} and~\eqref{neg_def_along_orthogonal} $W^s(y;f_d,\theta_d)\subset F_d$ for all critical points $y$ in $(F_d\cap \dot D)\setminus \cl{V_{d'}}$. These are precisely the critical points in $V_d\setminus \cl{V_{d'}}$. 

Now let $y$ be a critical point of $f_d$ in $\cl{V_{d'}}$. Then $y\in F_j$ for some $j\leq d'$ since $f_d$ coincides with $f_{d'}$ near $G_{d'}$ and (i) in ($C_{d'}$) holds. The inclusion $$ W^s(y;f_{d'},\theta_{d'}) \subset W^s(y;f_d,\theta_d) $$ holds because if $p\in W^s(y;f_{d'},\theta_{d'})$ then $\phi^t_{d'}(p) \in F_j$ for all $t\geq 0$ by (ii) in ($C_{d'}$). Hence by uniqueness $\phi^t_{d'}(p)=\phi^t_d(p)$ for all $t\geq 0$, so $p\in W^s(y;f_d,\theta_d)$. Now let $p\in W^s(y;f_d,\theta_d)$. Then, since $W^s(y;f_d,\theta_d)$ coincides with $W^s(y;f_{d'},\theta_{d'})$ near $y$, we get that $\phi^T_d(p) \in W^s(y;f_{d'},\theta_{d'})$ if $T\gg1$. By (ii) in ($C_{d'}$), we have $\phi^t_{d'}(\phi^T_d(p)) \in F_j$ for all $t$ for which this is well-defined. Since $(f_{d'},\theta_{d'})$ and $(f_d,\theta_d)$ coincide near $F_j$, we get $\phi^t_{d'}(\phi^T_d(p)) = \phi^t_d(\phi^T(p))$ for all $t\in [-T,0]$. Setting $t=-T$ we obtain $p\in W^s(y;f_{d'},\theta_{d'})$. We have shown that
\begin{equation}\label{coincidence_some_stable_mfds}
y\in \crit(f_d) \cap \cl{V_{d'}} \ \Rightarrow \ W^s(y;f_{d'},\theta_{d'}) = W^s(y;f_{d},\theta_{d})
\end{equation}
holds. Summarizing we have shown that (ii) in ($C_d$) follows.

\subsubsection*{Property (iii) of claim ($C_d$)}

We next prove Property (iii) in the claim ($C_d$).
Let $\{x,y\} \subset \crit(f_d)\cap V_d$. Denote $\M_d(x,y) = W^u(x;f_d,\theta_d) \cap W^s(y;f_d,\theta_d)$ and assume that $\M_d(x,y) \neq \emptyset$. As in the induction start, we consider two cases.

Let us first analyze the case $\{x,y\} \subset V_{d'}$. In particular, $\{x,y\}\subset G_{d'}$ and we find that $\M_d(x,y)\subset G_{d'}$ by what was proved above. We claim that there exists a neighborhood $\mathcal{N}$ of $\M_d(x,y)$ in $\dot D$ such that $$ W^u(x;f_d,\theta_d) \cap \mathcal{N} = W^u(x;f_{d'},\theta_{d'}) \cap \mathcal{N}. $$ In fact, take $p\in \M_d(x,y)$. Since $(f_{d'},\theta_{d'})$ and $(f_d,\theta_d)$ coincide near $x$, we know that $W^u(x;f_d,\theta_d)$ and $W^u(x;f_{d'},\theta_{d'})$ coincide near $x$. By continuity of the flow, we find a neighborhood $V_1$ of $p$ in $W^u(x;f_d,\theta_d)$ and $T_1\gg1$ such that $$ \phi^{-T_1}_d(V_1)\subset W^u(x;f_{d'},\theta_{d'}). $$ Again by continuity, there exists a neighborhood $V_2\subset V_1$ of $p$ in $W^u(x;f_d,\theta_d)$ such that $\phi_d^{[-T_1,0]}(V_2)$ is contained on a neighborhood of $G_{d'}$ where $(f_d,\theta_d)$ and $(f_{d'},\theta_{d'})$ coincide. $V_2$ exists because $\phi_d^{[-T_1,0]}(p) \subset G_{d'}$. By uniqueness of solutions $\phi^t_{d'}(\phi^{-T_1}_d(q)) = \phi^t_d(\phi^{-T_1}_d(q))$ for all $q\in V_2$ and $t\in[0,T_1]$. Setting $t=T_1$ we get $V_2\subset W^u(x;f_{d'},\theta_{d'})$. It follows that the neighborhood $\mathcal{N}$ as required exists. Hence, if $\{x,y\} \subset V_{d'}$, then (iii) in ($C_{d'}$) implies that $W^u(x;f_d,\theta_d)$ intersects $W^s(y;f_d,\theta_d)$ transversely.

Now we analyze the case $\{x,y\}\not\subset V_{d'}$. Then $y\not\in \cl{V_{d'}}$ because if $y\in \cl{V_{d'}}$ then $\{x,y\}\in G_{d'}$ since $\M_d(x,y)\neq\emptyset$ and $G_{d'}$ is closed and invariant under the flow. This shows that $y\in V_d\setminus \cl{V_{d'}}$ and, hence, $y\in F_d\cap \dot D$. By the properties of the preliminary pair~\eqref{preliminary_MS_pair} and (ii) in Lemma~\ref{lemma_crucial4} we conclude that  $W^u(x;f_d,\theta_d)$ intersects $W^s(y;f_d,\theta_d)$ transversely. Summing up we proved (iii) in ($C_d$). 

The induction step is complete and therefore also the proof of Theorem~\ref{main1}.

\section{Global invariant Morse-Smale pair for finite-cyclic group actions}
\label{sec:MS_global}

In this section we will show the existence of an invariant global Morse-Smale pair for any closed manifold endowed with a $\Z/k\Z$-action. The proof is by induction and similar to the local one. It is actually simpler, since it is not a perturbative argument where we have to care about the $C^2$-norm of the perturbation. In particular, we do not have to deal with regularized distance functions.

As in the previous sections, let $F_j = \{x \in M \mid a^j(x)=x\}$, $\iota_j: F_j \to M$ be the inclusion and given $d \in \div(k)$ define
\begin{equation}
\label{sets_G_d}
G_d = \bigcup_{j\in\div(k),j\leq d} F_j.
\end{equation}
Consider also for each $x\in M$ the integer $i_x\geq 1$ defined by
\begin{equation}
i_x = \min \{ i\geq 1 \mid x\in F_i \}.
\end{equation}
Clearly $i_x<k$ if, and only if, $x$ is a point with non-trivial isotropy group. Moreover, if $x\in F_d$ then $i_x$ divides $d$.

Similarly to the argument in the local case, consider the following family of claims indexed by $d\in\div(k)$:
\begin{itemize}
\item[$(C_d)$] There exists an arbitrarily small invariant open neighborhood $V_d$ of $G_d$, a metric $\theta_d$ on $M$ and a smooth function $f_d:V_d \to \R$ satisfying the following properties:
\begin{itemize}
\item[(i)] $(f_d,\theta_d|_{V_d})$ is invariant, and $W^s_{V_d}(x;f_d,\theta_d)$ intersects $W^u_{V_d}(y;f_d,\theta_d)$ transversely, for all $x,y \in \crit(f_d)$.
\item[(ii)] $\crit(f_d) \subset G_d$.
\item[(iii)] $W^s_{V_d}(x;f_d,\theta_d) \subset F_{i_x}$ for all $x\in \crit(f_d)$.
\end{itemize}
\end{itemize}
Above we wrote $W^s_{V_d},W^u_{V_d}$ with subscript $V_d$ to emphasize that these stable/unstable manifolds are taken with respect to the open set $V_d$. In what follows we shall carry out an inductive construction proving the claims $(C_d)$. Notice that $(C_k)$ gives our desired Morse-Smale invariant pair.

\subsubsection*{Starting the induction}

Consider a metric $\theta$ on $F_1$ and a function $f: F_1 \to \R$ such that $(f,\theta)$ is Morse-Smale on $F_1$. The existence of $(f,\theta)$ follows from standard arguments. Consider a $\Z_k$-invariant extension of $\theta$ to a metric $\theta_1$ on $M$. Applying Lemmas \ref{lemma_invariant_normal_decreasing} and \ref{lemma_crucial4} (ii) we get an invariant neighborhood $V_1$ of $G_1=F_1$ and an invariant function $f_1: V_1 \to \R$ satisfying the desired properties.

\subsubsection*{The inductive step}

Fix two consecutive integers $d'<d$ in $\div(k)$ and assume that $(C_{d'})$ holds. The action on $M$ induces a $\Z/d\Z$ action on $F_d$ which is generated by $a|_{F_d}$. 
Using  Lemma~\ref{lemma_intersection}, we compute its isotropy set
\begin{equation}\label{isotropy_locus_F_d}
\begin{aligned}
F_d \cap \left( \bigcup_{l\in\div(k),l<d} F_l \right) = F_d \cap \left( \bigcup_{l\in\div(k),l\leq d'} F_l \right) = F_d \cap G_{d'}.
\end{aligned}
\end{equation}
 Moreover, $F_d\cap V_{d'}$ is an invariant open neighborhood of $F_d \cap G_{d'}$ in $F_d$ and $f_{d'}|_{F_d\cap V_{d'}}$ is an invariant smooth function there. 

Let $h:F_d \to \R$ be any smooth $a|_{F_d}$-invariant function coinciding with $f_{d'}|_{F_d \cap V_{d'}}$ near $F_d \cap G_{d'}$. This can be obtained by taking any smooth function on $F_d$ coinciding with $f_{d'}|_{F_d \cap V_{d'}}$ near $F_d \cap G_{d'}$ and averaging it over the $\Z/d\Z$ action. After shrinking $V_{d'}$ we can assume that $h$ coincides with $f_{d'}$ on $F_d \cap V_{d'}$.

We claim that, possibly after perturbing $h$ on a compact subset of $F_d \setminus G_{d'}$, we may assume in addition that $h$ is Morse. As a matter of fact, by the properties of $f_{d'}$, the function $h$ has no critical points in the $a|_{F_d}$-invariant open set $(F_d \cap V_{d'})\setminus G_{d'}$. This is true because $h$ and $f_{d'}$ coincide on $F_d \cap V_{d'}$, the $\theta_{d'}$-gradient of $f_{d'}$ is tangent to $F_d$ at points in this set, and the critical set of $f_{d'}$ is contained in $G_{d'}$.

The $\Z/d\Z$ action on $F_d \setminus G_{d'}$ is free in view of the description~\eqref{isotropy_locus_F_d} of its isotropy locus. So we obtain a smooth manifold $X_d$ by the quotient of $F_d \setminus G_{d'}$ by the $\Z/d\Z$ action. Let $\Pi : F_d \setminus G_{d'} \to X_d$ denote the quotient map. The function $h|_{F_d \setminus G_{d'}}$ descends to a smooth function $h'\colon X_d \to \R$ with no critical points on the open subset $$ \mathcal O_d := \Pi((F_d \cap V_{d'})\setminus G_{d'}). $$ Clearly $\mathcal{O}_d$ is an end of $X_d$, and $X_d \setminus \mathcal O_d$ is a compact set containing the critical points of $h'$. Consequently we can $C^\infty$-slightly perturb $h'$ so that it becomes a Morse function on $X_d$ and, perhaps after shrinking $V_{d'}$ we may also assume that $h'$ remains unchanged in $\mathcal O_d$. As a consequence, the lift $h$ of $h'$ to $F_d\setminus G_{d'}$ has the desired properties.

Now we note that unstable and stable manifolds (taken in $F_d$ with respect to the metric $\iota_d^*\theta_{d'}$) of critical points of $h$ in $F_d \cap V_{d'}$ intersect transversely (in $F_d$). Namely, consider critical points $x_0$ and $x_1$ of $h$ on $F_d \cap V_{d'}$ which are connected by an anti-gradient trajectory $c$ of $h$ from $x_0$ to $x_1$ in $F_d$. Since $h$ coincides with $f_{d'}$ on $F_d\cap V_{d'}$ and the $\theta_{d'}$-gradient of $f_{d'}$ is tangent to $F_d$, we conclude that $x_0$ and $x_1$ are also critical points of $f_{d'}$, in particular, they belong to $G_{d'}$.

By (iii) in ($C_{d'}$), and uniqueness of solutions of ODEs, we know that the stable manifold of $x_1$ on $F_d$ with respect to the pair $(h,\iota_d^*\theta_{d'})$ is equal to $W^s_{V_{d'}}(x_1;f_{d'},\theta_{d'}) \subset F_{i_{x_1}}$. Since $F_{i_x}$ is compact, we get $x_0 \in F_{i_{x_1}}$. By (i) in ($C_{d'}$) we can apply Lemma~\ref{lemma_crucial4} item i) to conclude that $W^u(x_0;h,\iota_d^*\theta_{d'})$ intersects $W^s(x_1;h,\iota_d^*\theta_{d'})$ transversally along $c$, as desired.

 
By Lemma~\ref{lemma_crucial3}, the metric $\theta_{d'}$ can be slightly $C^\infty$-perturbed into an invariant metric $\theta_d$ such that $\theta_d-\theta_{d'}$ is compactly supported in a neighborhood of $\crit(h|_{F_d \setminus V_{d'}})$ and $(h,\iota_d^*\theta_d)$ is Morse-Smale on $F_d$.


Now we are ready to define $f_d$ and $V_d$ with the properties claimed in $(C_d)$. In order to do this, consider the extension $\bar h$ of $h$ to a small neighborhood of $F_d$ given by an application of Lemma~\ref{lemma_invariant_normal_decreasing} to $h$ and $F_d$. We need the following

\begin{lemma}
\label{lemma_consecutive_bumps}
Let $V'$ be an open neighborhood of $G_{d'}$ and $V$ a neighborhood of $F_d$. There are smooth invariant functions $\phi_0,\phi_1:M\to[0,1]$ satisfying
\begin{itemize}
\item[a)] $\supp(\phi_0)\subset V'$ and $\supp(\phi_1) \subset V$,
\item[b)] $\phi_0+\phi_1 \equiv 1$ on a neighborhood of $G_d = G_{d'} \cup F_d$,
\item[c)] $\supp(\phi_1) \cap G_{d'} = \emptyset$.
\end{itemize}
\end{lemma}

\begin{proof}
Define $\phi_0$ to be the average over the action of any function which is equal to one on a small neighborhood of $G_{d'}$ and supported in $V'$. Then define $\phi_1$ by $1-\phi_0$ and use an invariant cutoff function to achieve $\supp(\phi_1)\subset V$. As the submanifold $F_d$ is compact, such an invariant cutoff function can be constructed using the shell between two invariant neighborhoods as given by Corollary \ref{cor_inv_smooth_nbd}.
\end{proof}

By the previous lemma, we have bump functions $\phi_0$ and $\phi_1$ such that $\supp(\phi_0)\subset V_{d'}$ which implies that $\phi_1\equiv 1$ on a neighborhood of $F_d \setminus V_{d'}$. Define
\begin{equation}
f_d = \phi_0 f_{d'} + \phi_1 \bar h.
\end{equation}

It remains to check (i), (ii) and (iii) and we begin with (ii). Clearly, $f_d$ is invariant since so are $\phi_0$, $\phi_1$, $f_{d'}$ and $\bar h$. Moreover $f_d|_{F_d} \equiv h$, and $f_d$ coincides with $f_{d'}$ near $G_{d'}$ by item c) in Lemma~\ref{lemma_consecutive_bumps}.

We claim that if $x\in \crit(f_d) \cap G_d$ then one of the following holds:
\begin{itemize}
\item[A)] If $x \in G_{d'}$ then there exists an open neighborhood $U_x$ in $M$ such that $f_d|_{U_x} \equiv f_{d'}|_{U_x}$.
\item[B)] If $x \in F_d \setminus G_{d'}$ then there exists an open neighborhood $U_x$ in $M$ such that $f_d|_{U_x} \equiv \bar h|_{U_x}$.
\end{itemize}
To see this we start by noting that if $x\in \crit(f_d) \cap G_d$ then either $x\in G_{d'}$ or $x\in F_d\setminus G_{d'}$ because $G_d = F_d \cup G_{d'}$. Note that A) holds since $f_d$ coincides with $f_{d'}$ near $G_{d'}$. Let us check B). If $x\in F_d\setminus G_{d'}$ then $x \in F_d\setminus V_{d'}$ because $f_d$ coincides with $h|_{F_d \cap V_{d'}} = f_{d'}|_{F_d \cap V_{d'}}$ on $F_d\cap V_{d'}$ and $f_{d'}$ has no critical points in $(F_d \setminus G_{d'}) \cap V_{d'}$. Thus $\phi_1$ is identically equal to $1$ near $x$, so that $f_d\equiv\bar h$ near $x$.

In particular, all critical points in $\crit(f_d) \cap G_d$ are non-degenerate. This is obviously true in case A) since $f_{d'}$ is Morse. In case B), this follows from the construction in the proof of Lemma~\ref{lemma_invariant_normal_decreasing} because $h$ is Morse and $x$ is non-degenerate as a critical point of $\bar h$. As a consequence we find a small invariant open neighborhood $V_d$ of $G_d$ where 
\begin{equation}\label{crit_f_d}
\crit(f_d) \cap V_d \subset G_d.
\end{equation}
This proves condition (ii) of the Claim $(C_d)$.

Next, we prove condition (iii). Let $x\in \crit(f_d) \cap V_d$. By~\eqref{crit_f_d}, either $x$ falls into case A) or into case B).  In case~A) we can apply Lemma~\ref{lemma_crucial1} to find $$ W^s_{V_d}(x;f_d,\theta_d) = W^s_{V_{d'}}(x;f_{d'},\theta_{d'}) \subset F_{i_x} \subset G_{d'}. $$ In case~B), since $f_d|_{U_x}=\bar h|_{U_x}$, we can again apply Lemma~\ref{lemma_crucial1} to conclude that $$ W^s_{V_d}(x;f_d,\theta_d) \subset F_d, $$ but note that $F_{i_x}=F_d$ in this case.  Thus, in either case we find $$ W^s_{V_d}(x;f_d,\theta_d) \subset F_{i_x} \subset G_d $$ for every $x\in \crit(f_d) \cap V_d$.

To conclude the proof that $(f_d|_{V_d},\theta_d)$ satisfies $(C_d)$, it remains to establish condition (i).
Consider two critical point $x_0$ and $x_1$ in $\crit(f_d) \cap V_d$ and let $c$ be an anti-gradient trajectory of $f_d$ from $x_0$ to $x_1$. Again we consider different cases, analogous to cases A) and B) above. If $x_1 \in G_{d'}$, then we find that also $x_0 \in G_{d'}$ since 
\[
x_0 \in \overline{W^s_{V_d}(x_1;f_d,\theta_d)} \subset \overline{F_{i_{x_1}}} =  F_{i_{x_1}} \subset G_{d'}
\]
and we can apply $(C_{d'})$. Moreover, $c$ is an anti-gradient trajectory of $f_{d'}|_{V_{d'}}$.  
By (iii) in ($C_{d'}$) (notice that, by construction, we have $\theta_d|_{V_{d'}}=\theta_{d'}|_{V_{d'}}$) we get that $W^u_{V_d}(x_0;f_d,\theta_d)$ intersects $W^s_{V_d}(x_1;f_d,\theta_d)$ transversely along $c$. 

In the second case, i.e., if $x_1 \in F_d\setminus G_{d'}$, we also have $x_0\in F_d$ since 
\[
x_0 \in \cl{W^s_{V_d}(x_1;f_d,\theta_d)} \subset F_{i_{x_1}} = F_d.
\]
Thus $W^s_{V_d}(x_0;f_d,\theta_d) \subset F_{i_{x_0}} \subset F_d$ and $c$ is an anti-gradient of $f_d|_{F_d} = h$.  Since $(h,\iota_d^*\theta)$ is Morse-Smale we apply Lemma~\ref{lemma_crucial4} item ii) to find that $W^u_{V_d}(x_0;f_d,\theta_d)$ intersects $W^s_{V_d}(x_1;f_d,\theta_d)$ transversely along $c$. This completes the proof of the induction step and thus the Claims $(C_d)$ are established for all $d$.

\appendix

\section{Regularized distance functions}\label{app_reg_dist_functions}

The statement below is found in chapter VI of Stein's book~\cite{stein}.

\begin{theorem}\label{thm_existence_regularized_distance}
Let $X\subset \R^N$ be any closed set. Then there exists a function $\delta_X:\R^N\to[0,+\infty)$ with the following properties.
\begin{itemize}
\item[a)] $\delta_X$ is continuous on $\R^N$, and is smooth on $\R^N\setminus X$.
\item[b)] $c_1 \ \dist(x,X) \leq \delta_X(x) \leq c_2 \ \dist(x,X)$ for all $x\in \R^N$, with constants $c_1,c_2>0$ independent of $X$.
\item[c)] $|D^\alpha\delta_X(x)| \leq B_\alpha \dist(x,X)^{1-|\alpha|}$ for all $x\not\in X$, with constants $B_\alpha>0$ independent of $X$.
\end{itemize}
\end{theorem}

The purpose of this appendix is to prove the following refinement.

\begin{proposition}\label{prop_refinement_thm_reg_function}
Let $Y\subset\R^N$ be closed and $E\subset \R^N$ be a linear subspace, both invariant under a linear map $A\in O(N)$ satisfying $A^k=I$ for some $k\geq 1$. Setting $X=Y\cup E$, there exists a function $\delta_X:\R^N\to[0,+\infty)$ that satisfies all the properties of Theorem~\ref{thm_existence_regularized_distance} and is, in addition, $A$-invariant and coincides with $\dist(x,E)$ on a neighborhood of $E\setminus Y$.
\end{proposition}

To prove this proposition we follow the proof of Theorem~\ref{thm_existence_regularized_distance} from~\cite{stein} closely. By a cube in $\R^N$ with side $l>0$ we mean a product of half-open intervals $Q = \prod_{i=1}^N [a_i,b_i)$ where $l=b_i-a_i$ for all $i$. Sets of the form $\cl{Q}$ and $\intr(Q)$ will be referred to as closed and open cubes. Cubes $Q$ also have a center, the unique point $x$ with the following property: if $Q$ has side $l$ then $\overline Q$ is the closed ball with radius $l/2$ centered at $x$ with respect to the $\|\cdot\|_\infty$-norm.

We will say that a cube $Q$ touches another cube $Q'$ if $\cl{Q} \cap \cl{Q'} \neq \emptyset$. We might use the same terminology for open or closed cubes. If $Q$ is a (possibly open or closed) cube with center $x$ then we denote by $Q^*$ the cube
\[
Q^* := x + \frac{9}{8}(Q-x).
\]
It follows that
\begin{equation}\label{ineq_distQ_Q*}
\dist(x,Q) \leq \frac{1}{16}\diam Q \ \ \forall x\in Q^*
\end{equation}
and
\begin{equation}\label{ineq_diam_Q*}
\diam Q^* = \frac{9}{8} \diam Q.
\end{equation}

We start by recalling the following theorem also proved in~\cite[chapter VI]{stein}.

\begin{theorem}\label{thm_covering}
Let $X\subset \R^N$ be a closed set. There exists a countable collection of cubes $\F = \{Q_1,Q_2,\dots\}$ such that
\begin{itemize}
\item[a)] $\R^N \setminus X = \cup_k Q_k$.
\item[b)] Cubes in $\F$ are pairwise disjoint.
\item[c)] $\diam Q_k \leq \dist(Q_k,X) \leq 4\ \diam Q_k$ for all $k$.
\item[d)] If $Q\in \F$ touches $Q'\in\F$ then $\frac{1}{4} \diam Q \leq \diam Q' \leq 4\diam Q$.
\item[e)] If $Q\in\F$ then at most $(12)^N$ cubes in $\F$ touch $Q$.
\item[f)] Every $x\not\in X$ has a neighborhood  which intersects at most $(12)^N$ cubes in $\F^* = \{Q^* \mid Q\in \F\}$.
\end{itemize}
\end{theorem}

\begin{proof}[Proof of Proposition~\ref{prop_refinement_thm_reg_function}]
Let $\F=\{Q_k\}$ be a covering of $\R^N\setminus X$ given by Theorem~\ref{thm_covering}.

Let $C_0$ be the unit cube centered at the origin, and choose a smooth function $\varphi:\R^N \to [0,1]$ satisfying $\supp(\varphi) \subset \intr(C_0^*)$ and $\varphi|_{\overline{C_0}}\equiv1$. For each $k$ consider $\varphi_k(x) = \varphi((x-x_k)/l_k)$ where $x_k$ and $l_k$ are the center and the side of $Q_k$ respectively. Then $\supp(\varphi_k)\subset \intr(Q^*_k)$ and $\varphi_k|_{\overline{Q_k}} \equiv 1$. Note that
\begin{equation}\label{ineq_derivatives_varphi_k}
|D^\alpha \varphi_k(x)| \leq A_\alpha (\diam Q_k)^{-|\alpha|}
\end{equation}
where $A_\alpha$ depends only on $\varphi$. The function
\begin{equation}
\Phi(x) := \sum_k \varphi_k(x)
\end{equation}
is smooth on $\R^N\setminus X$. By f) in Theorem~\ref{thm_covering}, $\Phi$ satisfies the uniform pointwise estimate
\begin{equation}\label{ineq_bounds_big_phi}
1 \leq \Phi(x) \leq (12)^N
\end{equation}
for all $x\not\in X$.

We claim that
\begin{equation}\label{ineq_dist_diam_above_below}
\frac{3}{4} \ \diam Q_k \leq \dist(x,X) \leq 6 \ \diam Q_k \ \ \ \forall x\in Q_k^*.
\end{equation}
To see this choose any $x\in Q_k^*$. The second inequality follows from considering $p\in X$ and $q\in \overline Q_k$ satisfying $|p-q| = \dist(Q_k,X)$, and using~\eqref{ineq_diam_Q*} and c) in Theorem~\ref{thm_covering} to estimate
\begin{equation*}
\dist(x,X) \leq |x-p| \leq |x-q|+|q-p| \leq \diam Q_k^* + \dist(Q_k,X) \leq 6 \ \diam Q_k.
\end{equation*}
The first inequality in~\eqref{ineq_dist_diam_above_below} follows from taking $p\in X$ such that $\dist(x,X)=|x-p|$ and $q\in \overline Q_k$ such that $\dist(x,Q_k)=|x-q|$, and estimating with the help of c) in Theorem~\ref{thm_covering} and of~\eqref{ineq_distQ_Q*} as follows
\begin{equation*}
\begin{aligned}
\diam Q_k &\leq \dist(Q_k,X) = \dist(\overline Q_k,X) \leq |q-p| \leq |q-x|+|x-p| \\
&= \dist(x,Q_k) + \dist(x,X) \\
&\leq \frac{1}{16}\diam Q_k + \dist(x,X).
\end{aligned}
\end{equation*}
This proves \eqref{ineq_dist_diam_above_below}.

We now set
\[
U = \{ x\in \R^N \mid \dist(x,E) < \dist(x,Y) \}.
\]
This is an open neighborhood of $E\setminus Y$ in $\R^N\setminus Y$. Now consider
\[
\F_U = \{ Q\in \F \mid Q^* \subset U \}, \ \ \ \ \F'_U = \F \setminus \F_U.
\]

We claim that with these choices, there exists $c>0$ such that
\begin{equation}\label{ineq_comp_E}
x\in Q^*,Q\in \F'_U \Rightarrow \dist(x,Y) \leq c\ \dist(x,E).
\end{equation}
In fact, let us fix $x\in Q^*$, so with $z\in\R^N$ arbitrary and $e\in E$ such that $\dist(x,E)=|x-e|$ we can estimate
\[
\dist(z,E) \leq |z-e| \leq |z-x|+|x-e| = |z-x| + \dist(x,E).
\]
Assuming $Q\in \F'_U$ there exists $z\in Q^*$ such that $\dist(z,Y)\leq \dist(z,E)$, and we can choose $p\in Y$ such that $|z-p|=\dist(z,Y)$. Plugging this into the above inequality we get
\[
\begin{aligned}
\dist(x,Y) &\leq |x-p| \leq |x-z|+|z-p| \\
&\leq \diam Q^* + \dist(z,Y) \\
&\leq \diam Q^* + \dist(z,E) \\
&\leq \diam Q^* + |z-x| + \dist(x,E) \\
&\leq 2\diam Q^* + \dist(x,E) \\
&\leq \frac{9}{4} \diam Q + \dist(x,E) \\
&\leq 4\ \dist (x,E)
\end{aligned}
\]
as desired. Here we used~\eqref{ineq_dist_diam_above_below}.

Analogously to~\cite{stein} we define a function $\Delta:\R^N \setminus X \to (0,+\infty)$ by
\begin{equation}\label{defn_Delta}
\Delta(x) := \left( \sum_{Q_k \in \F'_U} \diam Q_k \ \varphi_k(x) \right) + \left( \sum_{Q_k \in \F_U} \dist(x,E) \ \varphi_k(x) \right).
\end{equation}
We claim that
\begin{equation}\label{ineq_Delta_below}
\Delta(x) \geq \frac{1}{6} \dist(x,X) \ \ \ \forall x\not\in X.
\end{equation}
If $x\in Q_k$ then $\varphi_k(x)=1$ and there are two possibilities: either $Q_k \in \F'_U$ and in this case we get $\Delta(x) \geq \diam Q_k \geq (1/6) \dist(x,X)$ in view of~\eqref{ineq_dist_diam_above_below}, or $Q_k \in \F_U$ and in this case $\Delta(x) \geq \dist(x,E) \geq \dist(x,X)$. Thus~\eqref{ineq_Delta_below} is proved.
Now we show that
\begin{equation}\label{ineq_Delta_above}
\Delta(x) \leq (12)^N (4/3) \dist(x,X) \ \ \ \forall x\not\in X.
\end{equation}
Fix $x\not\in X$ and consider the collection $\F_x = \{Q\in\F \mid x\in Q^*\}$. According to f) in Theorem~\ref{thm_covering} we have $\#\F_x \leq (12)^N$. If $Q_k \in \F_x \cap \F_U$ then the corresponding term in the sum~\eqref{defn_Delta} is $\varphi_k(x) \dist(x,E) \leq \dist(x,E) \leq\dist(x,X)$. Now note that if $Q_k \in \F_x \cap \F'_U$ then the corresponding term in the sum~\eqref{defn_Delta} is $$ \varphi_k(x) \diam Q_k \leq \diam Q_k \leq (4/3)\dist(x,X) $$ by~\eqref{ineq_dist_diam_above_below}. We have shown that all of the terms in~\eqref{defn_Delta} which do not vanish at $x$ are at most $(4/3)\dist(x,X)$. Since there are at most $(12)^N$ such terms,~\eqref{ineq_Delta_above} is proved.

Let $c>0$ be the constant in~\eqref{ineq_comp_E}. The set
\[
W = \{ x\in \R^N \mid c\ \dist(x,E) < \dist(x,Y) \} \subset \R^N \setminus Y
\]
is an open neighborhood of $E\setminus Y$. We note that~\eqref{ineq_comp_E} can be rewritten as $$ Q_k \in \F'_U \Rightarrow Q_k^* \cap W = \emptyset. $$ Thus we have $x\in W\setminus X \ \Rightarrow \ \Delta(x) = \Phi(x) \dist(x,E)$. Finally we consider
\[
{\widehat{\delta}_X}(x) = \frac{\Delta(x)}{\Phi(x)}.
\]
We claim that $\widehat\delta_X$ would be our desired function if we were not interested in $A$-invariance. By~\eqref{ineq_bounds_big_phi},~\eqref{ineq_Delta_below} and~\eqref{ineq_Delta_above} we get constants $c_1,c_2>0$ independent of $Y$ and $E$ such that
\[
c_1 \ \dist(x,X) \leq {\widehat{\delta}_X}(x) \leq c_2 \ \dist(x,X)
\]
for all $x\not\in X$. It follows from this that ${\widehat{\delta}_X}$ can be continuously extended to $\R^N$ by setting it equal to zero on $X$. Moreover, a) and b) of Theorem~\ref{thm_existence_regularized_distance} are true for $\widehat{\delta}_X$.  Note that
\begin{equation}\label{values_delta_hat_near_E}
x\in W \ \Rightarrow \ {\widehat{\delta}_X}(x) = \dist(x,E).
\end{equation}

Let us prove that $\widehat\delta_X$ satisfies c) in Theorem~\ref{thm_existence_regularized_distance}. For this we need to investigate the derivatives of $\Phi$ and $\Delta$. By~\eqref{ineq_derivatives_varphi_k} and f) in Theorem~\ref{thm_covering} we get
\[
|\alpha|\geq 1 \ \Rightarrow \ |D^\alpha\Phi(x)| \leq (12)^N A_\alpha \max \{ (\diam Q_k)^{-|\alpha|} \mid x\in Q_k^* \}.
\]
By~\eqref{ineq_dist_diam_above_below} we get
\[
|\alpha| \geq 1 \ \Rightarrow \ |D^\alpha\Phi(x)| \leq (12)^NA_\alpha 6^{|\alpha|} \ \dist(x,X)^{-|\alpha|}
\]
from where it follows that
\[
|\alpha| \geq 1 \ \Rightarrow \ |D^\alpha(1/\Phi)(x)| \leq M_\alpha \ \dist(x,X)^{-|\alpha|}
\]
for suitable constants $M_{\alpha}>0$ independent of $X$. Here we strongly used the bounds~\eqref{ineq_bounds_big_phi}. We now turn to derivatives of $\Delta$.
\begin{equation}\label{derivatives_Delta}
\begin{aligned}
&D^\alpha\Delta(x) \\
&= \left( \sum_{Q_k \in \F'_U} \diam Q_k \ D^\alpha\varphi_k(x) \right) + \left( \sum_{Q_k \in \F_U} D^\alpha[\dist(\cdot,E)\varphi_k](x) \right) \\
&= T_1 + T_2.
\end{aligned}
\end{equation}
With $x\not\in X$ fixed, the first term is bounded as 
\[
\begin{aligned}
T_1 &\leq (12)^NA_\alpha\max\{(\diam Q_k)^{1-|\alpha|} \mid x\in Q_k^*\} \\
&\leq (12)^NA_\alpha 6^{|\alpha|-1} \dist(x,X)^{1-|\alpha|}.
\end{aligned}
\]
Here f) in Theorem~\ref{thm_covering} was strongly used. The general estimate
\[
|D^\beta(\dist(\cdot,E))(z)| \leq C_\beta \dist(z,E)^{1-|\beta|}
\]
for all $z\in\R^N\setminus E$ together with the product rule and the estimates~\eqref{ineq_derivatives_varphi_k} implies that the second term in~\eqref{derivatives_Delta} is bounded from above by
\[
\begin{aligned}
T_2 &\leq \Gamma_\alpha \sum_{\beta_1+\beta_2=\alpha} |D^{\beta_1}(\dist(\cdot,E))(x)||D^{\beta_2}\varphi_k(x)| \\
&\leq \Gamma'_\alpha \sum_{\beta_1+\beta_2=\alpha} \dist(x,E)^{1-|\beta_1|}(\diam Q_k)^{-|\beta_2|} \\
&\leq \Gamma''_\alpha \sum_{\beta_1+\beta_2=\alpha} \dist(x,X)^{1-|\beta_1|}\dist(x,X)^{-|\beta_2|} \\
&= \Gamma'''_\alpha \ \dist(x,X)^{1-|\alpha|}.
\end{aligned}
\]
The third inequality used~\eqref{ineq_dist_diam_above_below} and that $\dist(x,E)=\dist(x,X)$ for a point $x\in Q_k \in \F_U$. Also f) in Theorem~\ref{thm_covering} was used. The conclusion is that
\begin{equation}
|D^\alpha\Delta(x)| \leq H_\alpha \ \dist(x,X)^{1-|\alpha|}
\end{equation}
for all $ x\not\in X$ with a suitable constant $H_\alpha$ independent of $X$. Using these estimates and the estimates for the derivatives of $1/\Phi$ obtained above, it follows again from the product rule that there exist constants $B_\alpha>0$ independent of $X$ such that
\[
|D^\alpha{\widehat{\delta}_X}(x)| \leq B_\alpha \ \dist(x,X)^{1-|\alpha|}
\]
as desired.

To obtain $A$-invariance we finally define
\[
\delta_X(x) = \frac{1}{k} \sum_{i=0}^{k-1} \widehat\delta_X(A^ix).
\]
Note that $Y,E,X$ are $A$-invariant sets, and consequently a) and b) of Theorem~\ref{thm_existence_regularized_distance} continue to hold since $\dist(\cdot,X)$ is an $A$-invariant function. By the same token~\eqref{values_delta_hat_near_E} continues to hold on $W$ (note that $\dist(\cdot,E)$ and $W$ are $A$-invariant). Again by $A$-invariance of $X$ and $\dist(\cdot,X)$ estimates like c) in Theorem~\ref{thm_existence_regularized_distance} hold for $\delta_X$. The proof is complete.
\end{proof}

\section{Invariance of local Morse homology with symmetries}\label{app_invariance}

Here we prove Proposition~\ref{prop_GM_pairs}. Let $\theta,f,x,U$ be as in the statement: $(f,\theta)$ is a pair consisting of a smooth function and a smooth metric on a manifold without boundary, $x$ is an isolated critical point of $f$ and $U$ is an open relatively compact isolating neighborhood for $(f,x)$, all of which are invariant under a smooth action of~$\Z_k$.

We need to find a $\Z_k$-invariant Gromoll-Meyer pair $(W,W_-)$ in $U$ and some $C^2$-neighborhood $\mathcal{N}_1$ of $(f,\theta)$ with the following properties: if $(f',\theta') \in \mathcal{N}_1$ is $\Z_k$-invariant and Morse-Smale on $U$ then the map~\eqref{invariance_map_local_hom} $$ \tilde\Psi_{(f',\theta',U)} : H_*(W,W_-) \stackrel{\sim}{\to} \HM_*(f',\theta',U) $$ in Definition~\ref{def_GM_pairs} is $\Z_k$-equivariant.

Choose an open $\Z_k$-invariant neighborhood $\O$ of $X\setminus U$ such that $x\not\in \cl{\O}$. First we need two preliminary lemmas.

\begin{lemma}\label{lemma_crossing_energy}
For every pair $(V_1,V)$ of open neighborhoods of $x$ satisfying $\cl{V_1}\subset V \subset \cl{V} \subset U$ we can find a $C^2$-neighborhood $\mathcal{N}$ of $(f,\theta)$ and $\epsilon>0$ with the following property:
\begin{equation}\label{time_delays_pre}
\begin{aligned}
&(f',\theta') \in \mathcal N, \ p \in \cl{V_1}, \ t > 0, \ \phi'_t(p) \not\in V, \ \phi'_{[0,t]}(p) \subset U \\ 
&\Rightarrow \ f'(p) - f'(\phi'_t(p)) > \epsilon.
\end{aligned}
\end{equation}
Here $\phi'_\cdot$ stands for the negative gradient flow of $(f',\theta')$.
\end{lemma}

\begin{proof}
There exists a $C^2$-neighborhood $\mathcal{N}$ and $L>0$ such that $(f',\theta') \in \mathcal N $ implies that $ |\nabla^{\theta'}f'|_{\theta'} \leq L$ holds pointwise on $\cl{U}$. Further shrinking $\mathcal N$ and increasing $L$ we may assume that $\dist_{\theta'}(\cl{V_1},X\setminus V) > 1/L$ and $|\nabla^{\theta'}f'|_{\theta'} \geq 1/L$ pointwise on $\cl{U} \setminus V_1$ whenever $(f',\theta') \in \mathcal N$. Since $\phi'_t(p) \not\in V$, there exists $0\leq t'<t$ such that $\phi'_{t'}(p) \in \cl{V_1}$ and $\phi'_{\tau}(p) \not\in \cl{V_1} $ for all $\tau \in (t',t]$. Thus
\[
\frac{1}{L} 
< \dist_{\theta'}(\cl{V_1},X\setminus V) \leq \int_{t'}^t |\nabla^{\theta'}f'(\phi'_\tau(p))|_{\theta'} \ d\tau \leq (t-t')L 
\]
which implies that
\[
 t-t' > \frac{1}{L^2}.
\]
We then compute
\[
\begin{aligned}
f'(p) - f'(\phi'_t(p)) 
&\geq f'(\phi'_{t'}(p)) - f'(\phi'_t(p)) = \int_{t'}^t |\nabla^{\theta'}f'(\phi'_\tau(p))|^2_{\theta'} \ d\tau \\
&\geq \frac{t-t'}{L^2} > \frac{1}{L^4}.
\end{aligned}
\]
Setting $\epsilon = L^{-4}$ finishes the proof.
\end{proof}

\begin{lemma}\label{lemma_local_pairs_support}
Let $V_1,V$ be $\Z_k$-invariant open neighborhoods of $x$ satisfying $\cl{V_1}\subset V\subset\cl{V} \subset U$. Let $\mathcal{N}$ be a $C^2$-neighborhood of $(f,\theta)$. There exists a $C^2$-neighborhood $\mathcal{N}_1\subset \mathcal{N}$ with the following properties: 
\begin{enumerate}
\item If $(f_1,\theta_1) \in \mathcal{N}_1$ then all critical points of $f_1|_U$ and all $\theta_1$-gradient trajectories of $f_1$ connecting them are contained in $V_1$.
\item If $(f_1,\theta_1) \in \mathcal{N}_1$ then there exists $(f',\theta') \in \mathcal{N}$ such that all critical points of $f'|_U$ and all $\theta'$-gradient trajectories of $f'$ connecting them are contained in $V_1$, $(f',\theta')$ coincides with $(f_1,\theta_1)$ on $V_1$ and coincides with $(f,\theta)$ on $X\setminus \cl{V}$. Moreover, if $(f_1,\theta_1)$ is $\Z_k$-invariant then $(f',\theta')$ can be taken $\Z_k$-invariantly.
\end{enumerate}
\end{lemma}

\begin{proof}
Choose a $\Z_k$-invariant bump function $\beta:X\to[0,1]$ such that $\beta|_{\cl{V_1}}\equiv1$ and $\supp(\beta) \subset V$. There exists a $C^2$-neighborhood $\mathcal{N}_0\subset\mathcal{N}$ of $(f,\theta)$ such that if $(f_0,\theta_0) \in \mathcal{N}_0$ then all critical points of $f_0|_U$ and all $\theta_0$-gradient trajectories of $f_0$ connecting them are contained in $V_1$. This follows from Lemma~\ref{lemma_crossing_energy} applied to $V_1$ and a small open neighborhood $V_2$ of $x$ satisfying $\cl{V_2}\subset V_1$. Indeed, if $\mathcal{N}_0$ is small enough, the difference between critical values of $f_0|_U$ is smaller than any positive constant fixed {\it a priori}. 

There exists $\mathcal{N}_1 \subset \mathcal{N}_0$ such that if $(f_1,\theta_1) \in \mathcal{N}_1$ then
$$
 (f',\theta') := (f+\beta(f_1-f),\theta+\beta(\theta_1-\theta)) \in \mathcal{N}_0
$$
satisfies the required properties, by the properties of $\mathcal N_0$ and by the fact that $\beta$ is identically $1$ on $V_1$.
\end{proof}

Let us get started with the proof of Proposition~\ref{prop_GM_pairs}. 
Suppose that $f(x)=0$, without loss of generality. Fix an open neighborhood $V_0$ of $x$ satisfying $\cl{V_0} \subset U\setminus \cl{\O}$. Choose a smooth $\Z_k$-invariant function $\beta_0:X\to [0,1]$ such that $\beta_0$ vanishes identically on a neighborhood of $X\setminus \O$ and $\supp(1-\beta_0)$ is a compact subset of $U$. In particular, $\supp(d\beta_0) \subset \O \cap U$ is compact. Note that $\beta_0$ can be chosen $\Z_k$-invariantly since $U$ and $\O$ are $\Z_k$-invariant.

Applying Lemma~\ref{lemma_crossing_energy} to the pair $(V_0,U\setminus\cl{\O})$ we find a $C^2$-neighborhood $\mathcal{N}_0$ of $(f,\theta)$ and $\epsilon>0$ such that
\begin{equation}\label{time_delays}
\begin{aligned}
&(f',\theta') \in \mathcal N_0, \ p \in \cl{V_0}, \ t > 0, \ \phi'_t(p) \in \cl{\O}, \ \phi'_{[0,t]}(p) \subset U \\ 
&\Rightarrow \ f'(p) - f'(\phi'_t(p)) > \epsilon.
\end{aligned}
\end{equation}
Perhaps after further shrinking $\mathcal{N}_0$, we can further assume that for all $(f_0,\theta_0) \in \mathcal N_0$ we have
\begin{equation}\label{first_aux_estimate-f_0}
2|df|_{\theta} > |df_0|_{\theta_0} > \frac{1}{2} |df|_{\theta} > 0 \ \ \text{ pointwise on} \ \ \cl{U} \cap \cl{\O}.\end{equation}
Select $a,b>0$ such that $a+b<\epsilon$, and for all $(f_0,\theta_0) \in \mathcal N_0$ the estimate
\begin{equation}\label{second_aux_estimate-f_0}
(a+b)|d\beta_0|_{\theta_0} < \frac{1}{2} |df_0|_{\theta_0} \ \ \text{holds pointwise on} \ \ \cl{U} \cap \cl{\O}.
\end{equation}
We shall complete the proof of the proposition by showing that
\begin{equation}
\begin{array}{ccc} W = \{ f\leq a \} \cap U & & W_- = \{ f-(a+b)\beta_0 \leq -b \} \cap U \end{array}
\end{equation}
is the desired pair.

Choose a compact neighborhood $K$ of $x$ satisfying
\[
K \subset V_0 \cap \{-b/2 \leq f \leq a/2\}
\]
and a $C^2$-neighborhood $\mathcal N\subset\mathcal N_0$ of $(f,\theta)$ such that 
$$
(f',\theta') \in \mathcal N \Rightarrow f'(K) \subset \left[ -\frac{2b}{3},\frac{2a}{3} \right]. 
$$
 Finally choose $\Z_k$-invariant open neighborhoods $V_1,V$ of $x$ such that 
 $$
  \cl{V_1}\subset V \subset \cl{V} \subset K .
  $$
   We obtain a $C^2$-neighborhood $\mathcal N_1 \subset \mathcal N$ by applying Lemma~\ref{lemma_local_pairs_support} to these choices of $V_1,V,\mathcal{N}$.

Let $(f_1,\theta_1) \in \mathcal N_1$ be a $\Z_k$-invariant pair which is Morse-Smale on $U$. Such a pair exists by Theorem~\ref{main1}. We derive some consequences of our previous constructions. All critical points of $f_1|_U$ are contained in $V_1$, as well as all $\theta_1$-gradient trajectories connecting them. Moreover, there is a $\Z_k$-invariant pair $(f',\theta')\in\mathcal N$ coinciding with $(f_1,\theta_1)$ on $V_1$ such that all critical points of $f'|_U$ are contained in $V_1$, as well as all $\theta'$-gradient trajectories connecting them. Furthermore, $(f',\theta')$ coincides with $(f,\theta)$ on $X\setminus \cl{V}$ and $f'(K) \subset [-2b/3,2a/3]$. In particular
\[
\begin{array}{ccc} 
\{f=a\} \cap U = \{f'=a\} \cap U & \text{and} & \{f=-b\} \cap U = \{f'=-b\} \cap U.
\end{array}
\]
It follows that $(f',\theta')$ is Morse-Smale on $U$ and 
$$
(\CM(f',\theta',U),\partial^{\rm Morse}) = (\CM(f_1,\theta_1,U),\partial^{\rm Morse}). 
$$

For $p$ and $q$ critical points of $f'|_U$ consider the sets $$ W_p = W^u(p) \cap \{f\geq -b\}, \ \ \ \ \ M(p,q)=W^u(p) \cap W^s(q). $$ Unstable and stable manifolds here are taken with respect to the negative gradient flow $\phi'_t$ of $(f',\theta')$ and the open set $U$ (Definition~\ref{def_stable_unstable_mfds}). Since $a+b<\epsilon$, $p\in\cl{V_0}$, $f'(p) \in [-2b/3,2a/3]$ for all $p\in\crit(f'|_U) = \crit(f') \cap U$, it follows from~\eqref{time_delays} that $W_p\cap\{f>-b\}\subset U$ is a smoothly embedded copy of a open $\ind(p)$-cell. By the Morse-Smale condition the $M(p,q)$ are smooth submanifolds of $U$. Standard compactness results for Morse trajectories tell us that for all $0\leq i\leq j\leq n=\dim X$ the set
\[
M_{ij} := \bigcup \ \{ M(p,q) \mid p,q \in\crit(f'|_U), \ i\leq\ind(q)\leq\ind(p)\leq j\}
\]
is a compact subset of $\{z\in U \mid f'(z) \in(-b,a)\}$. Here we strongly used~(iii) in Definition~\ref{def_MS}. Consider also
\[
M_{-1j} := \bigcup \ \{ W_p \mid p \in \crit(f'|_U), \ \ind(p)\leq j\}
\]
for every $0\leq j\leq d$. Similarly, $M_{-1j}$ are compact subsets of $(f')^{-1}([-b,a))$. 
All $M_{ij}$ are $\Z_k$-invariant.

Each $M_{ij}$ is dynamically isolated in the sense that there exist arbitrarily small open neighborhoods $B_{ij}\subset U$ of $M_{ij}$ such that for all $z\in B_{ij}\setminus M_{ij}$, there exists some $t_z\in\R$ such that $\phi'_{t_z}(z)\not\in B_{ij}$. 
To see this in the case $0\leq i\leq j$, choose a small compact neighborhood $C\subset U$ of $M_{ij}$ such that no critical point of index $\mu\not\in[i,j]$ belongs to $C$. If $z \in C$ and $\phi'_t(z)\in C$ for all $ t\in\R$ then there are critical points $q,p\in K$ of $f'$ such that $\phi'_t(z)$ converges to $q,p$ when $t\to+\infty,t\to-\infty$ respectively. By our choice of $K$, the critical points $q,p$ have indices in $[i,j]$, from where it follows that $z \in M_{ij}$, which is a contradiction. The case $i=-1$ is handled similarly.

Now we follow the non-trivial, yet elementary, arguments of Conley~\cite{conley} with obvious modifications, keeping track of the $\Z_k$-symmetry. Define 
\begin{align*}
& N_{-1} = \{f\leq -b\} \cap U.
\end{align*}
We claim that there exist subsets $N_0,\dots,N_n$ of $U$ satisfying (a)-(e) below:
\begin{itemize}
\item[(a)] $N_{-1} \subset N_0 \subset \dots \subset N_{n-1} \subset N_n$, $N_n\setminus N_{-1} \subset U\setminus \cl{\O}$, each $N_i$ is closed in $U$.
\item[(b)] Each $N_i$ is $\Z_k$-invariant. For every $y$ in the boundary $E_i$ of $N_i$ in $U$, there exists $\delta>0$ such that $\phi'_t(y) \in \intr(N_i)$ for all $ t\in (0,\delta]$ and $\phi'_t(y) \not\in N_i$ for all $t\in [-\delta,0)$.
\item[(c)] All $N_i$ are positive invariant under $\phi'$. Moreover, for all $0\leq i\leq j\leq n$ we have $M_{ij} \subset \intr(N_j\setminus N_{i-1})$, for all $ z\in N_j \setminus (N_{i-1}\cup M_{ij})$ there is some $t\in\R$ such that $\phi'_t(z) \in N_{i-1}$ or $\phi'_t(z) \not\in N_j$, and for all $ z\in N_j\setminus N_{i-1}$ there exists $ t>0$ such that $\phi'_{[0,t]}(z) \subset N_j$.
\item[(d)] $N_i$ is homotopy equivalent to $N_{i-1}\cup (\cup_pW^u(p))$ where the union is taken over all critical points $p$ of $f'|_U$ with $\ind(p)=i$.
\end{itemize}
Before stating (e) we explore some consequences of these first four conditions. First of all, condition (c) tells us that $(N_j,N_{i-1})$ is some kind of index pair for $M_{ij}$; more details about index pairs can be found in~\cite{salamon_index_theory,salamon}. Again (c) implies that $N_i$ contains no critical points $p$ of $\ind(p)>i$.

Consider groups
\[
\mathscr{C}_j = H_j(N_j,N_{j-1}).
\]
Given a critical point $p$ with $\ind(p)=i$, trajectories in $W^u(p)\setminus(\{p\}\cup N_{i-1})$ will hit $N_{i-1}$ before they can hit $\cl{\O}$. This follows from $p\in V_0$, $f'(p)\in[a,b]$, $a+b<\epsilon$ and~\eqref{time_delays}. Hence, using the topological transversality of the flow at $E_i$ described in (b), $\cl{W^u(p)\setminus N_{i-1}}$ is a compact topological disk of dimension $\ind(p)$ in $W^u(p)$, denoted by $D_p$, and $D_p \cup N_{i-1} = W^u(p) \cup N_{i-1}$. Orient $D_p$ from the orientation of $W^u(p)$ that one chooses in the definition of the differential of the Morse complex $(\CM(f',\theta',U),\partial^{\rm Morse})$. Then $D_p$ induces a homology class $\mathscr{D}_p\in \mathscr{C}_i$.

By conditions (a)-(d), the set $\{\mathscr{D}_p \mid \ind(p)=i\}$ is a basis for~$\mathscr{C}_i$. Note also that $H_s(N_j,N_{i-1}) = 0$ if $s\not\in[i,j]$. This follows from an induction argument using exact sequences of appropriate triples. $\mathscr{C}_*$ is a chain complex with differential
\[
\Delta : \mathscr{C}_j \to \mathscr{C}_{j-1}
\]
given by the connecting homomorphism of the long exact sequence of the triple $(N_j,N_{j-1},N_{j-2})$. The obvious $\Z_k$-action on $\mathscr{C}_*$ is an action by chain maps. There is an isomorphism of groups 
\begin{equation*}
T : (\CM_*(f',\theta',U),\partial^{\rm Morse}) \to  (\mathscr{C}_*,\Delta)
\end{equation*}
defined on generators by $T:p\mapsto\mathscr{D}_p$. This is $\Z_k$-equivariant by construction. We can now state property (e) which reads
\begin{itemize}
\item[(e)] $T$ is a $\Z_k$-equivariant chain map.
\end{itemize}

We now construct the sets $N_i$ for $i\geq 0$ as in \cite{conley} keeping track of the group action, see also~\cite[Theorem 3.1]{salamon}. The construction of $\{N_i\}$ is done inductively. For each $j$ such that the $N_0,\dots,N_j$ have been constructed satisfying (a)-(d) up to index $j$, we shall need to consider the subcomplex $\CM^{\leq j}$ generated by critical points of $\ind\leq j$ and their connecting trajectories, and the subcomplex $\mathscr{C}^{\leq j}$ defined as above using indices up to $j$. Then there is  a $\Z_k$-equivariant homomorphism $T^{\leq j}$ between these complexes as before, and we may consider the property
\begin{itemize}
\item[($e_j$)] $T^{\leq j}$ is a $\Z_k$-equivariant chain map.
\end{itemize}

Let us start the induction argument. For each $p\in\crit(f'|_U)$ with $\ind(p)=0$, denote by $B_p$ the connected component  of the set $\{f'\leq f'(p)+r\}$ containing $p$. If $r>0$ is sufficiently small then the $B_p$ are disjoint, do not intersect $N_{-1}$ and are diffeomorphic to $d$-dimensional Euclidean balls. Define $N_0= N_{-1} \cup (\cup_pB_p)$. Clearly (a)-(d) hold up to index $0$, and ($e_0$) also holds trivially.

Assume that $N_0,\dots,N_j$ have been constructed such that (a)-(d) hold up to index $j$ and also that ($e_j$) holds. For each $p\in\crit(f'|_U)$ with $\ind(p)=j+1$, we can choose a small Conley pair $(Q_p,Q_p^-)$ for the isolated invariant set $\{p\}$ given by a compact thickening of the pair consisting of a small smooth compact ($j+1$)-dimensional disk in $W^u(p)$ around $p$ and its boundary. The crucial point here is the obvious fact that we can choose $(Q_p,Q_p^-)$ to be invariant under the isotropy group of $p$. This property will be used as follows. Let $a$ be the diffeomorphism inducing the action of $1\in \Z_k$. We can select these small pairs in such a way that
\begin{equation}\label{equivariance_of_pairs}
(Q_{a^s(p)},Q^-_{a^s(p)}) = (a^s(Q_p),a^s(Q_p^-))
\end{equation}
for all $s\in\Z_k$.

To this end, we first choose a base point $p_*$ in each $\Z_k$-orbit of critical points of $f'|_U$ with index $j+1$, choose a small pair for $p_*$ as explained above which is invariant under the isotropy group of $p_*$, and move this pair by the group action. Now we use the forward flow to expand the pairs $(Q_{p},Q^-_{p})$ and obtain longer pairs $(\hat Q_{p},\hat Q^-_{p})$ satisfying $\hat Q_p \cap \cl{\O}=\emptyset$ and $\hat Q^-_{p} \subset \intr(N_j)$. Define $N_{j+1} = N_j \cup(\cup_{p}\hat Q_{p})$. This is a closed (in $U$) $\Z_k$-invariant neighborhood of $N_j \cup(\cup_pW^u(p))$. Moreover, using the transversality in (b) for $N_j$, one shows that $N_j \cup(\cup_pW^u(p))$ is a deformation retract of $N_{j+1}$, that $N_{j+1}$ is an attractor, and that the transversality in (b) holds for $N_{j+1}$. Hence $N_0,\dots,N_{j+1}$ satisfy (a)-(d) up to index $j+1$. By the arguments from~\cite{conley} without symmetry,  the map $T^{\leq j+1}$ is a chain map. Since it is $\Z_k$-equivariant by construction, we get property $(e_{j+1})$. This completes the induction step.

So far we know that $(\mathscr{C}_*,\Delta)$ computes $\HM(f,x)$, and that the subcomplex of $(\mathscr{C}_*,\Delta)$ consisting of $\Z_k$-invariant chains computes the homology of the subcomplex of $(\CM(f',\theta',U),\partial^{\rm Morse})$ consisting of $\Z_k$-invariant chains.

It follows from a $\Z_k$-symmetric version of the arguments in~\cite[appendix A]{milnor_book_ch} that the subcomplex of $(\mathscr{C}_*,\Delta)$ consisting of $\Z_k$-invariant chains computes the homology of the subcomplex of $C_*(N_n,N_{-1})$ consisting of $\Z_k$-invariant chains.

By~\eqref{first_aux_estimate-f_0} and~\eqref{second_aux_estimate-f_0}, if $(f_0,\theta_0) \in \mathcal N_0$ then 
\begin{equation}\label{aux_estimate_h}
\begin{aligned}
d(f_0 - (a+b)\beta_0) \cdot \nabla^{\theta_0}f_0 &= df_0 \cdot \nabla^{\theta_0} f_0 - (a+b) d\beta_0 \cdot \nabla^{\theta_0} f_0 \\
&\geq |df_0|_{\theta_0}^2 -(a+b) |d\beta_0|_{\theta_0}|\nabla^{\theta_0} f_0|_{\theta_0} \\
&\geq \frac{1}{2} |df_0|_{\theta_0}^2 > 0 \ \ \text{ on } \ \ \cl{U} \cap \cl{\O}.
\end{aligned}
\end{equation}
In other words, $f_0$ is a Lyapunov function for the negative $\theta_0$-gradient flow of $f_0 - (a+b)\beta_0$ whenever $(f_0,\theta_0) \in \mathcal N_0$. Consider $h = f' - (a+b)\beta_0$ and the pair $(\tilde N_n,\tilde N_{-1})$ defined by  
\[
\begin{array}{ccc} \tilde N_{-1} = \{h\leq -b\} \cap U & & \tilde N_n = \tilde N_{-1} \cup (N_n\setminus N_{-1}) .\end{array}
\]
Note that $W_- = \tilde N_{-1}$. By the definition of $\beta_0$,~\eqref{aux_estimate_h} and~(a), the number $-b$ is a regular value of $h$ on $U$, and $$ \tilde N_n \setminus \cl{\O} = N_n \setminus \cl{\O}, \qquad \tilde N_{-1} \setminus \cl{\O} = N_{-1} \setminus \cl{\O}. $$ Simple excision arguments show that the inclusion $(N_n,N_{-1}) \to (\tilde N_n,\tilde N_{-1})$ induces an isomorphism on relative singular homology. Now the transversality condition in (b) and the fact that the functions $h$ and $f'$ coincide on $U\setminus \cl{\O}$ imply that we can use the negative $\theta'$-gradient of $f'$ to construct a deformation retraction of $W = \{f\leq a\}\cap U$ onto $\tilde N_n$ which is stationary on $W_- = \tilde N_{-1}$. By $\Z_k$-equivariance of this flow, we find the last arrow in a sequence of $\Z_k$-equivariant isomorphisms
\begin{equation*}
\HM_*(f',\theta',U) \to H_*(\mathscr{C},\Delta) \to H_*(N_n,N_{-1}) \to H_*(\tilde N_n,\tilde N_{-1}) \to H_*(W,W_-)
\end{equation*}
as desired.

\section{Comparing with the Borel construction}\label{app_Borel}

Our goal here is to explain how the Borel construction in Morse homology, first implemented by Viterbo~\cite{viterbo}, can be used to define $\Z_k$-equivariant local Morse homology. These groups turn out to be isomorphic to the $\Z_k$-invariant local Morse homology groups defined in Section~\ref{sec_properties}. 

Let $(M,\theta)$ be a Riemannian manifold without boundary, equipped with an action of $\Z_k$ by isometries. Let $f:M\to\R$ be smooth and let $p\in M$ be an isolated critical point of $f$ which is also a fixed point of the $\Z_k$-action. For the arguments to be given below there is no loss of generality to assume that the anti-gradient vector field of $(f,\theta)$ is complete.

The group $\Z_k$ acts on $S^{2N+1}\subset\C^{N+1}$ as $$ m\cdot (z_0,\dots,z_N)=(e^{i2\pi\frac{m}{k}}z_0,\dots,e^{i2\pi\frac{m}{k}}z_N), \qquad m\in\Z_k. $$ This action is free, the orbit space is a lens space. Let $\theta_N$ denote the Euclidean metric on $\C^{N+1}$ pulled back to $S^{2N+1}$ by the inclusion map.

One can find a sequence of smooth functions $h_N:S^{2N+1}\to\R$ with the following properties:
\begin{itemize}
\item[(i)] Each pair $(h_N,\theta_N)$ is $\Z_k$-invariant and Morse-Smale.
\item[(ii)] $S^{2N-1}$ is a normally hyperbolic invariant manifold for the anti-gradient flow of $(h_N,\theta_N)$ such that the $\theta_N$-Hessian of $h_N$ has only positive eigenvalues in directions transverse to $S^{2N-1}$.
\item[(iii)] Morse indices of critical points of $h_N$ in $S^{2N+1} \setminus S^{2N-1}$ are equal to $2N$ or to $2N+1$.
\end{itemize}

The pair $(h_N,\theta_N)$ descends to a Morse-Smale pair $(\bar h_N,\bar\theta_N)$ on $S^{2N+1}/\Z_k$. It follows from (ii) and (iii) that inclusions of critical points induce chain maps
\begin{equation*}
(\CM(\bar h_N),\partial^{(\bar h_N,\bar\theta_N)}) \to (\CM(\bar h_{N+1}),\partial^{(\bar h_{N+1},\bar\theta_{N+1})})
\end{equation*}
on the associated Morse complexes. We use $\Q$-coefficients throughout this discussion. The associated maps on homology fit into a directed system, with respect to which we can take a limit 
\begin{equation*}
\lim_{N\to\infty} \HM(\bar h_N,\bar\theta_N).
\end{equation*}
One checks that this limit is isomorphic to the singular homology of $B\Z_k$.

The diagonal $\Z_k$-action on $M\times S^{2N+1}$ is free. The pair $(f+h_N,g\oplus \theta_N)$ on the product $M\times S^{2N+1}$ is $\Z_k$-invariant. It descends to a pair $(\overline{f+h_N},\overline{g\oplus \theta_N})$ on $(M\times S^{2N+1})/\Z_k$. The submanifold $(\{p\}\times S^{2N+1})/\Z_k$ is an isolated invariant set\footnote{Given a flow $\phi^t$, a compact invariant set $K$ is called isolated if it admits an open neighborhood $U$ such that $K = \{ p\in U \mid \phi^t(p) \in U \ \forall t\in\R\}$.} for the anti-gradient flow of $(\overline{f+h_N},\overline{g\oplus \theta_N})$. The homology of the associated Conley index is just the ``local'' Morse homology associated to the data consisting of $\overline{f+h_N}$, $\overline{g\oplus \theta_N}$ and $(\{p\}\times S^{2N+1})/\Z_k$, which we denote by
\begin{equation}\label{local_homology_ndbtimessphere_quotient}
\HM(\overline{f+h_N},\overline{g\oplus \theta_N},(\{p\}\times S^{2N+1})/\Z_k).
\end{equation}
Here we made use of a construction which is simple but maybe not too well-known, let us explain. A small (even in $C^\infty$) perturbation of the pair will make all critical points near $(\{p\}\times S^{2N+1})/\Z_k$ non-degenerate, and all connections between them transversely cut-out. Compactness for spaces of connecting trajectories between these critical points comes from the fact that $p$ is an isolated critical point of~$f$. The homology of the associated Morse complex is independent of the small perturbation, since we can also achieve transversality and compactness for local continuation maps. These provide canonical isomorphisms at the level of homology. In Section~\ref{sec_properties} we explained this construction in more detail for an isolated critical point. 

The properties of $(h_N,\theta_N)$ allow to fit the homology groups~\eqref{local_homology_ndbtimessphere_quotient} into a directed system, with maps induced by inclusions $(\{p\}\times S^{2N-1})/\Z_k \subset (\{p\}\times S^{2N+1})/\Z_k$. We finally define the $\Z_k$-equivariant local homology of $f$ at $p$ by
\begin{equation}\label{defn_local_equiv_Morse_homology}
\HM^{\Z_k}(f,p) = \lim_{N\to\infty} \HM(\overline{f+h_N},\overline{g\oplus \theta_N},(\{p\}\times S^{2N+1})/\Z_k).
\end{equation}
As the notation suggests, this turns out to be independent of the metric $g$ and of the data $\{(h_N,\theta_N)\}_{N\geq1}$. One could prove at this point, without referring to any isomorphism with our invariants, continuation properties identical to those stated in Proposition~\ref{prop_invariance_with_symmetries}.

Now we move on to describe an isomorphism with our $\Z_k$-invariant Morse homology groups. More precisely, we would like to build a canonical isomorphism 
\begin{equation}\label{iso_equiv_inv}
\HM^{\Z_k}(f,p) \simeq \HM(f,p)^{\Z_k} 
\end{equation}
in the sense that it commutes with continuation maps. We will merely provide a description, technical details of proofs will be omitted. We follow closely the case of closed manifolds explained in~\cite[appendix]{GHHM}.

Let $U$ be an isolating neighborhood for $(f,p)$. Let $(f',g')$ be a $C^2$-small perturbation of $(f,g)$ which is $\Z_k$-invariant, and Morse-Smale on $U$ in the sense of Definition~\ref{def_MS}. The local Morse homology of $(f,p)$ was defined as the homology of the chain complex
\[
(\CM(f',g',U),\partial^{(f',g',U)})
\]
generated by the critical points of $f'$ in $U$. The differential counts rigid anti-gradient trajectories connecting them. This complex inherits a $\Z_k$-action by chain maps, and $\HM(f,p)^{\Z_k}$ is defined as the homology of the subcomplex of $\Z_k$-invariant chains.

Similarly one considers a chain complex
\begin{equation}\label{local_Morse_ndbtimessphere}
(\CM(f'+h_n,g'\oplus \theta_N,U\times S^{2N+1}),\partial^{(f'+h_N,g'\oplus\theta_N,U\times S^{2N+1})})
\end{equation}
generated by critical points in $U\times S^{2N+1}$, with a differential that counts rigid anti-gradient trajectories of $(f'+h_n,g'\oplus \theta_N)$ between them. We need two important facts which can be checked by following definitions. \\

\noindent {\it Fact 1.} There is a natural identification of chain complexes
\begin{equation}\label{identification_tensor}
\begin{aligned}
& (\CM(f'+h_n,g'\oplus \theta_N,U\times S^{2N+1}),\partial^{(f'+h_N,g'\oplus\theta_N,U\times S^{2N+1})}) \\
&\simeq (\CM(f',g',U),\partial^{(f',g',U)}) \otimes (\CM(h_N,\theta_N),\partial^{(h_N,\theta_N)})
\end{aligned}
\end{equation}
This identification intertwines the $\Z_k$-action by chain maps on~\eqref{local_Morse_ndbtimessphere} with the diagonal $\Z_k$-action by chain maps on the tensor product. \\

\noindent {\it Fact 2.} The homology groups~\eqref{local_homology_ndbtimessphere_quotient} coincide with the homology of the subcomplex of $\Z_k$-invariant chains of the complex~\eqref{local_Morse_ndbtimessphere}. \\

We can write a direct sum of $\Z_k$-invariant subcomplexes 
\begin{equation}
\CM(f',g',U) = \CM(f',g',U)^{\Z_k} \oplus \bigoplus_\sigma V_\sigma
\end{equation}
where $\CM(f',g',U)^{\Z_k}$ denotes the set of $\Z_k$-invariant chains, i.e. the isotypical component associated to the trivial action, and the $V_\sigma$ denote the subcomplexes associated to the other isotypical components. Similarly we can write
\begin{equation}
\CM(h_N,\theta_N) = \CM(h_N,\theta_N)^{\Z_k} \oplus \bigoplus_\eta W_\eta^N.
\end{equation}
It follows that the subcomplex of $\Z_k$-invariant chains of the tensor product in~\eqref{identification_tensor} is
\begin{equation}\label{decomp_isotypical_comps}
\begin{aligned}
& \left( (\CM(f',g',U),\partial^{(f',g',U)}) \otimes (\CM(h_N,\theta_N),\partial^{(h_N,\theta_N)}) \right)^{\Z_k} \\
&= \left( \CM(f',g',U)^{\Z_k} \otimes \CM(h_N,\theta_N)^{\Z_k} \right) \ \oplus \ \bigoplus_{\sigma,\eta} \left( V_\sigma\otimes W_\eta^N \right)^{\Z_k}
\end{aligned}
\end{equation}

The complexes $W_\eta^N$ have trivial homology. This is so because $\HM(h_N,\theta_N)$ (no symmetry) turns out to be the homology of the subcomplex $\CM(h_N,\theta_N)^{\Z_k}$: one generator in degree zero and another in degree $2N+1$, both represented by $\Z_k$-invariant cycles. It follows that all $W^N_\eta$ are acyclic. Hence so are all $V_\sigma\otimes W_\eta^N$.

Putting these remarks together with {\it Fact 1} and {\it Fact 2}, we can pass to homology in~\eqref{decomp_isotypical_comps} and obtain 
\begin{equation}\label{prefinal_identification}
\begin{aligned}
& \HM(\overline{f+h_N},\overline{g\oplus \theta_N},(\{p\}\times S^{2N+1})/\Z_k) \simeq \HM(f,p)^{\Z_k} \otimes H(S^{2N+1}/\Z_k;\Q)
\end{aligned}
\end{equation}
Finally it is possible to check that, in the above isomorphism, the chain maps on Morse homologies~\eqref{local_homology_ndbtimessphere_quotient} induced by the inclusions $S^{2N-1} \subset S^{2N+1}$ commute with corresponding chain maps on singular homology of lens spaces on the right-hand side of~\eqref{prefinal_identification}. Taking limits on both sides, and noting that only one generator in degree zero survives in $\lim_{N\to\infty}H(S^{2N+1}/\Z_k;\Q)$, we get the isomorphism~\eqref{iso_equiv_inv}. 

\begin{remark}
We hope that the recipes described here, which are needed to implement Viterbo's construction at a local level, will convince the reader of two facts. Firstly, it is quite hard to work with the $\Z_k$-symmetry at the chain level using definition~\eqref{defn_local_equiv_Morse_homology} of local $\Z_k$-equivariant homology. Secondly, when applying definition~\eqref{defn_local_equiv_Morse_homology} to discrete action functionals, it will be quite hard to prove iteration properties. Not to mention that the definition of $\HM(f,p)^{\Z_k}$ is a lot simpler and more geometrically transparent than that of $\HM^{\Z_k}(f,p)$.
\end{remark}

\end{document}